\documentclass{article}
\usepackage{graphicx}
\usepackage{arxiv}

\usepackage[utf8]{inputenc} % allow utf-8 input
\usepackage[T1]{fontenc}    % use 8-bit T1 fonts
\usepackage{url}            % simple URL typesetting
\usepackage{booktabs}       % professional-quality tables
\usepackage{amsfonts}       % blackboard math symbols
\usepackage{nicefrac}       % compact symbols for 1/2, etc.
\usepackage[dvipsnames]{xcolor}
\usepackage{physics}
\usepackage{graphicx} % Required for inserting images
\graphicspath{ {./images} }
\usepackage{amsmath}
\usepackage{amssymb}
\usepackage{amsthm}
\usepackage{mathtools}
\usepackage{float}
\usepackage{todonotes}
\usepackage{array}
\usepackage{enumerate}
\usepackage{multicol}
\usepackage{multirow}
\usepackage{mathrsfs}
\usepackage{siunitx}
\usepackage{tabularx}
\usepackage{subcaption}
\usepackage[bookmarks, hidelinks]{hyperref}

\usepackage{bbm}
\usepackage{authblk}

\usepackage{subcaption}
\usepackage{hyperref}
\usepackage{booktabs}
\usepackage{cleveref}

\usepackage{adjustbox}
\usepackage{amsopn}

\theoremstyle{definition}

\newtheorem{defn}{Definition}

\newtheorem{theorem}{Theorem}[section]

\newtheorem{corollary}[theorem]{Corollary}
\newtheorem{lemma}[theorem]{Lemma}

\newtheorem{remark}{Remark}

\newcommand{\R}{\mathbb{R}}

\newcommand{\inp}[2]{\left\langle #1,\, #2 \right\rangle}

\renewcommand{\vec}[1]{\mathbf{#1}}

\DeclareMathOperator{\diag}{diag}

\newcommand{\lJump}{[\![}
\newcommand{\rJump}{]\!]}

\newcommand{\nocrash}[1]{\textcolor{Blue}{#1}}
\newcommand{\crash}[1]{\textcolor{Red}{#1}}

\newcommand{\tablescaling}{0.9}
\newcommand{\figurescaling}{0.99}

\renewcommand{\vec}[1]{\mathbf{#1}}

\newcommand{\fn}[2]{\mathinner{#1\mathopen{\left(#2\right)}}}

\title{A dual-pairing summation-by-parts finite difference framework for nonlinear conservation laws}

\author[1]{Dougal Stewart}
\author[2]{Nathan Lee}
\author[3,4]{Kenneth Duru}

\affil[1]{University of Melbourne, Australia}
\affil[2]{University of New South Wales, Australia}
\affil[3]{Department of Mathematical Sciences , University of Texas at El Paso, USA}
\affil[4]{Mathematical Sciences Institute, Australian National University, Canberra, Australia}

\date{}

\begin{document}

\maketitle

 \begin{abstract}
    Robust and convergent high-order numerical methods for solving partial differential equations are highly attractive due to their efficiency on modern and next-generation hardware architectures. However, designing such methods for nonlinear hyperbolic conservation laws remains a significant challenge. 
In this work, we introduce a framework based on dual-pairing (DP) and upwind summation-by-parts (SBP) finite difference (FD) and discontinuous Galerkin (DG) finite element methods, aimed at achieving accurate and robust numerical approximations of nonlinear conservation laws. The framework ensures entropy consistency and features an intrinsic high-order accurate "filter" designed to detect and resolve regions where the solution is poorly captured or discontinuities are present.
The DP SBP FD/DG operators form a dual pair of discrete derivative operators that collectively preserve the SBP property. Furthermore, these operators are constructed to be upwind, allowing them to incorporate dissipation within the elements themselves. This contrasts with traditional SBP and collocated DG spectral element methods, which typically induce dissipation solely through numerical fluxes at element interfaces.
Our framework facilitates the systematic combination of DP SBP FD/DG operators with skew-symmetric and upwind flux splitting techniques. This integration enables the development of robust, high-order accurate schemes for nonlinear hyperbolic conservation laws. The resulting semi-discrete formulation is provably entropy-stable for arbitrary nonlinear problems.
We illustrate the effectiveness of our approach with specific examples, including inviscid Burgers’ equation, nonlinear shallow water equations, and the compressible Euler equations of gas dynamics. Extensive numerical experiments are presented to verify the accuracy and demonstrate the robustness of our framework.
 \end{abstract}

%\tableofcontents

\section{Introduction}
Accurate and efficient numerical simulations of nonlinear conservation laws are essential for advancements in science and industry. They facilitate the development of new technologies, deepen scientific understanding, and enable new discoveries. The endeavor to derive fast, precise, and reliable numerical methods for approximating solutions to nonlinear conservation laws has been a central focus of research since the pioneering works of von Neumann, Richtmyer, Lax, Wendroff, Godunov, and Harten~\cite{VonNeumannandRichtmyer1950,PDLax1954,Lax-Wendroff-1960,godunov1959,HARTEN1983151}. This area continues to attract significant interest from mathematicians and computational scientists, as evidenced by recent studies such as ~\cite{Sjögreen2019,GERRITSEN1996245,FISHER2013353,STIERNSTROM2021110100,Nordström2006,Mattsson_etal2004,MattssonandRider2015,Carpenter_et_al2014,fluxo_SplitForm,gassner2016skewsym_swe,CHAN2018346,GassnerWinters2021,Ranocha2018,giraldo2002nodal,ChanRanochaHendrik_et_al2022,CHEN2017427,LiuLuShu2024}.

Historically, the numerical solution of nonlinear conservation laws was dominated by low-order (first- and second-order) finite volume and finite difference schemes~\cite{godunov1959,LeVeque1992conslaws,LeVeque_2002,Lax-Wendroff-1960}, which rely on Godunov-type approximate Riemann solvers. While low-order methods are robust, their excessive numerical dissipation and dispersion errors limit their capacity to accurately resolve complex fine scale features such as shocks, turbulence, vortices, and wave phenomena—including gravitational and Rossby waves. Conversely, high-order methods—pioneered by Kreiss and Oliger~\cite{Kreiss-1972}—exhibit substantially reduced dissipation and dispersion errors, rendering them highly effective for wave-dominated problems. However, their low dissipation can compromise robustness, making high-order schemes more susceptible to numerical noise and instabilities arising from unresolved nonlinear modes. As a result, not all high-order methods are suitable for solving nonlinear conservation laws.

Many prior studies have promoted high-order methods with provable linear stability, relying on the smoothness of solutions to simulate nonlinear problems~\cite{H-O-Kreiss1970,kreiss2004initial,Gustafsson-Sundstrom1978,Oliger-Sundstrom1978}. However, linear stability alone often lacks robustness for nonlinear problems and fails to produce reliable results in many scenarios. For instance, phenomena such as turbulence and shocks are highly nonlinear, and linearly stable methods are insufficient to capture their complexities accurately. Numerous numerical results in the literature~\cite{Tadmor2003entropystable,gassner2022stability,FISHER2013353,ChanRanochaHendrik_et_al2022} have demonstrated that nonlinear stability—particularly entropy stability—is essential for ensuring the robustness and high-order accuracy of numerical methods applied to nonlinear conservation laws.

The primary goal of this study is to develop a framework for constructing provably high-order accurate, entropy-stable numerical schemes for nonlinear conservation laws, utilizing dual-pairing (DP) summation-by-parts (SBP) operators for both finite difference (FD) and discontinuous Galerkin (DG) formulations~\cite{DovgilovichSofronov2015,MATTSSON2017upwindsbp,williams2024drp,GLAUBITZ2025113841}. To our knowledge, such an approach has not been reported previously in the literature. Notably, recent applications of DP-SBP operators to nonlinear conservation laws~\cite{ranocha2023highorder,LUNDGREN2020109784,STIERNSTROM2021110100,GLAUBITZ2025113841} have combined these operators with classical finite volume flux splitting~\cite{LeVeque1992conslaws,godunov1959,LeVeque_2002} to ensure provable linear stability. For example, \cite{STIERNSTROM2021110100} focuses on scalar conservation laws, while \cite{ranocha2023highorder,GLAUBITZ2025113841} extend the methodology to systems and DG methods, emphasizing robustness. As demonstrated later in this paper, our framework facilitates the design of robust, high-order methods based on DP FD/DG SBP operators for both scalar and systems of nonlinear conservation laws, including applications to turbulent flows with shocks.

%%%
% 

% 
A nonlinearly/entropy stable numerical method bounds or conserves a convex entropy functional, which is conserved by the PDE system for smooth solutions~\cite{Tadmor2003entropystable,chan2022entropy,nordstrom2022linear,Sjögreen2019,LeVeque_2002,ricardo2023conservation,hew2024stronglystabledualpairingsummation,gassner2022stability,DAFERMOS1973202,HARTEN1983151,CHEN2017427,LiuLuShu2024}. Entropy stability also has physical significance: for example, solutions of the compressible Euler equations should satisfy the second law of thermodynamics, meaning total entropy should not decrease in a closed system. This leads to an entropy inequality—a partial differential inequality that a physically correct weak solution must satisfy for all convex entropy functions~\cite{Tadmor2003entropystable,LeVeque1992conslaws,LeVeque_2002,PDLax1973}.
However, high-order numerical schemes cannot satisfy the entropy inequality  for all convex entropy functions. Nevertheless, we can identify a specific entropy functional on which a numerical scheme can be based~ \cite{Carpenter_et_al2014,fluxo_SplitForm,gassner2016skewsym_swe,CHAN2018346,GassnerWinters2021,Ranocha2018,giraldo2002nodal,ChenShu2020}.

The SBP principle~\cite{kreiss1974finite,BStrand1994,fernandez2014review,svard2014review,gassner2013skew,taylor2010compatible} is crucial for developing provably stable numerical methods for PDEs, thanks to its mimetic structure that ensures stability at the semi-discrete level when boundary conditions are carefully treated. At the continuous level, stability proofs rely on integration by parts, and sometimes on the chain and product rules. Although SBP operators mimic integration by parts discretely, they struggle to replicate the chain and product rules. To align discrete analysis with continuous theory, nonlinear conservation laws are often reformulated into skew-symmetric or entropy-conserving split-forms, avoiding the direct use of the chain and product rules. Several methods exist for this reformulation~\cite{Sjögreen2019,GERRITSEN1996245,FISHER2013353,STIERNSTROM2021110100,ricardo2024entropy,ricardo2023entropy,ChanRanochaHendrik_et_al2022,gassner2016skewsym_swe,fluxo_SplitForm,Nordstrom_2022_skewsym_comp}. However, identifying suitable entropy functions and skew-symmetric forms can be complex for many PDE systems.

Once an entropy functional and a compatible skew-symmetric reformulation are established, SBP derivative operators can produce entropy-conserving schemes. Although these schemes are theoretically nonlinearly stable, traditional SBP operators—based on central FD~\cite{kreiss1974finite,BStrand1994,fernandez2014review,svard2014review}—or collocated DG spectral element method (DGSEM)~\cite{gassner2013skew,taylor2010compatible,fluxo_SplitForm}—may admit spurious wave modes that compromise accuracy or cause simulations to crash. Recent work~\cite{gassner2022stability} shows that local linear energy stability is essential for convergence of numerical methods for nonlinear conservation laws. Unfortunately, current high-order entropy-stable SBP and DGSEM schemes in skew-symmetric split-form lack this local energy stability, failing to meet this critical criterion.

This paper introduces a multi-block DP upwind SBP FD/DG framework for accurate, robust, and efficient numerical solutions of nonlinear conservation laws. Building on recent advances \cite{DovgilovichSofronov2015,MATTSSON2017upwindsbp,williams2024drp}, the DP SBP framework was developed to enhance FD methods for waves and nonlinear problems in complex geometries. It was later extended to DG operators \cite{GLAUBITZ2025113841}. The dual pair of high-order DP-SBP discrete derivative operators preserves the SBP principle but requires careful combination—along with skew-symmetric and upwind flux splitting—to ensure entropy stability and preserve system's symmetry.
Our key contribution is the systematic integration of DP SBP operators with flux splitting, enabling the design of stable, high-order schemes for nonlinear hyperbolic conservation laws. The semi-discrete scheme is conservative, entropy consistent (reproduces correct entropy loss at shocks) and provably entropy-stable, a novel feature not previously reported.  Extensive numerical tests confirm the method’s accuracy, stability, entropy consistency, and conservation.

The paper is organized as follows: \cref{sec:contanalysis} reviews continuous analysis and flux splitting; \cref{sec:SBP framework} introduces the SBP FD framework, including assumptions underlying traditional and DP upwind SBP methods; \cref{sec:discrete_methods} derives semi-discrete schemes and proves their conservative and stability properties; \cref{sec:numerical-experiments} presents numerical results verifying accuracy, conservation, and robustness; \cref{sec:conclusion} summarizes findings and discusses future directions.

\section{Continuous Analysis}\label{sec:contanalysis}
In this section we present and analyze the model problems. We perform entropy stability analysis and discuss the skew-symmetric reformulations of the model problems.
\subsection{Conservation Laws}
We are concerned with the numerical approximation of solutions $\vb{u}(\vb x,t) \in \R^{m} $ to PDEs of the form
\begin{equation} \label{eq:conservation-law}
    \partial_t \vb{u}(\vb x,t) + \div \vb f(\vb{u}(\vb x, t)) = 0,
    \quad \vb x \in \Omega \subset \R^d,
    \quad t \in [0, T],
\end{equation}
where $\vb{u}: \Omega \times [0,T] \to \R^m$ with $m\ge 1$  is the unknown solution,  $ \vb f : \R^m \to \R^{m\times d}$ is the flux function of the PDE, $d\ge 1$ denotes the spatial dimension, $\vb x = (x_1, x_2, \ldots, x_d) \in  \R^{d}$ is the spatial variable and $t \in [0, T]$ denotes time and $T>0$ is the final time. We set the smooth initial condition $\vec{u}(\vec{x},0) = \vec{u}_0(\vec{x})\in C^{1}(\Omega)$.
Such PDEs are known as conservation laws and often appear in science and engineering where $\vb{u}$ is a set of conserved variables which may denote some physical quantity such as velocity, momentum, mass, pressure, density, energy, magnetic/electric field, etc. Conservation laws are aptly named because they state that the total quantity of $\vb{u}$ is conserved. This can be understood by integrating  \eqref{eq:conservation-law} and then applying the divergence theorem,
\begin{equation} \label{eq:conservation}
    \partial_t \vb{\mathcal{U}}(t) =
        -\oint_{\partial\Omega} \vb f(\vb{u}(\vb x,t)) \cdot \vb n \dd{\vb x}, \quad \vb{\mathcal{U}}(t) =\int_\Omega \vb{u}(\vb x, t).
\end{equation}
That is, a conservation law essentially states that the change in the total quantity $\vb{\mathcal{U}}(t)$ in any given region of space $\Omega$ is governed by the flux $\vb f(\vb{u})\cdot \vb n$ going in and out of the boundary $\partial\Omega$, where $\vb n$ is the outward unit normal on the boundary. In particular, for continuous periodic solutions in $\Omega$ the total quantity $\vb{\mathcal{U}}(t)$ is conserved, that is $ \partial_t \vb{\mathcal{U}}(t)=0$ for all $t\ge0$. For nonlinear hyperbolic conservation laws with weak solutions, the conservative properties of a numerical method are critical for the convergence of numerical solutions \cite{Lax-Wendroff-1960}.

% By using the chain rule and the product rule we can rewrite the nonlinear conservation \eqref{eq:conservation-law} in the quasilinear form 
% \begin{align} \label{eq:conservation-law-quasilinear}
%     \partial_t \vb{u} + \sum_{\nu=1}^{m}A_{x_\nu}\frac{\partial\vb{u}}{\partial x_\nu} = 0,
%     \quad A_{x_\nu} \in \R^{m\times m},
%     \quad t \in [0, T],
% \end{align}
% where the square matrices
% $
% A_{x_\nu}(\vb{u}) = \partial_{\vb{u}}{\vb f}_\nu
% $,
% for $\nu = 1, 2, \cdots, d$, are the Jacobian of the flux. We are particularly interested in nonlinear strongly hyperbolic conservation laws. This is equivalent to the matrix
% $
% A = \sum_{\nu=1}^{d}k_\nu A_{x_\nu}$, with $\sum_{\nu=1}^{d}k_\nu^2 = 1,
% $
% being diagonalisable with real eigenvalues~\cite{gustafsson1995time} for all admissible state $\vb{u}\in \mathbb{R}^{m}$.

In this work, we will focus on initial value problems (IVP) driven by initial conditions or internal forcing for systems of nonlinear hyperbolic conservation laws. %$m=1,2,3,4$ and $d=1,2$. 
We will consider one and two space dimensions and assume periodic boundaries such that the boundary integral in the right-hand side of~\eqref{eq:conservation} vanishes.
However, our framework can be extended to an arbitrary number of space dimensions. Well-posed non-periodic boundary conditions can be implemented, for example, using  penalty methods or numerical fluxes~\cite{carpenter1994time,Gustafsson2008,Nordström2006,hew2024stronglystabledualpairingsummation}. That said, it is important to note that the numerical treatment of boundary conditions is problem specific and comes with additional theoretical and practical difficulties, in particular for nonlinear conservation laws. Therefore, the formulation of well-posed boundary conditions for nonlinear conservation laws and their numerical treatments are beyond the scope of the present study.
%. 

\subsection{Entropy stability} \label{sec:entropy-stability}
Hyperbolic  conservation laws often develop discontinuities, like shocks or contact discontinuities, even from smooth initial data, due to a loss of derivatives. These discontinuities make classical derivatives undefined, so solutions are understood in a weak sense \cite{PDLax1973,PDLax1954,LeVeque1992conslaws}. However, multiple weak solutions can exist, requiring additional criteria for uniqueness. 
One such criterion involves physical principles. For instance, solutions to the compressible Euler equations must satisfy the second law of thermodynamics, meaning the total entropy should not decrease in a closed system. This leads to the entropy inequality—a PDE constraint ensuring physically meaningful solutions. A physically correct  weak solution must satisfy this entropy condition for all convex entropy functions \cite{Tadmor2003entropystable,LeVeque1992conslaws,LeVeque_2002,PDLax1973}.

We introduce the $L^2$ inner product and norm
%\begin{align}
    $(\vb{u}, \vb{v}) = \int_{\Omega} \vb{u}^\top\vb{v} \dd \vb x$,  $\|\vb{u}\|^2 =  (\vb{u}, \vb{u})$.
%\end{align}
We define an entropy pair $e,\vb q$, with $e: \mathbb{R}^m \to \mathbb{R}$, ${\vb q}: \mathbb{R}^m \to \mathbb{R}^d$, where $e$ is called the entropy function, $\vb{q}$ satisfying $\partial_{\vb u} \vb q(\vb u) = \vb{g}^\top \partial_{\vb u} \vb f(\vb u)$  is called the entropy flux, and $\vb{g}: \mathbb{R}^m \to \mathbb{R}^m$ with $\vb{g}=\partial_{\vb u} e(\vb u)$ are the entropy variables. In addition, we may require that the entropy $e$ is convex, that is the 
Hessian $\mathrm{H}_{\vb u}(e(\vb u)) \in \mathbb{R}^{m \times m}$ is positive semi-definite.
%for all admissible state $\vb{u}$. 

When $\vb u$ is smooth, that is $\vb{u} \in C^1\left(\Omega \times [0,T]\right)$, we multiply \eqref{eq:conservation-law} by $\vb{g}^\top$ and apply the chain rule, we have
\begin{align*}
    \vb{g}^\top \partial_t \vb u + \vb{g}^\top \nabla \cdot \vb f(\vb u) &= 0\iff
    \vb{g}^\top \partial_t \vb u + \vb{g}^\top\partial_{\vb u} \vb f(\vb u)\cdot\partial_{\vb x} \vb u = 0.
\end{align*}
%%%
Note that $\partial_{\vb u} \vb f(\vb u): \mathbb{R}^{m\times d} \to \mathbb{R}^{m\times d \times m}$ and $\partial_{\vb x} \vb u: \mathbb{R}^{m} \to \mathbb{R}^{m\times d}$ are tensors.
Using the fact $\partial_{\vb u} \vb q(\vb u) = \vb{g}^\top \partial_{\vb u} \vb f(\vb u)$ and reversing the chain rule in time and space giving 
\begin{align}\label{eq:entropy_law}
    \partial_t e + \nabla \cdot \vb q &=0.
\end{align}
%%%
Indeed, the entropy satisfies a conservation law, and by integrating \eqref{eq:entropy_law} in time, the total entropy is conserved
\begin{align}\label{eq:total_entropy}
\dv{}{t} E(t)  = -\oint_{\partial\Omega} (\vb q \cdot \vb n) \dd \vb x =0, \quad E(t) := \left({1}, e\right).
\end{align} 
%%%
% 
%
When $\vb u$ is non-smooth, that is $\vb{u} \notin C^1\left(\Omega \times [0,T]\right)$, we may impose the entropy inequality which states that the entropy solution $\vb u$ must weakly satisfy the inequality
\cite{Tadmor2003entropystable,LeVeque_2002}
\begin{align} \label{eq:entropy-inequality}
    \partial_t e + \nabla \cdot \vb{q} &\leq 0 \implies \dv{}{t} E(t) = \dv{}{t}\int_\Omega e(\vb u(\vb x, t)) \dd \vb x \le -\oint_{\partial\Omega} (\vb q \cdot \vb n) \dd \vb x =0,
\end{align}
%%%
for all entropy pairs $e,\vb{q}$. Clearly, the local entropy inequality implies that the total entropy at any future time $t \in [0,T]$ is bounded by the initial total entropy, that is $E(t) \le E(0)$.
When $\vb u$ is sufficiently smooth we have the  equality,  $E(t) = E(0)$ for $t\in [0, T]$, as shown earlier. For more elaborate discussions on entropy-stability of nonlinear conservation laws and vanishing viscosity solutions please see~\cite{LeVeque1992conslaws,Tadmor2003entropystable}. We will make the discussion more formal with the following definition.
\begin{defn}\label{def:cont_nonlinear_stability}
    Consider the nonlinear conservation law \eqref{eq:conservation-law} with periodic boundary conditions and subject to the smooth initial condition $\vb{u}(\vb{x}, 0) = \vb{u}_0(\vb{x})$. The system \eqref{eq:conservation-law} is entropy-stable if $E(t) =  E(0)$, $\forall \, \vb{u} \in C^1\left(\Omega \times [0,T]\right)$ and $E(t) \le E(0)$, $\forall \, \vb{u} \notin C^1\left(\Omega \times [0,T]\right)$, where $E(t)$ denotes the total entropy at $t \in [0,T]$. 
\end{defn}
Robust and accurate numerical methods \cite{Tadmor2003entropystable,ChanRanochaHendrik_et_al2022,nordstrom2022linear,Sjögreen2019,LeVeque_2002,ricardo2023conservation,hew2024stronglystabledualpairingsummation,gassner2022stability,DAFERMOS1973202,HARTEN1983151} for nonlinear conservation laws are often designed by emulating Definition \ref{def:cont_nonlinear_stability} at the discrete level. However,  high-order numerical schemes can not satisfy the entropy inequality $E(t) \le E(0)$ for all convex entropy functions.
%. 
Nevertheless, we can identify a specific entropy pair $e, \vb{q}$ on which a numerical scheme can be based. 

This strategy of showing that some total energy/entropy-like quantity is non-increasing, that is $\partial_t E \leq 0$, is called the energy method. Note that the key steps in the  derivation of the entropy  estimate for our model are integration by parts, the  product rule, and the chain rule. 
However, while discrete numerical derivative and integration operators can be designed to emulate integration by parts, it is challenging (or impossible) to design numerical operators to mimic the chain rule or  product rule at the discrete level for general functions. To enable the development of robust and provably entropy-stable numerical schemes, equations at the continuous level must be rewritten into the so-called \emph{skew-symmetric form} so that we can forgo the use of chain/product rule when performing the energy method at the discrete level.

\subsection{Skew-symmetric splitting}
Typically, a skew-symmetric form is found by splitting the divergence of the flux, which is really finding a convex combination of the nonlinear conservation laws in flux form and the quasi-linear form  so that the energy method can be applied without the chain rule and/or the product rule. For a scalar PDE flux $f(u)$ we have
$\partial_x f(u) = (1-\alpha)\partial_x f +  \alpha  \partial_u f \partial_x u, \, 0 \le \alpha \le 1$, which is a convex combination of the flux form and the advective form.
For example, $\alpha = 1/3$ gives entropy conservation for the Burger's equation, with $f(u) = u^2/2$.
We will try to make the discussion more general so that it is applicable to the multi-dimensional model problem \eqref{eq:conservation-law}.
Let us consider the equivalent skew-symmetric form of \eqref{eq:conservation-law}
\begin{align} \label{eq:conservation-law_skew-symmetry}
\partial_t \vb{u} + \vb{F}(\vb u, \vb{f}(\vb{u}), \partial_{\vb x}) =0, \quad \vb{F}(\vb u, \vb{f}(\vb{u}), \partial_{\vb x}) \equiv  \div \vb f(\vb{u}), \quad \vb x \in \Omega \subset \R^d, \quad t \ge 0.
\end{align}
%%%
Here $\vb{F}(\vb u, \vb{f}(\vb{u}), \partial_{\vb x})$ is a differential operator and models the functional equivalence of  the divergence of the flux $\div \vb f(\vb{u})$ obtained by taking  a convex combination of the flux form and the quasi-linear form, as above, such that
\begin{align}\label{eq:skew-sym-flux}
\left(\vb g, \vb{F}(\vb u, \vb{f}(\vb{u}), \partial_{\vb x})\right) =\left(1, \div \vb q\right)=\oint_{\partial\Omega} (\vb q \cdot \vb n) \dd \vb x,
\end{align}
%%%
holds by only applying integration-by-parts, and without the chain/product rule. Here, $\vb{g}=\partial_{\vb u} e(\vb u)$ are the entropy variables. Consequently,  we can prove  entropy conservation using  only integration-by-parts, and without the chain/product rule.  That is 
\begin{align*} 
\underbrace{\left(\vb g, \partial_t\vb{u}\right)}_{\dv{}{t} E(t) = \dv{}{t}\left(\vb 1,  e\right)}+
\underbrace{\left(\vb g, \vb{F}(\vb u, \vb{f}(\vb{u}), \partial_{\vb x})\right)}_{\left(\vb 1, \div \vb q\right)} =0 \iff \dv{}{t} E(t) =  -\oint_{\partial\Omega} (\vb q \cdot \vb n) \dd \vb x =0.
\end{align*}
Other than the primary motivation of showing entropy-stability using only integration by parts, it is also desirable that the skew-symmetric reformulation ensure structure preservation. 
For numerical analysis it is critical that the skew-symmetric form ensure that the proof of the conservative property \eqref{eq:conservation} uses only integration by parts without the chain rule. That is, by using only integration by parts, we have 
\begin{align}\label{eq:conservation_flux} 
 \partial_t \vb{\mathcal{U}_i}=\left(\vb{1}, \partial_t\mathbf{u}_i\right)=-\underbrace{\left(\vb 1, \vb{F}_i(\vb u, \vb{f}(\vb{u}), \partial_{\vb x})\right)}_{\left(\vb 1, \div \vb f_i\right)}   =
        -\oint_{\partial\Omega} \left(\vb f_i(\vb{u}) \cdot \vb n\right) \dd{\vb x}, \forall \, i = 1, 2, \cdots, m.
\end{align}
%%%
% 
As noted earlier, for nonlinear hyperbolic conservation laws with weak solutions, the conservative properties of a numerical method are critical for the convergence of numerical solutions \cite{Lax-Wendroff-1960}.

\section{The dual-pairing summation by parts  operator for \texorpdfstring{$d/dx$}{d/dx}}\label{sec:SBP framework}
In order to  construct provably entropy-stable and conservative numerical methods for nonlinear conservation laws, we will mimic the entropy-stability and conservative properties of the continuous problem.
%%%
%%%
To replicate the continuous analysis at the discrete level our discrete differential operators must satisfy  the integration by parts principle.
To be specific, for $f, g \in C^1(\Omega)$, where $\Omega=[0, L]$ is the spatial domain, we introduce the integration by parts property for the first derivative operator $d/dx$, 
\begin{align}\label{eq:ibp}
      \left( \dv{f}{x}, g \right)  + \left( f, \dv{g}{x} \right)
      = B(fg)
      := \oint_{\partial \Omega} \left(fg\right) \cdot n_x \dd{x}
      = fg|_{x = L} - fg|_{x = 0},
\end{align}
%%%
where $n_x$ is the $x$-component of the normal vector on the boundary, that is $n_x = -1$ at $x = 0$ and $n_x = 1$ at $x = L$.
 Discrete approximations of the first derivative operator $d/dx$ that satisfy the integration by parts property \eqref{eq:ibp} are called SBP operators \cite{BStrand1994,gustafsson1995time,kreiss1974finite,williams2024drp,MATTSSON2017upwindsbp,gassner2013skew,taylor2010compatible,fluxo_SplitForm}.  We will give the formal definitions below. We will discuss the recent DP SBP framework of \cite{williams2024drp,MATTSSON2017upwindsbp}, formulated through the assumptions given in \cite{williams2024drp}. The traditional SBP operators \cite{BStrand1994,gustafsson1995time,kreiss1974finite,gassner2013skew,taylor2010compatible,fluxo_SplitForm} are a restriction of the DP SBP  framework to anti-symmetric discrete operators. 

For $n \in \mathbb{N}$ we define the uniformly spaced  grid points given by
    $x_{j} =(j-1) \Delta{x}$,  $\Delta{x} = {L}/{(n-1)}, \quad L>0$,  $j \in  \{1, 2, 3, \ldots, n\}$.
Given $f \in L^2([0, L])$ we  denote the grid function by 
    $\mathbf{f} :=  \left(f(x_1), \ldots, f(x_n)\right)^T \in \mathbb{R}^n$. 
We introduce the positive definite diagonal matrix $H := \mathrm{diag}\left( \left(h_1, \dots, h_n\right) \right) \in \mathbb{R}^{n \times n}$, $h_j >0$, $j \in \{ 1, \ldots, n \}$, and  the discrete inner-product
${\inp{\mathbf{f}}{\mathbf{g}}_H}
    := \mathbf{f}^T H \mathbf{g} =\sum_{j = 1}^n h_j f(x_j) g(x_j)$.
We note that $h_j = \Delta{x} w_j >0$ with the constants $w_j>0$ being the weights of a composite quadrature rule and $\Delta{x}>0$ is the spatial step.
In this way, for a grid function $\mathbf{g}$ the multiplication by $\mathbf{\mathbf{1}}^TH$ with $\mathbf{\mathbf{1}}= \left(1, \dots, 1\right)^T \in \mathbb{R}^n$, from the left, may be interpreted as a numerical quadrature  given
${\inp{\mathbf{1}}{\mathbf{g}}_H}
    =\sum_{j = 1}^n h_j  g(x_j) \approx \int_{0}^{L} g(x) dx$ .

Recently, the  DP SBP FD framework \cite{MATTSSON2017upwindsbp,DovgilovichSofronov2015,williams2024drp} was introduced to improve the accuracy and flexibility of numerical approximations of PDEs. The  DP SBP  operators are a pair of forward and backward finite difference stencils that together obey  the SBP property. As shown in \cite{williams2024drp}, the DP upwind SBP framework for the first derivative $d/dx$ \cite{MATTSSON2017upwindsbp,DovgilovichSofronov2015} can be expressed through the following assumptions:
\begin{enumerate}[($\mathbf{A}$.1)]
    \item\label{itm:dpsbp-measure} There exists $H: \mathbb{R}^n \to \mathbb{R}^n$ which defines a positive discrete measure
    \[
        \langle\mathbf{g}, \mathbf{g} \rangle_{H}
            = \mathbf{g}^T{H}\mathbf{g} > 0
            \quad \forall \, \mathbf{g} \in \mathbb{R}^n, 
        \quad\langle \mathbf{1}, \mathbf{g} \rangle_{H}
            = \sum_{j = 1}^n h_j  g(x_j) \to \int_{0}^{L} g(x) \dd{x}
            \quad \forall\, g \in L^2(\Omega).
    \]
%     %
%     %
    
    \item\label{itm:dpsbp-derivative} There exists a pair of linear operators $D_{\pm}: \mathbb{R}^n \to \mathbb{R}^n $ with
    $\left(D_{\pm} \mathbf{f}\right)_j = \dv{f}{x}|_{x=x_j}$,
    for all $j \in \{1, 2, \dots, n\}$  and $f \in V^p$ where $V^p$ is a polynomial space of at most degree $p \ge 0$.
    \item\label{itm:dpsbp-dpsbp} The linear operators $D_{\pm}$ together obey $\langle D_{+} \mathbf{f}, \mathbf{g} \rangle_{H} + \langle \mathbf{f}, D_{-} \mathbf{g} \rangle_{H} =f_ng_n - f_1g_1$ for all $\mathbf{f}, \, \mathbf{g} \in \mathbb{R}^n$.

    \item\label{itm:dpsbp-upwind} The linear operators $D_{\pm}$ together obey 
    %\begin{align}\label{eq:B4}
         $\langle \mathbf{f}, \left(D_{+}-D_{-} \right)\mathbf{f} \rangle_{H} \le 0$ for all $\mathbf{f} \in \mathbb{R}^n$.
\end{enumerate}
Assumption ($\mathbf{A}$.1) equips the solution space with a quadrature rule and a discrete $l_2$-norm defined by
$
\|\mathbf{g}\|^2_H :=\langle\mathbf{g}, \mathbf{g} \rangle_{H}> 0, \, \forall \, \mathbf{g} \in \mathbb{R}^n.
$
 Assumption ($\mathbf{A}$.2) encodes consistent discrete first derivative operators $D_{\pm}$ such that for a smooth function $u$, with $\vb u \in \R^{n}$ and  $u_j = u(x_j)$, we have
%%%
\[
  ({D_{\pm}}\mathbf{u})_j = \left.\dv{u}{x}\right|_{x=x_j} + \mathbb{T}_j^{\pm}, \quad \mathbb{T}_j^{\pm} = O(\Delta{x}^{p+1}), \quad \forall \, j \in \{1, 2, \cdots, n\}.  
\]
%%%
Assumption ($\mathbf{A}$.3) is the DP SBP property, which is equivalent
$Q_{\pm} = HD_{\pm}$,  $Q_- + Q_+^T ={\diag}([-1, 0, \cdots, 0, 1])$,  $H=H^T>0$.
In particular, Assumption ($\mathbf{A}$.3) equips $D_{\pm}$ with the integration by parts principle, which is critical for proving numerical stability.
Assumption ($\mathbf{A}$.4) is the upwind DP property, which is equivalent to
%%%
%
$A= H\left(D_{+}-D_{-} \right)$,  $A = A^T$,  $\mathbf{g}^TA\mathbf{g} \le 0$,  for all $\mathbf{g} \in \mathbb{R}^n$.
%%%
%%%
\begin{remark}
    The DP SBP DG framework of \cite{GLAUBITZ2025113841} satisfies Assumptions ($\mathbf{A}$.1)--($\mathbf{A}$.4) with $x_j$, $w_j>0$ being the nodes and weights of a Legendre-Gauss-Lobatto (LGL) numerical quadrature.
\end{remark}

The  SBP operators by themselves do not  result in stable numerical methods. For a given PDE system, SBP operators must be carefully applied in a manner that allows the proof of numerical stability, as long as suitable numerical treatments of boundary conditions are available.
As discussed in \cite{williams2024drp}, a main benefit of the DP-SBP framework \cite{MATTSSON2017upwindsbp, williams2024drp} over traditional SBP schemes \cite{BStrand1994} is that there are a pair of derivative operators and more degrees of freedom, which enable the flexibility and possibility of designing numerical schemes with special qualities, which could potentially capture  important characteristics of a target application.  For example, for conservative PDE systems, we can design energy conserving schemes, thanks to the DP-SBP property Assumption ($\mathbf{A}$.3), and the fact that if
%\begin{align}\label{eq:upwind}
    $\mathcal{Q}_{\pm} :=  {Q}_{\pm}  - \frac{1}{2}B$, ${Q}_{\pm} := H D_{\pm}$
%\end{align}
then $\mathcal{Q}_{-} +\mathcal{Q}_{+}^T =0$.
We can also construct entropy/energy dissipating schemes, for example to upwind transport terms in our systems or dissipate unresolved modes for nonlinear problems. This is mainly a consequence of upwind DP-SBP, Assumption ($\mathbf{A}$.4). In the current work we will focus on entropy dissipating schemes for robust and  accurate simulations of  nonlinear hyperbolic conservation laws.
\begin{remark}
The traditional SBP operator $(D, H)$—based on central FD~\cite{kreiss1974finite,BStrand1994,fernandez2014review,svard2014review}—or collocated DGSEM~\cite{gassner2013skew,taylor2010compatible,fluxo_SplitForm}—also satisfies the DP-SBP framework, that is Assumptions ($\mathbf{A}$.1)--($\mathbf{A}$.4), with   $D_- =D_+ =D$  and $\langle \mathbf{f}, \left(D_{+}-D_{-} \right)\mathbf{f} \rangle_{H} = 0$.
 Similarly, given the upwind DP-SBP operator $(D_-, D_+, H)$ which satisfies Assumptions ($\mathbf{A}$.1)--($\mathbf{A}$.4), the averaged operator
$
D := \frac12 \left(D_+ + D_-\right)
$
is a discrete derivative operator which satisfies the traditional SBP framework $(D, H)$. 
\end{remark}
%%%
\section{The numerical method and analysis}\label{sec:discrete_methods}
We derive a numerical method that replicates the continuous analysis at the discrete level by combining DP SBP operators to preserve model symmetries, ensure entropy stability, high-order accuracy, and maintain key invariants. In particular, we will outline how to construct conservative and provably entropy-stable semi-discrete schemes for nonlinear hyperbolic conservation laws using upwind DP-SBP operators in a multi-block framework. Although we focus on 1D systems for simplicity, these concepts extend to higher dimensions.

To begin, we consider the 1D hyperbolic conservation law
\begin{align} \label{eq:conservation-law_1D}
\partial_t \vb{u} + \partial_x \vb{f}(\vb{u}) = 0,  \,  x \in \Omega =[0, L], \quad  t \in [0,T],
\end{align}
where $\vb{f}(\vb{u}) $ is a continuous  flux function, $L>0$ is the length of the spatial interval $\Omega \subset \mathbb{R}$. We close the boundaries at $x \in \{0, L\}$ with periodic boundary conditions, that is  
\begin{align}\label{eq:periodic_bc}
   \vb{u}(0, t) = \vb{u}(L, t) \implies \vb{f}(\vb{u}(0, t)) = \vb{f}(\vb{u}(L, t)). 
\end{align}
%%%
Thus for smooth solutions,  the system \eqref{eq:conservation-law_1D} with \eqref{eq:periodic_bc}  is conservative and conserves total entropy, that is
%%%
\begin{align} \label{eq:conservation_skew-symmetric_1D}
 \partial_t \vb{\mathcal{U}}= \vb f(\vb{u}(0, t)) - \vb f(\vb{u}(L, t)) =0, \quad 
\dv{}{t} E(t) =   q(\vb{u}(0, t)) -  q(\vb{u}(L, t)) =0.
\end{align}
%%%
To be useful, numerical approximations of the nonlinear conservation law \eqref{eq:conservation-law_1D}, with  periodic boundary conditions \eqref{eq:periodic_bc}, must as far as possible mimic \eqref{eq:conservation_skew-symmetric_1D} at the discrete level, leading to conservative and entropy stable discrete approximations.%
\subsection{Multi-block DP SBP method}
We decompose the spatial domain into $K\ge 1$ elements/blocks and denote the $k$-th element/block $\Omega_k = [x_{k-1}, x_{k}]$, where  the element width is denoted $|\Omega_k| = x_k -x_{k-1} >0$, and the external boundaries are denoted $x_0 =0$, $x_{K} = L$. In each element the local solutions and fluxes are denoted ${\vec u}^k$ and ${\vec f}^k$. At an element interface $x_k$, shared by the two adjacent elements $\Omega_k$ and $\Omega_{k+1}$,  we will denote the jump of the solutions and fields across the interface as $\lJump {\vec u}^k\rJump = \left.{\vec u}^{k+1}\right|_{x = x_k}-\left.{\vec u}^{k}\right|_{x = x_k}$. 
\subsubsection{Multi-block DP SBP operators}
To keep the derivations and the analysis simple, we will split the spatial domain into two elements, with $K=2$ and
$\Omega = \Omega_k \cup \Omega_{k+1}$. The method and analysis extend to many more elements. For each element $\Omega_k$, $k \in \{1, 2\}$, we discretize the domain into $n$ grid points and sample the solutions on the grids. For the multi-block FD method the grids are equidistant nodes and for the DG method the grids are given by the LGL nodes. The semi-discrete solution is denoted $\mathbf{u}^k =  (\mathbf{u}_1^k, \mathbf{u}_2^k, \cdots, \mathbf{u}_n^k)^T$. The two elements are connected at the interface $x = x_k$ with the interface conditions
\[
\lJump{\vb{u}} \rJump:  = \mathbf{u}_1^{k+1}-\mathbf{u}_n^k= 0, \quad \lJump{\vb{f}} \rJump:  = \mathbf{f}_1^{k+1}-\mathbf{f}_n^k= 0.
\]
%%%
For the global solutions in the two elements, we introduce the  augmented  discrete solution vector and discrete flux vector, given by 
\[
\mathbf{u} = \begin{bmatrix}
\mathbf{u}_1^k, \cdots, \mathbf{u}_n^k,\mathbf{u}_1^{k+1},  \cdots, \mathbf{u}_n^{k+1}
\end{bmatrix}^T \in \mathbb{R}^{2mn}, \quad 
\mathbf{f} = \begin{bmatrix}
\mathbf{f}_1^k, \cdots, \mathbf{f}_n^k,\mathbf{f}_1^{k+1},  \cdots, \mathbf{f}_n^{k+1}
\end{bmatrix}^T  \in \mathbb{R}^{2mn}.    
\]
%%%
To implement the interface conditions and the periodic boundary conditions we also define the boundary and interface penalty matrices
{%\small
%\scriptsize
\footnotesize
%\small
\begin{align}\label{eq:SBP_SAT_and_Interface_Matrices}
\vec{B}_{I} = \begin{pmatrix}
-\vb{e}_n\vb{e}_n^T & \vb{e}_n\vb{e}_1^T\\
-\vb{e}_1\vb{e}_n^T & \vb{e}_1\vb{e}_1^T
\end{pmatrix},
\quad
\widetilde{\vec{B}}_{I} = \begin{pmatrix}
-\vb{e}_n\vb{e}_n^T & \vb{e}_n\vb{e}_1^T\\
\vb{e}_1\vb{e}_n^T & -\vb{e}_1\vb{e}_1^T
\end{pmatrix}, \quad \vec{B}_{n} = \begin{pmatrix}
\vb{e}_1\vb{e}_1^T & -\vb{e}_1\vb{e}_n^T\\
\vb{e}_n\vb{e}_1^T & - \vb{e}_n\vb{e}_n^T
\end{pmatrix},
\quad
\widetilde{\vec{B}}_{n} = \begin{pmatrix}
-\vb{e}_1\vb{e}_1^T & \vb{e}_1\vb{e}_n^T\\
\vb{e}_n\vb{e}_1^T & - \vb{e}_n\vb{e}_n^T
\end{pmatrix}.
\end{align}
}
For $m=1$, that is $\mathbf{u} \in \mathbb{R}^{2n}$ with $\lJump{\vb{u}}\rJump=u_1^{k+1}-u_n^{k}$, we have
{
$$
\vec{B}_{I}\mathbf{u} = \left(
0, \cdots, 0,\lJump{\vb{u}}\rJump, \lJump{\vb{u}}\rJump, 0,  \cdots, 0
\right)^T, \quad \widetilde{\vec{B}}_{I}\mathbf{u} = \left(
0, \cdots, 0,\lJump{\vb{u}}\rJump, -\lJump{\vb{u}}\rJump, 0,  \cdots, 0
\right)^T, 
$$
$$
\vec{B}_{n}\mathbf{u} = \left(
(u_1^{1}-u_n^{K}), 0,  \cdots, 0, (u_1^{1}-u_n^{K})
\right)^T, \quad \widetilde{\vec{B}}_{n}\mathbf{u} = \left(-
(u_1^{1}-u_n^{K}), 0,  \cdots, 0, (u_1^{1}-u_n^{K})
\right)^T.
$$
}
%%%
The interface matrices $\vec{B}_{I}$, $\widetilde{\vec{B}}_{I}$ will be used to weakly couple the elements together to the global domain and the boundary matrices $\vec{B}_{n}$, $\widetilde{\vec{B}}_{n}$ will weakly implement the period boundary conditions. Note also that the matrices $\widetilde{\vec{B}}_{I}$, $\widetilde{\vec{B}}_{n}$ are symmetric and negative semi-definite, that is $\widetilde{\vec{B}}_{I} = \widetilde{\vec{B}}_{I}^T$, $\widetilde{\vec{B}}_{n} = \widetilde{\vec{B}}_{n}^T$ and 
$
\mathbf{u}^T\widetilde{\vec{B}}_{I}\mathbf{u} = -\lJump{\vb{u}}\rJump^2=-(u_1^{k+1}-u_n^{k})^2\le 0,
$
$
\mathbf{u}^T\widetilde{\vec{B}}_{n}\mathbf{u} = -(u_1^{1}-u_n^{K})^2\le 0.
$
We introduce the multi-block operators
{\small
\begin{align}\label{eq:periodic_interface_SBP_SAT}
\vec{H} = \begin{pmatrix}
H & 0\\
0 & H
\end{pmatrix},
\quad
\vec{D} = \begin{pmatrix}
D & 0\\
0 & D
\end{pmatrix},
\quad
\vec{D}_{\pm} = \begin{pmatrix}
D_{\pm} & 0\\
0 & D_{\pm}
\end{pmatrix},
\quad 
\widetilde{\vec{D}}_{\pm} = \vec{D}_{\pm} + \frac12 \vec{H}^{-1}\left(\vec{B}_n + \vec{B}_I\right),
\end{align}
}
%%%
 where the penalty terms  weakly implement the interface and  the periodic boundary conditions. The penalized multi-block operators $\widetilde{\vec{D}}_{\pm}$ also satisfy the identities
 %\vspace{-0.25cm}
 {\small
\begin{align}\label{eq:B4_periodic}
\inp{\widetilde{\vec{D}}_{+} \mathbf{f}}{\mathbf{g}}_{\vec{H}} + \inp{ \mathbf{f}}{\widetilde{\vec{D}}_{-} \mathbf{g}}_{\vec{H}} = 0,
    \,
\inp{ \mathbf{f}}{\left(\widetilde{\vec{D}}_{+}-\widetilde{\vec{D}}_{-} \right)\mathbf{f}}_{\vec{H}}=    \inp{\mathbf{f}} {\left(\vec{D}_{+}-\vec{D}_{-} \right)\mathbf{f}}_{\vec{H}} \le 0, \quad \forall \, \mathbf{f}, \, \mathbf{g} \in \mathbb{R}^{2n}.
\end{align}
}
%%%
    % 
%
 For simplicity, we will ignore the $\sim$ in the penalized SBP operators and write $\vec{D}_{\pm}=\widetilde{\vec{D}}_{\pm}$ where $\widetilde{\vec{D}}_{\pm}$ are multi-block and  periodic operators implemented using penalties as in \eqref{eq:periodic_interface_SBP_SAT}.
 We define the multi-block spatial operators for systems
 $$\mathbf{D}_x= \left(I_m \otimes \vec{D}\right), \quad \mathbf{H}_x= \left(I_m \otimes \vec{H}\right), \quad{\vec{B}}_{Ix}= \left(I_m \otimes \widetilde{\vec{B}}_I\right), \quad{\vec{B}}_{nx}= \left(I_m \otimes \widetilde{\vec{B}}_n\right). 
 $$
 It also follows that
{
\begin{align}\label{eq:B4_periodic_systems}
\inp{{\vec{D}}_{x} \mathbf{f}}{\mathbf{g}}_{\vec{H}_x} + \inp{ \mathbf{f}}{{\vec{D}}_{x} \mathbf{g}}_{\vec{H}_x} = 0,
    \quad  \inp{\mathbf{f}} {\left( I_m \otimes \left(\vec{D}_{+}-\vec{D}_{-} \right)\right)\mathbf{f}}_{\vec{H}_x} \le 0, \quad \forall \, \mathbf{f}, \, \mathbf{g} \in \mathbb{R}^{2mn}.
\end{align}
}
and
$
{\vec{B}}_{Ix} = {\vec{B}}_{Ix}^T$, ${\vec{B}}_{nx} = {\vec{B}}_{nx}^T$,  $\mathbf{u}^T{\vec{B}}_{Ix}\mathbf{u}  \le 0$,  $\mathbf{u}^T{\vec{B}}_{Ix}\mathbf{u} \le 0$, for all $\mathbf{u} \in \mathbb{R}^{2mn}
$.
\subsubsection{Finite volume flux splitting and linear-stability}
To enable the effectiveness of the upwind  DP SBP operators we will need a splitting strategy. 
%%%%
%%%%
We will consider  classical flux splitting which is prevalent in the finite volume literature, see e.g. \cite{LeVeque1992conslaws,godunov1959,LeVeque_2002,LUNDGREN2020109784,ranocha2023highorder}.
%\todo[inline]{Insert classical finite volume papers}
%%%
In each element, $\Omega_k$, we introduce the upwind flux splitting where
\begin{align}\label{eq:flux-splitting}
{\vec f}_i^k(\vec u)  =  \underbrace{\frac12(\vec {f}_i^k +\gamma^k_i \vec{g}_i^k)}_{\vec f_i^+} +  \underbrace{\frac12(\vec f_i^k -\gamma^k_i \vec{g}_i^k)}_{\vec f_i^-} \equiv \vec f_i^+ + \vec f_i^-, \quad i = 1, 2, \ldots, m,
\end{align}
%%%
$\gamma_i^k \ge 0$ are (time-dependent) parameters and $\vec{g}_i^k$ are carefully chosen functions to induce dissipation when upwinded.  The parameters $\gamma_i^k > 0$ are chosen such that $\vec{f}_i^k$ and $\gamma_i^k \vec{g}_i^k$ have consistent units.
We will denote $\vec {f}^+$ the positive-going  flux and  $\vec {f}^-$ is  the negative-going flux. 
%%%%
If $\vec{g}_i^k=\vec{u}_i^k$ and $\gamma_i^k = \max_j |\lambda_j(A_x^k)|$ where $A_x^k = \partial_{\vec u^k}\vec f^k$, then we get the classical global but element-local Lax-Friederich flux splitting \cite{LeVeque1992conslaws,LeVeque_2002}.
%%%%
%%%%
%
%%%%
%%%%
Note that the element $\Omega_k$ is discretized on an $n$ grid points, and for a  uniform grid we have $\Delta{x} = |\Omega_k|/(n-1)$. We sample the solution and PDE flux on the grid, having
%%%
{%\small
\[
    \mathbf{u}^k(t) :=  \left(\mathbf{u}_1^k(t), \ldots, \mathbf{u}_n^k(t)\right)^T, \quad \mathbf{f}^{\pm}(\mathbf{u}^k) :=  \left({\vec f}^{\pm}(\mathbf{u}_1^k(t)), \ldots, {\vec f}^{\pm}(\mathbf{u}_n^k( t))\right)^T.
\]
}
%%%%
To approximate the divergence of the flux we will use the DP DG/FD operators to upwind the positive-going  flux  $\vec f^+$ and the negative-going flux $\vec f^-$  having 
%%%%
%%%%
%%%%
{\small
\[
    \partial_x \vec f(\vec u^k)  \approx \left(I_m \otimes {D}_{+}\right)\mathbf{f}^{-} + \left(I_m \otimes {D}_{-}\right)\mathbf{f}^{+}= \left(I_m \otimes {D}\right)\mathbf{f}^k - \frac{1}{2}\left(\Gamma^k \otimes\left({D}_{+}-{D}_{-}\right)\right)\mathbf{g}^k,
    \  {D} := \frac12 \left({D}_+ + {D}_-\right),
\]
}%
%%%
where $\Gamma^k = \fn{\diag}{[\gamma_1^k, \gamma_2^k, \cdots, \gamma_m^k]}$, $\gamma_i^k\ge 0$, and $\otimes$ denote the Kronecker product. 
Note that for smooth solutions, we have 
$ |\left(\Gamma^k \otimes\left({D}_{+}-{D}_{-}\right)\right)\mathbf{g}^k| = O(\Delta{x}^{p+1})$ where $p\ge 1$ is the order of accuracy of the discrete operators $D_{\pm}$.

 A two-element semi-discrete approximation of the systems of 1D conservation laws \eqref{eq:conservation-law_1D} reads
%%%%
{
\begin{align}\label{eq:semi-discrete-scalar-multi-block}
    \dv{}{t} \mathbf{u} + \mathbf{D}_x\mathbf{f} = \frac{1}{2}\mathbf{H}_x^{-1}\left(\boldsymbol{\alpha} \otimes\left(\vec{B}_{Ix}+\vec{B}_{nx}\right)\right)\mathbf{g} + \frac{1}{2}\Gamma\left(I_m \otimes\left(\vec{D}_{+}-\vec{D}_{-}\right)\right)\mathbf{g}, \quad 
\end{align}
}
%%%
% %%%%
where $I_{m} \in \mathbb{R}^{m\times m}$ is the identity matrix,  the diagonal and positive matrices $\boldsymbol{\alpha}$, $\boldsymbol{\Gamma}$ are given by
{\small
$$\boldsymbol{\alpha} = \frac{1}{2}\diag\left([\alpha_{1}^1+\alpha_{1}^2, \alpha_{2}^1+\alpha_{2}^2, \cdots, \alpha_{m}^1+\alpha_{m}^2]\right),
\quad \Gamma = \diag\left([\gamma_1^1, \gamma_2^1, \cdots, \gamma_m^1,\gamma_1^2, \gamma_2^2, \cdots, \gamma_m^2]\right)\otimes I_{2n},
$$
}
%%%
with $I_{2n} \in \mathbb{R}^{2n\times 2n}$ being the identity matrix, $\alpha_{i}^k, \gamma_i^k\ge0$, $\mathbf{f}$ is the discrete PDE flux and $\vec{g}$ is a carefully chosen function to induce dissipation in the volume. The first term in the right hand side of \eqref{eq:semi-discrete-scalar-multi-block} are surface terms at element faces induced by an upwind numerical flux, the second term in the right hand side of \eqref{eq:semi-discrete-scalar-multi-block} are volume terms inside the element and induced by the upwind property of the DP-SBP framework. Note that for the traditional SBP framework we have $\vec{D}_{+}=\vec{D}_{-}=\vec{D}$, thus, the volume upwind terms vanish identically which yields the standard multi-block SBP scheme or the DGSEM. It is also noteworthy that the semi-discrete approximation \eqref{eq:semi-discrete-scalar-multi-block} is compactly written for the two-elements case $K=2$. However, the method is valid in a single element, with $K=1$, and extends to arbitrary number of elements $K\ge 1$.

For the semi-discrete approximation \eqref{eq:semi-discrete-scalar-multi-block}, the discrete  conservative principle follows. We formulate the result as the theorem.
%%%
%%%%
\begin{theorem}
Consider the semi-discrete upwind DP-SBP approximation \eqref{eq:semi-discrete-scalar-multi-block} of the system \eqref{eq:conservation-law_1D}.
  If   the 3-tuple $(D_-, D_+, H)$ satisfies the upwind DP SBP operator given by Assumptions ($\mathbf{A}$.1)--($\mathbf{A}$.4), then the semi-discrete approximation \eqref{eq:semi-discrete-scalar-multi-block} is element-local and globally conservative.  That is, $\forall \, i = 1, 2, \cdots, m$, the total quantities $\vb{\mathcal{U}}^k_{ih}(t)={\inp{\vec 1}{ \vb{u}_i^k}_{H_x}}$ and $\vb{\mathcal{U}}_{ih} = \sum_{k=1}^{K}\vb{\mathcal{U}}^k_{ih}$ satisfy
%%%%%%
%\begin{align}\label{eq:semi-discrete-scalar-conservation}
{\small
    \begin{align}
   & \dv{}{t} \vb{\mathcal{U}}^k_{ih}(t)  = \frac12 \left(f_{ni}^{k-1}+f_{1i}^k - \frac{\alpha_{i}^{k-1}+\alpha_{i}^{k}}{2}\lJump{\vb{g}_{i}^{k-1}}\rJump\right) - \frac12 \left(f_{ni}^{k}+f_{1i}^{k+1} - \frac{\alpha_{i}^{k}+\alpha_{i}^{k+1}}{2}\lJump{\vb{g}_{i}^{k}}\rJump\right),
  \\
   \nonumber
   & \dv{}{t} \vb{\mathcal{U}}_{ih}(t)  =0.
  \end{align}
  }
\end{theorem}
%%%%%
\begin{proof}
    We consider $\dv{}{t}{\inp{\vec 1}{ \vb{u}_i^k}_{H_x}}$ and use the SBP property to obtain
   \begin{align*}
    \dv{}{t} \vb{\mathcal{U}}^k_{ih}(t)  &=\inp{ D {\vec 1}}{  \vb f_i^k}_H + \frac {\gamma_i^k}{2} \inp{ \left(D_+- D_-\right)\mathbf{1}}{\mathbf{g}_i^k}_H + \frac12 \left(f_{ni}^{k-1}+f_{1i}^k - \frac{\alpha_{i}^{k-1}+\alpha_{i}^{k}}{2}\lJump{\vb{g}_{i}^{k-1}}\rJump\right) \\
    &- \frac12 \left(f_{ni}^{k}+f_{1i}^{k+1} - \frac{\alpha_{i}^{k}+\alpha_{i}^{k+1}}{2}\lJump{\vb{g}_{i}^k}\rJump\right).
    \end{align*}
%%%
    Since $ D {\vec 1}=0$ and $ D_{\pm} {\vec 1}=0$, then the first two terms in the right hand side vanish and we have 
    \begin{align*}
    \dv{}{t} \vb{\mathcal{U}}^k_{ih}(t)  &= \frac12 \left(f_{ni}^{k-1}+f_{1i}^k - \frac{\alpha_{i}^{k-1}+\alpha_{i}^{k}}{2}\lJump{\vb{g}_{i}^{k-1}}\rJump\right) - \frac12 \left(f_{ni}^{k}+f_{1i}^{k+1} - \frac{\alpha_{i}^{k}+\alpha_{i}^{k+1}}{2}\lJump{\vb{g}_{i}^k}\rJump\right).
    \end{align*}
%%%
    Summing contributions from all elements the flux terms at the interfaces cancel out giving 
    $$
     \dv{}{t} \vb{\mathcal{U}}_{ih}(t)  =  \dv{}{t} \sum_{k=1}^{K}\vb{\mathcal{U}}^k_{ih}  =0.
    $$
%%%
    The proof is complete.
\end{proof}

%%%%%
% 
The proof of (nonlinear) entropy-stability for the semi-discrete approximation \eqref{eq:semi-discrete-scalar-multi-block} does not necessarily follow. Following \cite{LUNDGREN2020109784,ranocha2023highorder,GLAUBITZ2025113841}, we can prove that the semi-discrete approximation \eqref{eq:semi-discrete-scalar-multi-block} is linearly-stable.
%%%%%
%%%%%
%
We will formulate the result as the theorem.
%%%%%
%%%%%
\begin{theorem}
Consider the semi-discrete approximation \eqref{eq:semi-discrete-scalar-multi-block} for a  linear PDE flux, with  $\vec{f} = A_x\vec{u}$ where $A_x\in \mathbb{R}^{m\times m}$ a symmetric (or symmetrizable) constant matrix.
%
%%%
 For the classical global Lax-Friedrich flux splitting, that is if $\vec{g}^k=\vec{u}^k$ and $\gamma_i^k = \max_x |\max_j\lambda_j(A_x^k)|$, and  the 3-tuple $(D_-, D_+, H)$ satisfies the upwind DP SBP operator given by Assumptions ($\mathbf{A}$.1)--($\mathbf{A}$.4), with $\|\mathbf{u}\|_{H_x}^2=  \sum_{k=1}^{K}\|\mathbf{u}^k\|_{H}^2$, then we have
%%%%%%
\begin{align}\label{eq:semi-discrete-scalar-energy}
\dv{}{t} \|\mathbf{u}\|_{H_x}^2= -\sum_{k=1}^{K}\sum_{i=1}^m\frac{\alpha_{i}^{k} + \alpha_{i}^{k+1}}{2}\lJump{\mathbf{u}_i^k}\rJump^2  -\sum_{k=1}^{K}\sum_{i=1}^m{\gamma_i^k\inp{ \mathbf{u}_i^k} {\left(D_{-}-D_{+} \right)\mathbf{u}_i^k}_{H}}  \le 0.
\end{align}
%%%%%%
%where $A = H(D_+ - D_-)$ is a symmetric and negative semi-definite matrix.
\end{theorem}
\begin{proof}
   We set $\vec{f} = A_x\vec{u}$ where $A_x\in \mathbb{R}^{m\times m}$ a symmetric constant coefficients matrix and we proceed with the energy method. That is from the left, we multiply \eqref{eq:semi-discrete-scalar-multi-block} with $\mathbf{u}^T\left(I_m \otimes H\right)$ and add the transpose of the product, and we have
   \begin{align*}
   \dv{}{t} \|\mathbf{u}^k\|_{H_x}^2&= -\underbrace{(\mathbf{u}^k)^{T}\left(A_x \otimes \left(Q+Q^T\right)\right)\mathbf{u}_k}_{(\mathbf{u}^k_n)^TA_x\mathbf{u}^k_n-(\mathbf{u}_1^k)^TA_x\mathbf{u}^k_1} - \frac{1}{2}(\mathbf{u}_1^k)^T\left(A_x\mathbf{u}^k_1 - A_x\mathbf{u}^{k-1}_n +\frac{\alpha_{i}^{k-1} + \alpha_{i}^{k}}{2}\lJump{\vb{u}^{k-1}}\rJump \right) \\
   &+ \frac{1}{2}(\mathbf{u}_n^k)^T\left(A_x\mathbf{u}^{k}_n - A_x\mathbf{u}^{k+1}_1 +\frac{\alpha_{i}^{k} + \alpha_{i}^{k+1}}{2}\lJump{\vb{u}^{k}}\rJump \right)  -  \sum_{i=1}^m{\gamma_i^k\inp{ \mathbf{u}_i^k} {\left(D_{-}-D_{+} \right)\mathbf{u}_i^k}_{H}}.   
   \end{align*}
%%%
   Simplifying further we have
   \begin{align*}
   \dv{}{t} \|\mathbf{u}^k\|_{H}^2&=   \frac{1}{2}(\mathbf{u}_1^k)^T\left(A_x\mathbf{u}^{k-1}_n + A_x\mathbf{u}^{k}_1 -\frac{\alpha_{i}^{k-1} + \alpha_{i}^{k}}{2}\lJump{\vb{u}^{k-1}}\rJump \right) \\
   &-\frac{1}{2}(\mathbf{u}_n^k)^T\left(A_x\mathbf{u}^{k}_n + A_x\mathbf{u}^{k+1}_1 -\frac{\alpha_{i}^{k} + \alpha_{i}^{k+1}}{2}\lJump{\vb{u}^{k}}\rJump \right)    -  \sum_{i=1}^m{\gamma_i^k\inp{ \mathbf{u}_i^k} {\left(D_{-}-D_{+} \right)\mathbf{u}_i^k}_{H}}.   
   \end{align*}
%%%
    Note that the periodic boundaries give 
$
\frac{\alpha_i^{K} + \alpha_i^{K+1}}{2}\lJump{\mathbf{u}_i^K}\rJump^2 =\frac{\alpha_i^{K} + \alpha_i^{1}}{2}({\mathbf{u}_{i1}^{1}-\mathbf{u}_{in}^K})^2.
$
   Summing contributions from all elements the PDE flux terms at the interfaces cancel out giving 
    $$
     \dv{}{t} \|\mathbf{u}\|_{H_x}^2= \dv{}{t} \sum_{k=1}^{K}\|\mathbf{u}^k\|_{H_x}^2 = -\sum_{k=1}^{K}\sum_{i=1}^m\frac{\alpha_{i}^{k} + \alpha_{i}^{k+1}}{2}\lJump{\mathbf{u}_i^k}\rJump^2 -\sum_{k=1}^{K}\sum_{i=1}^m{\gamma_i^k\inp{ \mathbf{u}_i^k} {\left(D_{-}-D_{+} \right)\mathbf{u}_i^k}_{H}}  \le 0.
    $$
%%%
    This completes the proof.
\end{proof}
The linearly-stable upwind semi-discrete approximation \eqref{eq:semi-discrete-scalar-multi-block} extends to multi-dimensional nonlinear conservation laws via tensor products and various flux splitting methods, such as global Lax-Friedrich, Steger-Warming, and van Leer-Hänel \cite{LeVeque1992conslaws,godunov1959,LeVeque_2002}. For detailed discussions, see \cite{LUNDGREN2020109784,ranocha2023highorder}. 
This study focuses on provably nonlinearly/entropy stable, conservative, high-order schemes for nonlinear hyperbolic systems. Numerical results presented later in this study demonstrate that nonlinearly-stable schemes can be more robust than linearly/entropy stable ones, especially for shocks and turbulent flows.

\subsubsection{The entropy-stable semi-discrete approximation}
We now describe a method which can yield a provably nonlinearly/entropy stable, conservative, and high-order accurate numerical method for systems of nonlinear conservation laws. The derived numerical method is robust and minimizes spurious oscillations. The method combines three main ingredients: 1) \textbf{the DP SBP DG/FD operators}, together with 2) \textbf{skew-symmetric reformulation}, and 3) \textbf{upwind flux splitting}.

Again, for simplicity we consider the 1D conservation law,
\[
    \partial_t \vec u + \partial_x \vec f(\vec u) = 0 \iff \partial_t \vb{u} + \vb{F}(\vb u, \vb{f}(\vb{u}), \partial_x) =0,
\]
%%%
 where $\vb{F}(\vb u, \vb{f}(\vb{u}), \partial_x)\equiv \partial_x \vec f(\vec u)$ is the skew-symmetric form given in \eqref{eq:conservation-law_skew-symmetry}  and defined through \eqref{eq:skew-sym-flux}  such that entropy stability can be proven by using the energy method with only integration by parts, and without the chain/product rule. 
 We will begin with the following Lemma which is a discrete analogue of \eqref{eq:skew-sym-flux}.
 %%%
 %%%
 \begin{lemma}
    Consider the semi-discrete approximation of the skew-symmetric form of the element local and global (multi-elements) differential operators
    \[
        \vb{F}( \vec u^k, \vec{f}({\vec u}^k), {D}) \approx \vb{F}(\vb u^k, \vb{f}(\vb{u}^k), \partial_x),
        \quad
        \vb{F}( \vec u, \vec{f}({\vec u}), \vec{D}_x) \approx \vb{F}(\vb u, \vb{f}(\vb{u}), \partial_x),
    \]
%%%
    where $D\approx \partial_x$ is the local derivative without flux corrections, $\vec{f}({\vec u}^k)$ is the PDE flux and $\vec{D}_x \approx \partial_x$ is the global derivative operator which includes flux corrections at the element boundaries. If the 2-tuple $(D,H)$  satisfy the traditional SBP operator, then
    $$
{\inp{\vec g^k}{ \vb{F}(\vb u^k, \vb{f}(\vb{u}^k),  D)}_H}=-{\inp{D \vec 1}{   \vb q^k}_H + \vb q^k\left.\right|_{x_{k}} - \vb q^k\left.\right|_{x_{k-1}} = \vb q^k\left.\right|_{x_{k}} - \vb q^k \left.\right|_{x_{k-1}}},
    $$
    $$
    {\inp{\vec g}{ \vb{F}(\vb u, \vb{f}(\vb{u}),  \vec{D}_x)}_{H_x}}=-{\inp{\vec{D}_x \vec 1}{   \vb q}_{H_x}} =0,
    $$
%%%
     where $\vb g^k \in \mathbb{R}^{mn}$, $\vb g \in \mathbb{R}^{Kmn}$ are the semi-discrete  entropy variables and $\vb q^k \in \mathbb{R}^{mn}$, $\vb q \in \mathbb{R}^{Kmn}$ are the semi-discrete entropy flux.
 \end{lemma}
 Next we generalize the flux splitting \eqref{eq:flux-splitting}, by introducing the upwind matrix $\Gamma^k \in \mathbb{R}^{m\times m}$, which will  yield entropy stability when upwinded with the DP SBP DG/FD operators. Specifically, we chose 
 \begin{align}\label{eq:flux-splitting-general}
{\vec f}^k(\vec u)  =  \underbrace{\frac12(\vec {f}^k + \Gamma^k \vec{g}^k)}_{\vec f^+} +  \underbrace{\frac12(\vec f^k -\Gamma^k \vec{g}^k)}_{\vec f^-} \equiv \vec f^+ + \vec f^-,
\end{align}
 with the entropy variables $\vec{g}^k_i = \partial_{\vec u_i^k} e({\vec u}^k)$, and $\alpha_i^k > 0$ and $\Gamma^k \in \mathbb{R}^{m\times m}$ are constants within the element and chosen such that $\vec{f}_i^k$, $(\Gamma^k \vec{g}^k)_i$ and $\alpha_i^k \vec{g}_i^k$ have consistent physical units. In addition the symmetric part of $\Gamma^k$ is positive semi-definite $\Gamma^k + (\Gamma^{k})^T \ge 0$. The flux splitting \eqref{eq:flux-splitting} corresponds to a diagonal upwind matrix with $\Gamma^k = \fn{\diag}{[\gamma_1^k, \gamma_2^k, \cdots, \gamma_m^k]}$.
Combining  \textbf{the skew-symmetric form} with  \textbf{entropy-stable upwind splitting}, and  \textbf{the DP SBP operators} yield the  two-element semi-discrete approximation of the systems of 1D conservation laws \eqref{eq:conservation-law_1D}
%%%%
\begin{align}\label{eq:skew_symm_upwind_SBP_SAT-multi-block}
    \dv{}{t} \mathbf{u} +\vb{F}( \vec u, {\vec f}({\vec u}), \vec{D}_x) = \frac{1}{2}\mathbf{H}_x^{-1}\left(\boldsymbol{\alpha} \otimes\left(\vec{B}_{Ix}+\vec{B}_{nx}\right)\right)\mathbf{g} + \frac{1}{2}\Gamma\left(I_m \otimes\left(\vec{D}_{+}-\vec{D}_{-}\right)\right)\mathbf{g}, 
\end{align}
%%%
where $I_{m} \in \mathbb{R}^{m\times m}$ is the identity matrix and the  matrices $\boldsymbol{\alpha}, \boldsymbol{\Gamma} \in \mathbb{R}^{2mn\times 2mn}$ are given by
{\small
$$
\boldsymbol{\alpha} = \frac{1}{2}\diag\left([\alpha_{1}^1+\alpha_{1}^2, \alpha_{2}^1+\alpha_{2}^2, \cdots, \alpha_{m}^1+\alpha_{m}^2]\right), \quad
\Gamma = \diag\left([\Gamma^1,\Gamma^2]\right)\otimes I_{2n},
$$
}%
%%%
%
with $I_n \in \mathbb{R}^{2n\times 2n}$ being the identity matrix, and $\alpha_{i}^k\ge0$, $\Gamma^k + (\Gamma^{k})^T \ge 0$. As before, the first term in the right hand side of \eqref{eq:skew_symm_upwind_SBP_SAT-multi-block} are surface terms at element faces induced by an upwind numerical flux, and the second term in the right hand side of \eqref{eq:skew_symm_upwind_SBP_SAT-multi-block} are volume terms inside the element, and result from the upwind property of the DP SBP framework. For the traditional SBP framework we have $\vec{D}_{+}=\vec{D}_{-}=\vec{D}$, thus, the volume upwind terms vanish identically which yields the standard entropy stable multi-block SBP scheme or the DGSEM. It is also noteworthy that the semi-discrete approximation \eqref{eq:skew_symm_upwind_SBP_SAT-multi-block} is compactly written for the two-elements case $K=2$. However, the method is valid in a single element, with $K=1$, and extends to an arbitrary number of elements $K\ge 1$. 
%  
%%%
Again, for smooth solutions, the right hand side of \eqref{eq:skew_symm_upwind_SBP_SAT-multi-block} is very small, that is 
$ |\left(\Gamma^k \otimes\left(D_{+}-D_{-}\right)\right)\mathbf{g}^k| = O(\Delta{x}^{p+1})$ where $p\ge 1$ is the order of accuracy of the discrete operators $D_{\pm}$.
%%%

We note that the main difference between our proposed method \eqref{eq:skew_symm_upwind_SBP_SAT-multi-block} and the so-called upwind methods \cite{LUNDGREN2020109784,ranocha2023highorder} exemplified by \eqref{eq:semi-discrete-scalar-multi-block} is the discretization of the semi-discrete flux form $\mathbf{D}_x\mathbf{f} \approx \partial_x\mathbf{f}$ and the skew-symmetric form $\vb{F}( \vec u, {\vec f}({\vec u}), \vec{D}_x) \approx \partial_x\mathbf{f}$, on the left hand sides of \eqref{eq:skew_symm_upwind_SBP_SAT-multi-block}, \eqref{eq:semi-discrete-scalar-multi-block}, and the choice of the upwind variables in $\mathbf{g}$ in the right hand sides of  \eqref{eq:skew_symm_upwind_SBP_SAT-multi-block}, \eqref{eq:semi-discrete-scalar-multi-block}. Herein, in \eqref{eq:skew_symm_upwind_SBP_SAT-multi-block}, we have chosen the skew-symmetric form $\vb{F}( \vec u, {\vec f}({\vec u}), \vec{D}_x)$,  the entropy variables $\vec{g}_i^k = \partial_{\vec u_i^k} e({\vec u^k})$ and the volume upwind matrices $\Gamma$ such that we can prove nonlinear/entropy stability at the semi-discrete level.
\begin{remark}
 The significance and novelty of the current study is the systematic combination of the DP SBP FD/DG operators together with the skew-symmetric and upwind flux splitting that allows the design of robust and high-order accurate numerical schemes for systems of nonlinear hyperbolic conservation laws. To the best of our knowledge, this has never been reported in the literature.
\end{remark}
Next, we will prove entropy stability for the method \eqref{eq:skew_symm_upwind_SBP_SAT-multi-block}. To begin, we approximate the total entropy \eqref{eq:total_entropy} by 
%%%
%%%
\begin{align}\label{eq:total_entropy_discrete}
    E_h(t):= \inp{\vec 1}{  e}_{H_x}=\sum_{k=1}^{K} \inp{\vec 1}{  e^k}_H.
\end{align}
%%%
%%%
The following definition, which is similar to the continuous analogue Definition \ref{def:cont_nonlinear_stability}, is central to this study.
\begin{defn}\label{def:disc_nonlinear_stability}
Consider the semi-discrete approximation \eqref{eq:skew_symm_upwind_SBP_SAT-multi-block} and the denote the semi-discrete total entropy $E_h(t) \ge 0$, defined by \eqref{eq:total_entropy_discrete}. The  semi-discrete approximation \eqref{eq:skew_symm_upwind_SBP_SAT-multi-block} is called entropy-stable (entropy-conserving) if  $E(t) \le E_h(0)$ ($E_h(t) = E_h(0)$) for all $\vb{u}^k \in \mathbb{R}^{mn}$, $k=1, 2, \cdots, K$ and $t \in [0,T]$. 
\end{defn}
 %%%
Our main theorem proves that the semi-discrete approximation \eqref{eq:skew_symm_upwind_SBP_SAT-multi-block} is  entropy-stable.
%%%
% %\vspace{-0.25cm}
  \begin{theorem}
Consider the multi-element DP-SBP semi-discrete approximation \eqref{eq:skew_symm_upwind_SBP_SAT-multi-block} where $\vb{F}( \vec u, \vec{f}({\vec u}), \vec{D}_x)$ is the multi-block semi-discrete approximation of the skew-symmetric form with  periodic boundary conditions.
  If $\vec{g}_i^k = \partial_{\vec u_i^k} e({\vec u}^k)$, $\alpha_i^k\ge 0$ and $\Gamma^k + (\Gamma^{k})^T \ge 0$, and  the 3-tuple $(D_-, D_+, H)$ satisfies the upwind DP SBP operator, then we have
%%%%%%
%%\vspace{-0.25cm}
{\small
\begin{align}\label{eq:semi-discrete-scalar-energy_theorem}
\dv{}{t} E_h(t)
            =  -\frac{1}{2}\sum_{k=1}^{K}\sum_{i=1}^m\frac{\alpha_{i}^k + \alpha_{i}^{k+1}}{2}\lJump{\mathbf{g}_i^k}\rJump^2-\frac{1}{2}\sum_{k=1}^{K}{\inp{ \mathbf{g}^k} {\left(\left(\frac{\Gamma^k + (\Gamma^{k})^T}{2}\right)\otimes\left(D_{-}-D_{+} \right)\right)\mathbf{g}^k}_{H}} 
              \le 0.
\end{align}
}
\end{theorem}
\begin{proof}
   Consider
   {\small
   %\footnotesize
        \begin{align*} 
            \underbrace{\inp{\vec g}{ \partial_t\vb{u}}_{H_x}}_{\dv{}{t} E_h(t) = \dv{}{t}\inp{\vec 1}{  e}_{H_x}}
            + \underbrace{\inp{\vec g}{ \vb{F}(\vb u, \vb{f}(\vb{u}),  \vec{D}_x)}_{H_x}}_{\inp{\vec 1}{ \vec{D}_x \vb q}_{H_x} = 0}
            =  &-\frac{1}{2}\sum_{k=1}^{K-1}\sum_{i=1}^m\frac{\alpha_{i}^k + \alpha_{i}^{k+1}}{2}\lJump{\mathbf{g}_i^k}\rJump^2\\
            &-\frac{1}{2}\sum_{k=1}^{K}{\inp{ \mathbf{g}^k} {\left(\left(\frac{\Gamma^k + (\Gamma^{k})^T}{2}\right)\otimes\left(D_{-}-D_{+} \right)\right)\mathbf{g}^k}_{H}} 
            \le 0.
        \end{align*}
    }
    Note that the periodic boundaries give 
$
\frac{\alpha_i^{K} + \alpha_i^{K+1}}{2}\lJump{\mathbf{g}_i^K}\rJump^2 =\frac{\alpha_i^{K} + \alpha_i^{1}}{2}({\mathbf{g}_{i1}^{1}-\mathbf{g}_{in}^K})^2.
$
%%%
    Then we have
     {\small
   %\footnotesize
        \begin{align*} 
            \dv{}{t} E_h(t)
            =  &-\frac{1}{2}\sum_{k=1}^{K}\sum_{i=1}^m\frac{\alpha_{i}^k + \alpha_{i}^{k+1}}{2}\lJump{\mathbf{g}_i^k}\rJump^2-\frac{1}{2}\sum_{k=1}^{K}{\inp{ \mathbf{g}^k} {\left(\left(\frac{\Gamma^k + (\Gamma^{k})^T}{2}\right)\otimes\left(D_{-}-D_{+} \right)\right)\mathbf{g}^k}_{H}}  \le 0.
        \end{align*}
    }
%%%
    %
\end{proof}
\subsection{Numerical conservation properties}
%\subsubsection{Numerical conservation properties}
We will show that our semi-discrete scheme \eqref{eq:skew_symm_upwind_SBP_SAT-multi-block} is conservative. Conservative properties of the scheme will be important in ensuring the convergence of the numerical
method for nonlinear problems with weak solutions \cite{Lax-Wendroff-1960,godunov1959}. In particular, we will show that \eqref{eq:skew_symm_upwind_SBP_SAT-multi-block} satisfies the discrete equivalence of \eqref{eq:conservation_flux}. Here we will focus on global conservation properties of the scheme. However, DG-type element local conservation properties can also be shown. First, we state the following Lemma.
%%%
 \begin{lemma}\label{lem:conservative_flux}
     Consider the multi-element semi-discrete approximation of the skew symmetric form
     \begin{align}
         \vb{F}( \vec u, \vec{f}({\vec u}), \vec{D}_x) \approx \vb{F}(\vb u, \vb{f}(\vb{u}), \partial_x),
     \end{align}
     where $\vec{D}_x\approx \partial_x$  including flux corrections at element boundaries, $\vec{f}({\vec u})$ is the PDE flux and $\vb{F}(\vb u, \vb{f}(\vb{u}), \partial_x)$ is the skew-symmetric form of the divergence of the flux satisfying \eqref{eq:conservation_flux}. If $\vec{D}_x = I_m \otimes \vec{D} $ and  the 2-tuple $(\vec{D},\vec{H})$ satisfies the first identity in  \eqref{eq:B4_periodic_systems}, 
     then
     \begin{align}
         {\inp{\vec 1}{ \vb{F}_i(\vb u, \vb{f}(\vb{u}),  \vec{D}_x)}_{H_x}}=-{\inp{ \vec{D} {\vec 1}}{  \vb f_i}_{H} = \vb{0}}, \quad \forall \, i = 1, 2, \cdots, m, 
     \end{align}
%%%
     where  $\vb f_i \in \mathbb{R}^{2n}$ are the components of the semi-discrete PDE flux restricted on the grid.
 \end{lemma}
 The second main theorem proves that the DP-SBP method \eqref{eq:skew_symm_upwind_SBP_SAT-multi-block} is globally conservative.
  \begin{theorem}
Consider the  multi-block upwind DP-SBP semi-discrete approximation \eqref{eq:skew_symm_upwind_SBP_SAT-multi-block} where $\vb{F}( \vec u, \vec{f}({\vec u}), \vec{D}_x)$ is the semi-discrete approximation of the symmetric form with the interface and   periodic boundary conditions implemented weakly.
  If Lemma \ref{lem:conservative_flux} holds and  the 3-tuple $(D_-, D_+, H)$ satisfies the  DP SBP operator given by Assumptions ($\mathbf{A}$.1)--($\mathbf{A}$.4), then the semi-discrete total quantity $\vb{\mathcal{U}}_{ih}(t)={\inp{\vec 1}{ \vb{u}_i}_H}$, for all $i = 1, 2, \cdots, m$, satisfies
%%%%%%
{\small
\begin{align}\label{eq:semi-discrete-scalar-conservation}
    \dv{}{t} \vb{\mathcal{U}}_{ih}(t) =\inp{ \vec{D} {\vec 1}}{  \vb f_i}_H + \frac {\alpha_i}{2} \inp{ \left(\vec{B}_{I}+\vec{B}_{n}\right)\mathbf{1}}{\mathbf{g}_i} + \frac {1}{2} \inp{ \left(\vec{D}_+- \vec{D}_-\right)\mathbf{1}}{(\Gamma\mathbf{g})_i}_H = \vb{0}.
\end{align}
}
%%%%%% 
\end{theorem}
%%%%%
%%%%%
\begin{proof}
%See \cite{}.
   Consider
   {\small
        \begin{align} 
            \underbrace{\inp{\vec 1}{  \partial_t\vb{u}_i}_H}_{\dv{}{t} \vb{\mathcal{U}}_{ih}(t) = \dv{}{t}\inp{\vec 1}{ \vb{u}_i}_H}+
            \underbrace{\inp{\vec 1}{ \vb{F}_i(\vb u, \vb{f}(\vb{u}),  D)}_H}_{-\inp{\vec{D} \vec 1}{  \vb f_i}_H = 0} =  \frac {\alpha_i}{2} \inp{ \left(\vec{B}_{I}+\vec{B}_{n}\right)\mathbf{1}}{\mathbf{g}_i} + \frac {1}{2} \inp{ \left(\vec{D}_+- \vec{D}_-\right)\mathbf{1}}{(\Gamma\mathbf{g})_i}_H =0.
        \end{align}
        }
        Then we have
        %\begin{align} 
        $
            \dv{}{t} \vb{\mathcal{U}}_{ih}(t)
             =  0 \implies \vb{\mathcal{U}}_{ih}(t) = \vb{\mathcal{U}}_{ih}(0), \quad \forall \, t \in [0, T].
             $
        %\end{align}
% which  yields a conservative scheme.
 %%  
\end{proof}
%%%%%
%%%%%
Thus, the semi-discrete upwind DP-SBP method \eqref{eq:skew_symm_upwind_SBP_SAT-multi-block} is conservative and entropy stable, for arbitrary nonlinear  conservation laws in skew-symmetric form. 
\begin{remark}
   As noted earlier, a DG-type local numerical conservation property also holds. This can be demonstrated by restricting the discrete scalar products to one element and then utilizing the SBP property to convert the volume terms into surface terms only. When contributions from all elements are summed, these surface terms cancel out, thereby reaffirming the global conservation property.
\end{remark}
In the Appendix, namely sections \ref{sec:Burgers-equation-numerical-method}, \ref{sec:swe-equation-numerical-method}, \ref{sec:euler-equation-numerical-method}, we  give a few examples in one space dimension by considering the inviscid Burger's equation, the nonlinear shallow water equation and the compressible Euler equations. By using tensor products the analysis extends to higher spatial dimensions, $2$D and $3$D.
Below, we provide a few remarks to compare and contrast the proposed method \eqref{eq:skew_symm_upwind_SBP_SAT-multi-block} with traditional and state-of-the-art methods.

\begin{remark}
The primary difference between the proposed method \eqref{eq:skew_symm_upwind_SBP_SAT-multi-block} and both traditional and some recent  approaches in the literature, such as \cite{Sjögreen2019,GERRITSEN1996245,FISHER2013353}, lies in the second term on the right-hand side of \eqref{eq:skew_symm_upwind_SBP_SAT-multi-block}, which acts to control grid-scale oscillations. For smooth solutions, where $\vb{u} \in C^1\left(\Omega \times [0,T]\right)$, the terms on the right-hand side of \eqref{eq:skew_symm_upwind_SBP_SAT-multi-block} are proportional to $O(\Delta{x}^{p +1})$ and are bounded by the truncation error $O(\Delta{x}^{p})$. This ensures they remain small for smooth solutions on well-resolved meshes. However, for unresolved meshes and non-smooth solutions, where $\vb{u} \notin C^1\left(\Omega \times [0,T]\right)$, the terms on the right-hand side of \eqref{eq:skew_symm_upwind_SBP_SAT-multi-block} can become significant and will dissipate entropy when $\Gamma^k + (\Gamma^{k})^T >0$.
\end{remark}
\begin{remark}
We elaborate on the efficiency of the proposed method \eqref{eq:skew_symm_upwind_SBP_SAT-multi-block}.  It is  particularly noteworthy that the discrete upwind operator comprises the difference of first derivative DP SBP operators, which are computationally efficient to evaluate. This contrasts with traditional and state-of-the-art artificial hyper-viscosity operators \cite{Szepessy1989,STIERNSTROM2021110100,Nordström2006,Mattsson_etal2004,MattssonandRider2015}, which require the evaluation of high even-order spatial derivatives $\sim \partial^{2q}/\partial x^{2q}$, $q \ge  1$. Such operators can be computationally expensive, may compromise the conservation of important invariants, and often further restrict explicit time-step sizes for the fully discrete problem.
\end{remark}
\begin{remark}
The volume upwind term on the right-hand side of \eqref{eq:skew_symm_upwind_SBP_SAT-multi-block} acts as a penalty on the discrete gradient of the solution on the grid. This penalty weakly enforces smoothness and serves as a measure of unresolved features in the numerical solution. For smooth solutions, where $\vb{u} \in C^1\left(\Omega \times [0,T]\right)$, this term is small and tends to vanish as the mesh is refined. Conversely, for unresolved meshes and non-smooth solutions, where $\vb{u} \notin C^1\left(\Omega \times [0,T]\right)$, the term can be significant and  suppresses grid-scale errors. 
Thus, the proposed numerical framework inherently incorporates a "filter" that is accurate up to discretization errors. Its purpose is to identify and  resolve regions where the solution is poorly captured or contains discontinuities. Numerical experiments presented later in this paper support the analysis.

\end{remark}
\begin{remark}
The proposed DP SBP FD/DG  method
 \eqref{eq:skew_symm_upwind_SBP_SAT-multi-block} is structurally equivalent to entropy-stable multi-block SBP FD/DG methods develop in the literature, such as \cite{Carpenter_et_al2014,fluxo_SplitForm,gassner2016skewsym_swe,CHAN2018346,GassnerWinters2021,Ranocha2018,giraldo2002nodal}. In particular when  $\Gamma =0$, the method reduces the so-called entropy DG schemes which are predominantly available in the literature. However, the DP SBP operators are designed to be upwind, that is they come with
    some built-in dissipation everywhere, as opposed to standard DG methods for hyperbolic PDEs which can only induce dissipation through numerical fluxes acting at element interfaces. 
\end{remark}

\section{Numerical Results} \label{sec:numerical-experiments}
In this section, we present numerical experiments to validate the theoretical analysis outlined in the previous sections. These experiments are designed to assess the accuracy and stability of the numerical method. In particular, we will examine the robustness of the method and the improvements it offers over standard approaches. No artificial viscosity, limiting, or filtering is employed in any of the numerical experiments presented herein. The numerical methods are implemented in the Julia programming language~\cite{Julia-2017}. We utilize the DP FD operators~\cite{MATTSSON2017upwindsbp} and the DRP DP FD operators~\cite{williams2024drp} provided in the Julia package \emph{SummationByPartsOperators.jl}~\cite{ranocha2021sbp}. Our DG schemes employ the DP DG operators developed by Glaubitz et al.~\cite{GLAUBITZ2025113841} on LGL nodes, with the parameter \(\lambda_n = -0.1\).

For time integration, we will employ the explicit five-stage, fourth-order accurate Strong Stability Preserving Runge-Kutta (SSP RK) method as described in \cite{Ruuth2006_SSPRK} as well as the four stage third-order
accurate adaptive SSP RK method of Kraaijevanger~\cite{Kraaijevanger1991} with
the error estimator of Fekete et al.~\cite{Fekete2022} and the step size
controller and implementation by Ranocha et al.~\cite{Ranocha2021SSPRK}. Both time-stepping methods are  provided by the Julia package \emph{OrdinaryDiffEq.jl}~\cite{DifferentialEquations.jl-2017}. All figures in the subsequent numerical experiments are generated using the Julia package \emph{Makie.jl}~\cite{DanischKrumbiegel2021} or \emph{ParaView}~\cite{Ahrens2005ParaViewAE}. The simulation code is publicly available at \url{https://github.com/Dougal-s/paper-2024-DP-SBP-nonlinear}, ensuring reproducibility for interested readers. %

The numerical experiments are conducted with increasing complexity. We begin with the 1D Burgers' equation, followed by the 1D, then 2D, compressible Euler equations of gas dynamics. Detailed numerical experiments for the 1D and 2D SWEs are documented in the Appendix \ref{swe-num-experiments}. 
Throughout, we compare the linearly stable semi-discrete approximation \eqref{eq:semi-discrete-scalar-multi-block}, utilizing local Lax-Friedrichs flux splitting with volume upwinding $\Gamma >0$, with the standard entropy-stable multi-block SBP FD scheme or DGSEM \eqref{eq:skew_symm_upwind_SBP_SAT-multi-block} without volume upwinding $\Gamma =0$, against the entropy-stable DP DG/FD scheme \eqref{eq:skew_symm_upwind_SBP_SAT-multi-block} with volume upwinding for $\Gamma >0$.

\subsection{Inviscid Burgers' Equation}
We begin with the 1D inviscid Burgers' equation, with the PDE $f(u) = u^2/2$,  and verify numerical accuracy and robustness.

To verify accuracy, we force Burgers' equation to satisfy the exact solution
$u = 2 + 0.3 \fn{\sin}{2\pi(x - t)}$ on the spatial domain $x \in [-1, 1]$.
We use a fixed timestep of \(\Delta t = 0.01\Delta x\) and run the simulation on a sequence of varying discretization parameters  until the final time $t=2$. The $l^2$ errors and the convergence rates of the errors are shown in \cref{fig:burgers-MMS-convergence}. The observed numerical convergence rates are slightly better than the theoretical rates. 
\begin{figure}[thbp]
    \centering
    \begin{subcaptionblock}{0.47\textwidth}
        \centering
        \includegraphics[scale=\figurescaling]{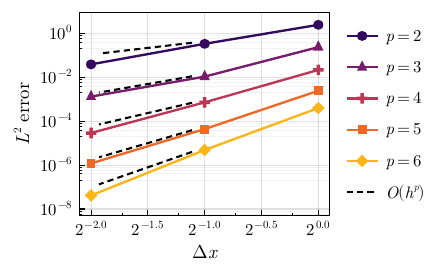}
        \caption{ DP DG method.}
    \end{subcaptionblock}\hfill
    \begin{subcaptionblock}{0.52\textwidth}
        \centering
        \includegraphics[scale=\figurescaling]{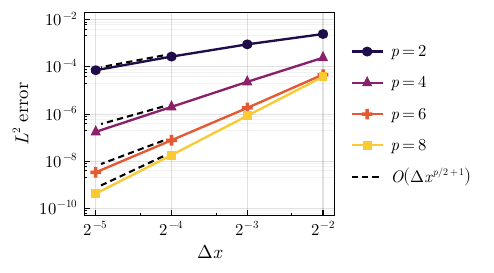}
        \caption{DP FD method on 4 elements.}
    \end{subcaptionblock}
    %\vspace{-0.5cm}
    \caption{Numerical errors and convergence of errors for the Burgers' equation.}
    \label{fig:burgers-MMS-convergence}
    %\vspace{-0.5cm}
\end{figure}

To investigate the robustness of the schemes, we consider the  initial condition
%\small
%\begin{align*}
%
    $u(x,0) = \exp[-(10x - 3)^2]$
    %\quad
    on the spatial domain $x \in [0, 1]$.
%\end{align*}
%\vspace{-0.5cm}
%%%
For the DG scheme we vary the number of elements $4\le K\le 16$ and the polynomial degree $3\le p\le 7$. We consider high-order accurate DP FD operators of interior order of accuracy $7,8,9$ discretized with $8$ elements and vary the number of nodes within an element, $17\le n\le 65$.
We use the fixed timestep of \(\Delta t = 0.01\Delta x\) and set the final time $t = 10$.
Snapshots of the solution for the DP FD and DP DG methods on a given mesh are shown in \Cref{fig:burgers-gaussian-state}.
\begin{figure}[htb!]
    \centering
    \includegraphics[scale=\figurescaling]{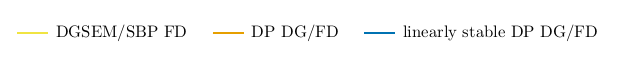}
    \\[-1em]
    \begin{subcaptionblock}{\textwidth}
        \centering
        \includegraphics[scale=\figurescaling]{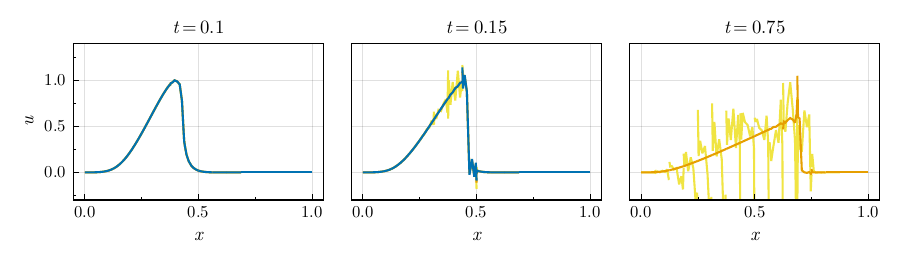}
        \\[-1.0em]
        \caption{DG methods with degree 8 polynomials on 16~elements.}
    \end{subcaptionblock}
    %\vspace{-1cm}
    \begin{subcaptionblock}{\textwidth}
        \centering
        \includegraphics[scale=\figurescaling]{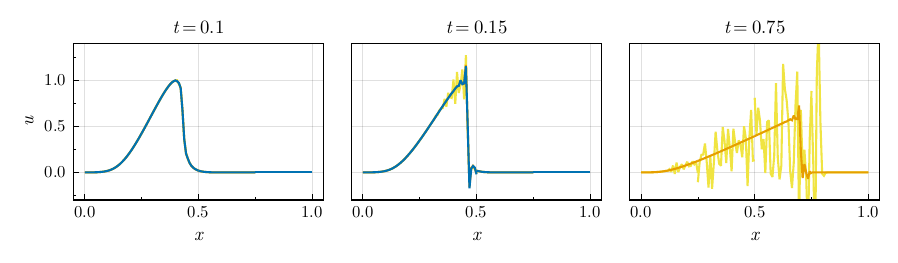}
        \\[-1.0em]
        \caption{8th order FD methods on 4~elements with 33~nodes for each element.}
    \end{subcaptionblock}
    \\[-0.5em]
    \caption{
    Snapshots of the solution to the Burger's equation for a Gaussian initial condition.}
    \label{fig:burgers-gaussian-state}
    %\vspace{-0.5cm}
\end{figure}
Although the initial datum is smooth, a right-going shock forms at $t \approx 0.1$. As time progresses, numerical oscillations emerge from the shock front. By $t=0.2$, the linearly stable scheme (Lax-Friedrich flux splitting) breaks down. In contrast, the entropy-stable DGSEM/SBP schemes (without interface upwinding and $\Gamma=0$) and our DP SBP FD/DG schemes (with interface upwinding and $\Gamma >0$)) remain stable, producing bounded numerical solutions.
\Cref{tab:burgers-gaussian-crash-times-fd,tab:burgers-gaussian-crash-times-dg} document the crash or final times of various schemes across different orders of accuracy and mesh resolutions. Notably, the DGSEM/SBP schemes with $\gamma =0$ generate severe numerical oscillations that corrupt the solution throughout the domain. Importantly, our DP SBP FD/DG schemes (with $\gamma >0$) effectively suppress these oscillations while maintaining numerical stability.
In \Cref{fig:burgers-gaussian-conservation}, we plot the relative change over time of the total entropy, \(E_h(t) =\inp{1}{\vec u^2/2}_H\), and the total "mass," \(\mathcal{U}(t)=\inp{1}{\vec u}_H\). These verify entropy stability and demonstrate the conservative properties of the numerical schemes.
%
% Despite the initial data being sufficiently smooth, a right going shock forms at $t = 0.1$. As time increases, numerical oscillations are generated from the shock front. At $t=0.2$, the linearly stable (Lax-Friederich flux splitting with $\gamma >0$) crashes. The  numerical solutions of the entropy stable DGSEM/SBP schemes (with interface upwinding and $\gamma =0 $) and our DP SBP FD/DG schemes (with interface upwinding and $\gamma >0 $) are stable with bounded numerical solutions.
% \Cref{tab:burgers-gaussian-crash-times-fd,tab:burgers-gaussian-crash-times-dg} document the crash/final times of the various  schemes for different order of accuracy and mesh resolutions. We note, however, the DGSEM/SBP with $\gamma =0 $ schemes generate wild numerical oscillations which corrupts the solutions everywhere. It is significantly important to note that our DP SBP FD/DG schemes (with $\gamma >0 $) tames the wild oscillations while ensuring numerical stability. 
% %%
% In \Cref{fig:burgers-gaussian-conservation} we plot the relative change with time of the total  entropy, \(E_h(t) =\inp{1}{\vec u^2/2}_H\) and total "mass", \(\mathcal{U}(t)=\inp{1}{\vec u}_H\) which verify entropy stability and demonstrate the conservative properties of the numerical schemes.

% Timestepping:
%     t ∈ [0, 10]
%     5-stage 4-th order fixed timestep SSPRK
%     Δt = 10⁻² × Δx where Δx = minᵢ (xᵢ₊₁ - xᵢ)

\begin{table}[htb!]
    \centering
    \scalebox{\tablescaling}{
        \begin{tabular}{@{}r*{3}{@{}c@{}*{5}{r}@{}}}
	\toprule
    &
	  & \multicolumn{5}{c}{linearly stable DG} &
	  & \multicolumn{5}{c}{DGSEM} &
	  & \multicolumn{5}{c}{DP DG}
      \\
	\cmidrule{3-7}
	\cmidrule{9-13}
	\cmidrule{15-19}
    & \hspace{2.5em}
	  & \multicolumn{5}{c}{polynomial degree} & \hspace{2em}
	  & \multicolumn{5}{c}{polynomial degree} & \hspace{2em}
	  & \multicolumn{5}{c}{polynomial degree} \\
	K &
    & 3 & 4 & 5 & 6 & 7 &
    & 3 & 4 & 5 & 6 & 7 &
    & 3 & 4 & 5 & 6 & 7 \\
	\midrule
	4 &
    & \crash{7.96} & \crash{2.08} & \crash{1.07} & \crash{0.82} & \crash{0.57} &
	& \nocrash{10} & \nocrash{10} & \nocrash{10} & \nocrash{10} & \nocrash{10} &
	& \nocrash{10} & \nocrash{10} & \nocrash{10} & \nocrash{10} & \nocrash{10} \\
	8 &
    & \nocrash{10} & \crash{3.07} & \crash{1.06} & \crash{0.74} & \crash{0.37} &
	& \nocrash{10} & \nocrash{10} & \nocrash{10} & \nocrash{10} & \nocrash{10} &
	& \nocrash{10} & \nocrash{10} & \nocrash{10} & \nocrash{10} & \nocrash{10} \\
	12 &
    & \nocrash{10} & \nocrash{10} & \nocrash{10} & \crash{0.90} & \crash{0.39} &
	& \nocrash{10} & \nocrash{10} & \nocrash{10} & \nocrash{10} & \nocrash{10} &
	& \nocrash{10} & \nocrash{10} & \nocrash{10} & \nocrash{10} & \nocrash{10} \\
	16 &
    & \nocrash{10} & \nocrash{10} & \nocrash{10} & \nocrash{10} & \crash{0.70} &
	& \nocrash{10} & \nocrash{10} & \nocrash{10} & \nocrash{10} & \nocrash{10} &
	& \nocrash{10} & \nocrash{10} & \nocrash{10} & \nocrash{10} & \nocrash{10} \\
	\bottomrule
\end{tabular}}
    \caption{Crash or final times of the DG simulations of the Burger's equation. The linearly stable DG with volume upwinding $\Gamma>0$ crashes for several configuration while the entropy stable DGSEM without volume upwinding $\Gamma=0$ and our DP DG with volume upwinding $\Gamma>0$ do not crash.}
    \label{tab:burgers-gaussian-crash-times-dg}
     \vspace{-0.5cm}
\end{table}
%\vspace{-0.75cm}
\begin{table}[htb!]
    \centering
    \scalebox{\tablescaling}{
        \begin{tabular}{@{}r*{3}{@{}c@{}*{3}{r}@{}}}
	\toprule
    &
	  & \multicolumn{3}{c}{linearly stable FD} &
	  & \multicolumn{3}{c}{SBP FD} &
	  & \multicolumn{3}{c}{DP FD}
      \\
	\cmidrule{3-5}
	\cmidrule{7-9}
	\cmidrule{11-13}
    & \hspace{2em}
	  & \multicolumn{3}{c}{accuracy order} & \hspace{2.5em}
	  & \multicolumn{3}{c}{accuracy order} & \hspace{2.5em}
	  & \multicolumn{3}{c}{accuracy order} \\
	n &
    & 7 & 8 & 9 &
    & 7 & 8 & 9 &
    & 7 & 8 & 9 \\
	\midrule
	17 &
	& \nocrash{10} & \crash{0.20} & \crash{0.19} &
	& \nocrash{10} & \nocrash{10} & \nocrash{10} &
	& \nocrash{10} & \nocrash{10} & \nocrash{10} \\
	33 &
	& \nocrash{10} & \crash{0.21} & \crash{0.21} &
	& \nocrash{10} & \nocrash{10} & \nocrash{10} &
	& \nocrash{10} & \nocrash{10} & \nocrash{10} \\
	49 &
	& \nocrash{10} & \crash{0.21} & \crash{0.21} &
	& \nocrash{10} & \nocrash{10} & \nocrash{10} &
	& \nocrash{10} & \nocrash{10} & \nocrash{10} \\
	65 &
    & \nocrash{10} & \crash{0.22} & \crash{0.22} &
	& \nocrash{10} & \nocrash{10} & \nocrash{10} &
	& \nocrash{10} & \nocrash{10} & \nocrash{10} \\
	\bottomrule
\end{tabular}}
    \caption{Crash or final times of the FD simulations of the Burger's equation. The linearly stable DG with volume upwinding $\Gamma>0$ crashes for several configuration while the entropy stable SBP FD without volume upwinding $\Gamma=0$ and our DP FD with volume upwinding $\Gamma>0$ do not crash.}
    \label{tab:burgers-gaussian-crash-times-fd}
     \vspace{-0.5cm}
\end{table}
\begin{figure}[htbp]
    \centering
    \includegraphics[scale=\figurescaling]{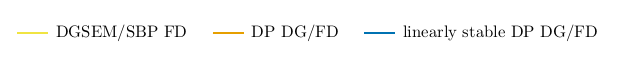} \\[-1em]
    \begin{subcaptionblock}{0.49\textwidth}
        \centering
        \includegraphics[scale=\figurescaling]{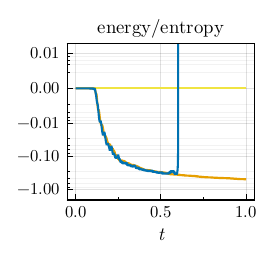}\hspace{-0.5em}%
        \includegraphics[scale=\figurescaling]{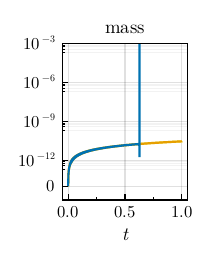}
        \\[-1em]
        \caption{DG with degree 8 polynomials on 16 elements.}
    \end{subcaptionblock}\hfill%
    \begin{subcaptionblock}{0.49\textwidth}
        \centering
        \includegraphics[scale=\figurescaling]{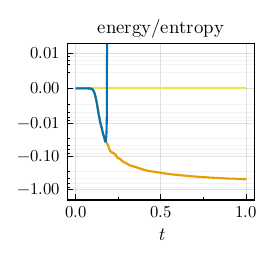}\hspace{-0.5em}%
        \includegraphics[scale=\figurescaling]{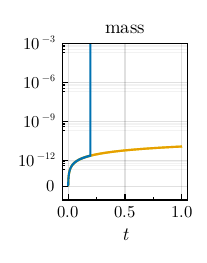}
        \\[-1em]
        \caption{8th order FD on 4 elements with 33 nodes each.}
    \end{subcaptionblock}%
   % \vspace{-0.5cm}
    \caption{Relative change in the total entropy  and total mass  for the Burger's equation.}
    \label{fig:burgers-gaussian-conservation}
    % \vspace{-0.5cm}
\end{figure}
\subsection{Compressible Euler Equations}
Finally, we consider the  Euler equations of compressible gas dynamics in 1D and 2D. We will investigate robustness and verify accuracy. For the linearly stable DP DG/FD schemes we will consider both local Lax-Friedrichs and van Leer-H\"anel flux splitting.

\subsubsection{Sod Shock Tube Problem}
We consider the Sod shock tube problem~\cite{Sod1978} with the initial conditions given by
{\small
\begin{align*}
    \varrho_0 = \begin{cases}
        1 & \text{if } x < 0 \\
        \frac 1 8 & \text{if } x \ge 0
    \end{cases},\quad
    u_0 = 0,\quad
    p_0 = \begin{cases}
        1 & \text{if } x < 0 \\
        \frac{1}{10} & \text{if } x \ge 0
    \end{cases},\quad
    x \in [-6,6],\quad 
    t \in [0,2].
\end{align*}
}
Again, for the DG scheme we vary the number of elements $32\le K\le 128$ and the degree of the polynomial approximation $3\le p\le 6$. We also consider the  DP FD difference operators with interior orders of accuracy from $5$ to $9$  discretized with $16$ elements and vary the number of nodes $17\le n\le 65$ within an element.
We use a fixed timestep of \(\Delta t = 0.002 \Delta x\) and run the simulation until the final time, $t=2$. Snapshots of the mass density $\rho$, momentum density $\rho v$, and energy density $\rho e$ are shown in \cref{fig:ce-sod-shock-snaphot} at $t=1.5$. 
%%%
%%%
\begin{figure}[htbp]
    \centering
    \includegraphics[scale=\figurescaling]{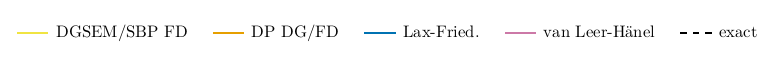} \\[-0.5em]
    \begin{subcaptionblock}{0.48\linewidth}
        \centering
        \includegraphics[scale=\figurescaling]{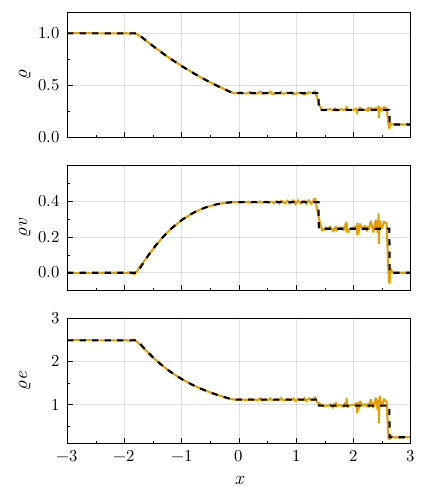}
        \caption{DG methods using degree 6 polynomials on 64 elements.}
    \end{subcaptionblock}\hfill
    \begin{subcaptionblock}{0.48\linewidth}
        \centering
        \includegraphics[scale=\figurescaling]{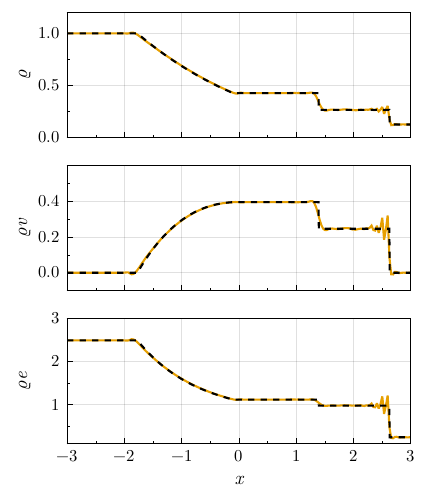}
        \caption{6th order FD methods on 12 elements with 33 nodes each.}
    \end{subcaptionblock}
    %\vspace{-0.5cm}
    \caption{Snapshots of the conserved variables at \(t=1.5\) for the Sod shock tube problem.}
    \label{fig:ce-sod-shock-snaphot}
  %  \vspace{-0.5cm}
\end{figure}
%%%
%%%
In Figure \ref{fig:ce-sod-shock-conservation} we have plotted the relative changes in total mass, total momentum, total energy and total entropy. Note that   our DP SBP FD/DG schemes (with volume upwinding $\Gamma>0$) are stable and completed the simulation until the final time without crashing.  However all other schemes, including the linearly stable schemes (with volume upwinding $\Gamma>0$) and the standard SBP/DGSEM (without volume upwinding $\Gamma=0$) crashed before the final simulation time.
%%%
%\vspace{-0.5cm}
\begin{figure}[htbp]
    \centering
    \includegraphics[scale=\figurescaling]{images/comp-euler-1D/sod-shock-tube/N_7_deriv_order_6_deriv_type_GlaubitzEtal2024_01_nblocks_64/legend-entr-horizontal.pdf} \\[-1em]
    \begin{subcaptionblock}{\textwidth}
        \centering
        \includegraphics[scale=\figurescaling]{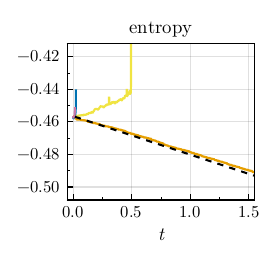}%
        \includegraphics[scale=\figurescaling]{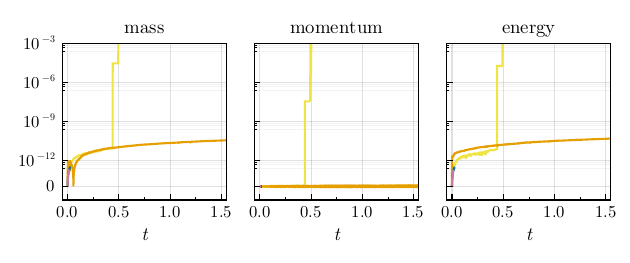}
     %   \vspace{-0.2cm}
        \caption{DG methods using 6th order polynomials on 64 elements.}
    \end{subcaptionblock}
    \begin{subcaptionblock}{\textwidth}
        \centering
        \includegraphics[scale=\figurescaling]{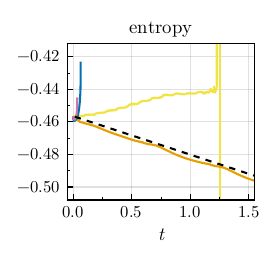}%
        \includegraphics[scale=\figurescaling]{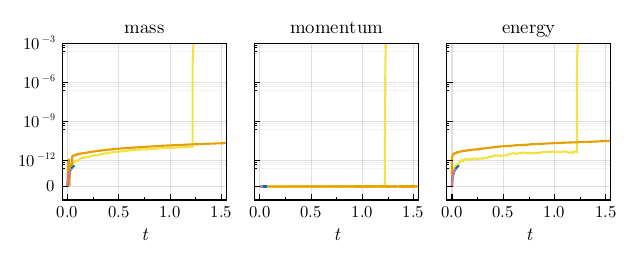}
      %  \vspace{-0.2cm}
        \caption{6th order FD methods on 12 elements with 33 nodes each.}
    \end{subcaptionblock}
   % \vspace{-0.85cm}
    \caption{The change over time in the total entropy and the relative change over time in the total mass, momentum, and energy for the Sod shock tube problem.}
    \label{fig:ce-sod-shock-conservation}
 %   \vspace{-0.25cm}
\end{figure}
%%%
 In addition, our entropy stable DP SBP FD/DG schemes (with  $\Gamma>0$) are conservative and entropy consistent, as shown in \cref{fig:ce-sod-shock-conservation}.
%We use a fixed timestep of \(\Delta t = 0.002\Delta x\).
As before, we have also performed numerical simulations with different discretisation configurations. The crash or final  times are reported in  \cref{tab:ce-sod-shock-crash-times-fd,tab:ce-sod-shock-crash-times-dg}. It is also important to reiterate that our entropy stable DP SBP FD/DG schemes with volume upwinding ($\Gamma >0 $) run until the final time, $t=2$, for all configurations without crashing. These again demonstrate the robustness and efficacy of  our entropy stable DP SBP FD/DG schemes (with $\Gamma >0 $).
%
% Crash times for varying orders and mesh resolutions can be found in \cref{tab:ce-sod-shock-crash-times-fd,tab:ce-sod-shock-crash-times-dg} for DP FD and DP DG methods respectively.
%
% \Cref{fig:ce-sod-shock-snapshot} contains snapshots of the solution at time \(t=1\) for the DP FD and DP DG methods respectively.
%
% \Cref{fig:ce-sod-shock-conservation,fig:ce-sod-shock-entropy} show the change in the conserved quantities and the entropy with respect to time.
%
\begin{table}[htbp]
    \centering
    % \begin{subcaptionblock}{0.49\textwidth}
    %     \centering
    %     \scalebox{\tablescaling}{
    %     \input{data/comp-euler-1D/sod-shock-tube/crash-times-Mattsson2017-LxF}
    %     }
    %     \caption{Lax-Friedrichs}
    % \end{subcaptionblock}
    % \begin{subcaptionblock}{0.49\textwidth}
    %     \centering
    %     \scalebox{\tablescaling}{
    %     \input{data/comp-euler-1D/sod-shock-tube/crash-times-Mattsson2017-vLH}
    %     }
    %     \caption{van Leer-H\"anel}
    % \end{subcaptionblock}
    % \begin{subcaptionblock}{0.49\textwidth}
    %     \centering
    %     \scalebox{\tablescaling}{
    %     \input{data/comp-euler-1D/sod-shock-tube/crash-times-Mattsson2017-Ns}
    %     }
    %     \caption{entropy conservative}
    % \end{subcaptionblock}
    % \begin{subcaptionblock}{0.49\textwidth}
    %     \centering
    %     \scalebox{\tablescaling}{
    %     \input{data/comp-euler-1D/sod-shock-tube/crash-times-Mattsson2017-UNs}
    %     }
    %     \caption{entropy stable}
    % \end{subcaptionblock}
    \scalebox{\tablescaling}{
        \begin{tabular}{@{}r@{}*{4}{c@{}*{4}{r}@{}}}
	\toprule
	& \hspace{2em}
	& \multicolumn{4}{c}{Lax-Friedrichs} & \hspace{2.5em}
    & \multicolumn{4}{c}{van Leer-H\"anel} & \hspace{2.5em}
    & \multicolumn{4}{c}{SBP FD} & \hspace{2.5em}
    & \multicolumn{4}{c}{DP FD} \\
	\cmidrule{3-6}
	\cmidrule{8-11}
	\cmidrule{13-16}
	\cmidrule{18-21}
	&
    & \multicolumn{4}{c}{accuracy order} &
    & \multicolumn{4}{c}{accuracy order} &
    & \multicolumn{4}{c}{accuracy order} &
    & \multicolumn{4}{c}{accuracy order} \\
	$n$ &
    & 6 & 7 & 8 & 9 &
    & 6 & 7 & 8 & 9 &
    & 6 & 7 & 8 & 9 &
    & 6 & 7 & 8 & 9 \\
	\midrule
	17 &
    & % LxF
    \crash{0.04} & \crash{0.04} & \crash{0.04} & \crash{0.04} &
    & % vLH
    \crash{0.01} & \crash{0.01} & \crash{0.01} & \crash{0.01} &
	& % Ns
    \nocrash{2} & \nocrash{2} & \crash{1.25} & \crash{0.79} &
	& % UNs
    \nocrash{2} & \nocrash{2} & \nocrash{2} & \nocrash{2} \\
	33 &
    & % LxF
    \crash{0.02} & \crash{0.02} & \crash{0.02} & \crash{0.02} &
    & % vLH
    \crash{0.01} & \crash{0.01} & \crash{0.00} & \crash{0.00} &
	& % Ns
    \crash{0.93} & \crash{1.05} & \crash{0.67} & \crash{0.65} &
	& % UNs
    \nocrash{2} & \nocrash{2} & \nocrash{2} & \nocrash{2} \\
	65 &
    & % LxF
	\crash{0.01} & \crash{0.01} & \crash{0.01} & \crash{0.01} &
    & % vLH
    \crash{0.00} & \crash{0.00} & \crash{0.00} & \crash{0.00} &
	& % Ns
    \crash{0.53} & \crash{1.33} & \crash{0.29} & \crash{0.22} &
	& % UNs
    \nocrash{2} & \nocrash{2} & \nocrash{2} & \nocrash{2} \\
	\bottomrule
\end{tabular}

    }
    \caption{
        Crash or final simulation times for  FD methods with the $16$ elements for the Sod shock tube problem.
    }
    \label{tab:ce-sod-shock-crash-times-fd}
    \vspace{-0.5cm}
\end{table}

%\todo[inline]{We need to be consistent with legend DGSEM/SBP FD}

\begin{table}[htbp]
    \centering
    % \begin{subcaptionblock}{0.49\textwidth}
    %     \centering
    %     \scalebox{\tablescaling}{
    %         \input{data/comp-euler-1D/sod-shock-tube/crash-times-GlaubitzEtal2024-LxF}
    %     }
    %     \caption{Lax-Friedrichs}
    % \end{subcaptionblock}
    % \begin{subcaptionblock}{0.49\textwidth}
    %     \centering
    %     \scalebox{\tablescaling}{
    %         \input{data/comp-euler-1D/sod-shock-tube/crash-times-GlaubitzEtal2024-vLH}
    %     }
    %     \caption{van Leer-H\"anel}
    % \end{subcaptionblock}
    % \begin{subcaptionblock}{0.49\textwidth}
    %     \centering
    %     \scalebox{\tablescaling}{
    %         \input{data/comp-euler-1D/sod-shock-tube/crash-times-GlaubitzEtal2024-Ns}
    %     }
    %     \caption{entropy conservative}
    % \end{subcaptionblock}
    % \begin{subcaptionblock}{0.49\textwidth}
    %     \centering
    %     \scalebox{\tablescaling}{
    %         \input{data/comp-euler-1D/sod-shock-tube/crash-times-GlaubitzEtal2024-UNs}
    %     }
    %     \caption{entropy stable}
    % \end{subcaptionblock}
    \scalebox{\tablescaling}{
        \begin{tabular}{@{}r@{}*{4}{c@{}*{4}{r}@{}}}
    \toprule
    & \hspace{2em}
    & \multicolumn{4}{c}{Lax-Friedrichs} &\hspace{2.5em}
    & \multicolumn{4}{c}{van Leer-H\"anel} &\hspace{2.5em}
    & \multicolumn{4}{c}{DGSEM} &\hspace{2.5em}
    & \multicolumn{4}{c}{DP DG} \\
    \cmidrule{3-6}
    \cmidrule{8-11}
    \cmidrule{13-16}
    \cmidrule{18-21}
    &
    & \multicolumn{4}{c}{polynomial degree} &
    & \multicolumn{4}{c}{polynomial degree} &
    & \multicolumn{4}{c}{polynomial degree} &
    & \multicolumn{4}{c}{polynomial degree} \\
    K &
    & 3 & 4 & 5 & 6 &
    & 3 & 4 & 5 & 6 &
    & 3 & 4 & 5 & 6 &
    & 3 & 4 & 5 & 6 \\
    \midrule
    32 &
    % LxF
    & \crash{0.10} & \crash{0.06} & \crash{0.04} & \crash{0.03} &
    % vLH
    & \crash{0.02} & \crash{0.01} & \crash{0.01} & \crash{0.01} &
    % Ns
    & \nocrash{2} & \crash{0.97} & \crash{0.73} & \crash{0.99} &
    % UNs
    & \nocrash{2} & \nocrash{2} & \nocrash{2} & \nocrash{2} \\
    64 &
    % LxF
    & \crash{0.05} & \crash{0.03} & \crash{0.02} & \crash{0.01} &
    % vLH
    & \crash{0.01} & \crash{0.01} & \crash{0.00} & \crash{0.00} &
    % Ns
    & \crash{1.55} & \crash{0.45} & \crash{0.35} & \crash{0.36} &
    % UNs
    & \nocrash{2} & \nocrash{2} & \nocrash{2} & \nocrash{2} \\
    128 &
    % LxF
    & \crash{0.02} & \crash{0.01} & \crash{0.01} & \crash{0.01} &
    % vLH
    & \crash{0.01} & \crash{0.00} & \crash{0.00} & \crash{0.00} &
    % Ns
    & \crash{0.76} & \crash{0.22} & \crash{0.18} & \crash{0.18} &
    % UNs
    & \nocrash{2} & \nocrash{2} & \nocrash{2} & \nocrash{2} \\
    \bottomrule
\end{tabular}

    }

    \caption{
         Crash or final simulation times for the DG methods for the Sod shock tube problem.
    }
    \label{tab:ce-sod-shock-crash-times-dg}
    \vspace{-0.5cm}
\end{table}
%\vspace{-0.25cm}
% \begin{figure}[htbp]
%     \centering
%     \includegraphics[scale=\figurescaling]{images/comp-euler-1D/sod-shock-tube/N=7_deriv_order=6_deriv_type=GlaubitzEtal2024(-0.1)_nblocks=64_LwTp9jMbsIe/legend-entr-horizontal.pdf}
%     \begin{subcaptionblock}{0.49\textwidth}
%         \centering
%         \includegraphics[scale=\figurescaling]{images/comp-euler-1D/sod-shock-tube/N=33_deriv_order=7_deriv_type=Mattsson2017_nblocks=12_5K32JxnDn1a/entr-wide-Thermodynamic Entropy.pdf}
%     \end{subcaptionblock}
%     \begin{subcaptionblock}{0.49\textwidth}
%         \centering
%         \includegraphics[scale=\figurescaling]{images/comp-euler-1D/sod-shock-tube/N=7_deriv_order=6_deriv_type=GlaubitzEtal2024(-0.1)_nblocks=64_LwTp9jMbsIe/entr-Thermodynamic Entropy.pdf}
%     \end{subcaptionblock}
%     \caption{Plots of the total thermodynamic entropy for the 7th order DP FD methods on 12 blocks with 33 nodes per block (left) and DP DG methods using 6th order polynomials on 64 elements.}
%     \label{fig:ce-sod-shock-entropy}
% \end{figure}
\subsubsection{Isentropic Vortex}
To verify accuracy, we consider the isentropic vortex test problem~\cite{Shu1998} modeled by the 2D compressible Euler equation. Specifically, we use the initial conditions as given by~\cite{ranocha2023highorder}
{\small
\begin{align*}
&\varrho_0 = \bar{\varrho}_0(T/\bar{T}_0)^{1/(\gamma-1)},\quad
    \begin{bmatrix}
        u_0 \\ v_0
    \end{bmatrix} = \begin{bmatrix}\bar{u}_0\\\bar{v}_0\end{bmatrix} + \frac{\varepsilon}{2\pi}\exp((1-r^2)/2)\begin{bmatrix}
        -y \\ x
    \end{bmatrix}, \\
   & p_0 = \bar{T}_0\bar{\varrho}_0, \quad T = \bar{T}_0-\frac{(\gamma-1)\varepsilon^2}{8\gamma \pi^2}\exp(1-r^2),
\end{align*}    
}%
% \vspace{-0.5cm}
% %%%
% \noindent
where $T$
% \[
%     T = \bar{T}_0-\frac{(\gamma-1)\varepsilon^2}{8\gamma \pi^2}\exp(1-r^2),
% \]
is the temperature, $\varepsilon = 10$ is the vortex strength, $r=\sqrt{x^2 + y^2}$, $\bar{\varrho}_0 = 1$ is the background density, $ \bar{u}_0 = \bar{v}_0 = 1$ is the background velocity, $p_0 = 10$ is the background pressure, $\gamma = 7/5$, and $\bar{T}_0 = {p}_0/\bar{\varrho}_0$ the background temperature, with the spatial domain $[-8,8]^2$ and periodic boundary conditions.
We  run the simulation on a sequence of varying discretization parameters  until the final time $t=16$, with time-step of $\Delta t = 0.04 \Delta x$. The $l^2$ errors and the convergence rates of the errors are shown in \cref{fig:ce-isentropic-vortex-convergence}. Note that the numerical convergence rates are slightly better than the theoretical rates. 
% Timestepping:
%     t ∈ [0, 16]
%     5-stage 4-th order fixed timestep SSPRK
%     Δt = 10⁻² × Δx where Δx = minᵢ (xᵢ₊₁ - xᵢ)

\begin{figure}[htbp]
    \centering
    \begin{subcaptionblock}{0.47\textwidth}
        \centering
        \includegraphics[scale=\figurescaling]{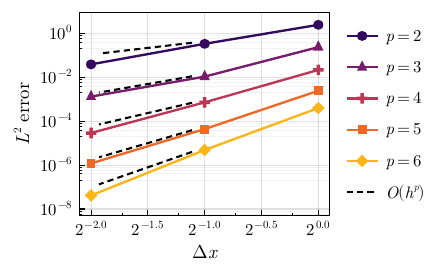}
        \caption{DP DG method.}
    \end{subcaptionblock}\hfill
    \begin{subcaptionblock}{0.52\textwidth}
        \centering
        \includegraphics[scale=\figurescaling]{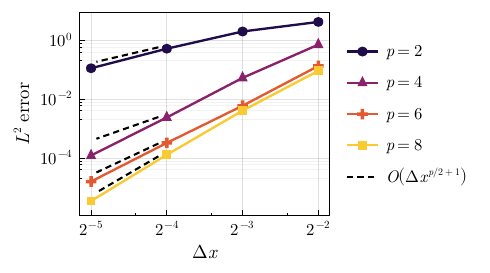}
        \caption{4 element DP FD method.}
    \end{subcaptionblock}
    %\vspace{-0.5cm}
    \caption{Numerical errors and convergence of errors for the  isentropic vortex problem.}
    \label{fig:ce-isentropic-vortex-convergence}
   % \vspace{-0.5cm}
\end{figure}
% \begin{table}[htb]
%     \begin{subcaptionblock}{\textwidth}
%         \centering
%         \scalebox{\tablescaling}{
%         \input{data/comp-euler-2D/isentropic-vortex/convergence-deriv_order=4-LlfdvcsfQOb}
%         }
%         \caption{4-th order accurate}
%     \end{subcaptionblock}
%     \begin{subcaptionblock}{\textwidth}
%         \centering
%         \scalebox{\tablescaling}{
%         \input{data/comp-euler-2D/isentropic-vortex/convergence-deriv_order=5-GPdA3UeWYNN}
%         }
%         \caption{5-th order accurate}
%     \end{subcaptionblock}
%     \begin{subcaptionblock}{\textwidth}
%         \centering
%         \scalebox{\tablescaling}{
%         \input{data/comp-euler-2D/isentropic-vortex/convergence-deriv_order=6-FZ4i2F0HmAb}
%         }
%         \caption{6-th order accurate}
%     \end{subcaptionblock}
%     \begin{subcaptionblock}{\textwidth}
%         \centering
%         \scalebox{\tablescaling}{
%         \input{data/comp-euler-2D/isentropic-vortex/convergence-deriv_order=7-HrvsLax3tIC}
%         }
%         \caption{7-th order accurate}
%     \end{subcaptionblock}
%     \caption{}
%     \label{tab:ce-isentropic-vortex-fd}
% \end{table}

\subsubsection{Kelvin-Helmholtz Instability}
To investigate the robustness of our entropy stable DP DG/FD numerical framework for  under-resolved inviscid turbulent flows, we simulate the Kelvin-Helmholtz instability for the compressible Euler equations in 2D. 
The domain is \([-1,1]^2\) with a
final time \(t = 15\). We consider the initial conditions
\begin{equation} \label{eq:kh-init}
\begin{split}
	&\varrho(\vec x)
		= \varrho_\text{lo}
		+ (\varrho_\text{hi} - \varrho_\text{lo}) B(x_2),
	\qquad
	\vec v(\vec x) =
	\begin{bmatrix}
		B(x_2) - 0.5 \\
		\frac{1}{10} \sin(2\pi x_1) \\
	\end{bmatrix},
	\qquad
	p(\vec x) = 1,
    \\
    &
    B(x_2) = \frac{\tanh(15(x_2 + 0.5)) - \tanh(15(x_2 - 0.5))}{2},
	\quad
	\varrho_\text{lo} = 0.5,
	\quad
	\varrho_\text{hi} = \frac{1 + A}{1 - A} \varrho_\text{lo},
    \end{split}
\end{equation}
% % \vspace{-0.5cm}
% % %%%
%
% % \noindent
% % where
% \begin{align}
%     B(x_2) = \frac{\tanh(15(x_2 + 0.5)) - \tanh(15(x_2 - 0.5))}{2},
% 	\quad
% 	\varrho_\text{lo} = 0.5,
% 	\quad
% 	\varrho_\text{hi} = \frac{1 + A}{1 - A} \varrho_\text{lo},
% \end{align}
% \vspace{-0.25cm}
% %%%
% \noindent
where \(\gamma = 7/5\) is the adiabatic index and \(A \in [0, 1)\) is the Atwood number \cite{BANSAL2024134276,ChanRanochaHendrik_et_al2022}, which defines the ratio of the background densities $\varrho_\text{hi} / \varrho_\text{lo}$.
Note that $\varrho_\text{hi} / \varrho_\text{lo}\to \infty$ as $A\to 1$, which triggers strong vorticity, faster growth of mixing layers, and a rapid transition to turbulent mixing, making the problem more challenging for numerical methods, as shown in \cite{ChanRanochaHendrik_et_al2022}.

To begin, we consider the standard setting \cite{ranocha2023highorder,GLAUBITZ2025113841},  with a modest Atwood number $A = 0.6$. For the DG scheme we vary the number of elements $16\le K\le 64$ and  the degree of polynomial approximation $3\le p\le 6$. We also  consider DP FD difference operators of interior order of accuracy $6, 7, 8, 9$  discretized with $8$ elements and vary the number of nodes $17\le n\le 65$ within an element.
We use the explicit and third-order accurate adaptive time-stepping method~\cite{Kraaijevanger1991} with relative tolerance $\mathrm{tol}_\mathrm{rel} = 10^{-6}$, absolute tolerance $\mathrm{tol}_\mathrm{abs} = 10^{-7}$ and initial timestep of \(\Delta t = 0.01 \Delta x\). Snapshots of the mass density $\rho$ on a given mesh are shown in \cref{fig:snapshots-comp-euler-kh} for DP DG, FD, and DRP operators. 
% The initial conditions are
% Initial Conditions:
% {
% \small
% \begin{equation*}
%     \varrho_0 = \frac{1}{2} + \frac{3}{2} B(x, y),\quad
%     u_0 = B(x, y) - \frac{1}{2},\quad
%     v_0 = \frac{1}{10} \sin(2\pi x),\quad
%     p_0 = 1, \quad B(x, y) = \frac{\tanh(15y + 7.5) - \tanh(15y - 7.5)}{2}
% \end{equation*}
% }
% where
% \[
%     B(x, y) = \frac{\tanh(15y + 7.5) - \tanh(15y - 7.5)}{2}.
% \]

% spatial domain: [-1,1]^2

%   SSPRK43 adaptive time stepping with
%       absolute tolerance: 1e-6
%       relative tolerance: 1e-6
%       min timestep: 5e-6 × Δx
%   final time: 15
\begin{table}[htbp]
    \centering
    \scalebox{\tablescaling}{\begin{tabular}{@{}r@{}*{4}{c@{}*{4}{r}@{}}}
    \toprule
    & \hspace{2em}
    & \multicolumn{4}{c}{Lax-Friedrichs} & \hspace{2.5em}
    & \multicolumn{4}{c}{van Leer-H\"anel} & \hspace{2.5em}
    & \multicolumn{4}{c}{SBP FD} & \hspace{2.5em}
    & \multicolumn{4}{c}{DP FD} \\
	\cmidrule{3-6}
	\cmidrule{8-11}
	\cmidrule{13-16}
	\cmidrule{18-21}
	&
    & \multicolumn{4}{c}{accuracy order} &
    & \multicolumn{4}{c}{accuracy order} &
    & \multicolumn{4}{c}{accuracy order} &
    & \multicolumn{4}{c}{accuracy order} \\
	n &
    & 6 & 7 & 8 & 9 &
    & 6 & 7 & 8 & 9 &
    & 6 & 7 & 8 & 9 &
    & 6 & 7 & 8 & 9 \\
	\midrule
	17 &
    & % LxF
    \crash{3.82} & \crash{3.77} & \crash{2.47} & \crash{2.45} &
    & % vLH
    \crash{3.72} & \crash{3.68} & \crash{2.40} & \crash{1.94} &
	& % Ns
    \crash{4.02} & \crash{3.44} & \crash{2.75} & \crash{2.96} &
	& % UNs
    \nocrash{15} & \nocrash{15} & \nocrash{15} & \nocrash{15} \\
	33 &
    & % LxF
    \crash{4.31} & \crash{4.25} & \crash{3.67} & \crash{3.68} &
    & % vLH
    \crash{4.35} & \crash{4.01} & \crash{3.61} & \crash{3.62} &
	& % Ns
    \crash{3.75} & \crash{3.10} & \crash{3.39} & \crash{3.47} &
	& % UNs
    \nocrash{15} & \nocrash{15} & \nocrash{15} & \nocrash{15} \\
	65 &
    & % LxF
    \crash{4.41} & \crash{4.76} & \crash{3.66} & \crash{3.65} &
    & % vLH
    \crash{3.87} & \crash{3.85} & \crash{3.82} & \crash{3.67} &
	& % Ns
    \crash{3.80} & \crash{3.83} & \crash{3.45} & \crash{3.45} &
	& % UNs
    \nocrash{15} & \nocrash{15} & \nocrash{15} & \nocrash{15} \\
	\bottomrule
\end{tabular}}
    \caption{
         Crash or final simulation times for the \(8^2\) element FD approximations of the Kelvin-Helmholtz instability with the Atwood number $A =0.6$.
    }
    \label{tab:ce-kh-crash-times-FD}
   % \vspace{-0.5cm}
\end{table}
\begin{table}[htbp]
    \centering
    \scalebox{\tablescaling}{\begin{tabular}{@{}r@{}*{4}{c@{}*{4}{r}@{}}}
    \toprule
    & \hspace{2em}
    & \multicolumn{4}{c}{Lax-Friedrichs} &\hspace{2.5em}
    & \multicolumn{4}{c}{van Leer-H\"anel} &\hspace{2.5em}
    & \multicolumn{4}{c}{DGSEM} &\hspace{2.5em}
    & \multicolumn{4}{c}{DP DG} \\
    \cmidrule{3-6}
    \cmidrule{8-11}
    \cmidrule{13-16}
    \cmidrule{18-21}
    &
    & \multicolumn{4}{c}{polynomial degree} &
    & \multicolumn{4}{c}{polynomial degree} &
    & \multicolumn{4}{c}{polynomial degree} &
    & \multicolumn{4}{c}{polynomial degree} \\
    K &
    & 3 & 4 & 5 & 6 &
    & 3 & 4 & 5 & 6 &
    & 3 & 4 & 5 & 6 &
    & 3 & 4 & 5 & 6 \\
    \midrule
    16 &
    % LxF
    & \crash{2.39} & \crash{3.57} & \crash{2.93} & \crash{3.06} &
    % vLH
    & \crash{4.83} & \crash{1.59} & \crash{3.24} & \crash{1.87} &
    % Ns
    & \crash{2.68} & \crash{2.53} & \crash{3.03} & \crash{2.72} &
    % UNs
    & \nocrash{15} & \nocrash{15} & \nocrash{15} & \nocrash{15} \\
    32 &
    % LxF
    & \crash{3.37} & \crash{3.81} & \crash{3.54} & \crash{3.48} &
    % vLH
    & \crash{3.68} & \crash{3.60} & \crash{3.54} & \crash{3.58} &
    % Ns
    & \crash{3.26} & \crash{2.87} & \crash{3.23} & \crash{3.17} &
    % UNs
    & \nocrash{15} & \nocrash{15} & \nocrash{15} & \nocrash{15} \\
    64 &
    % LxF
    & \crash{3.59} & \crash{3.53} & \crash{3.49} & \crash{3.50} &
    % vLH
    & \crash{3.57} & \crash{3.53} & \crash{3.47} & \crash{3.28} &
    % Ns
    & \crash{3.27} & \crash{3.34} & \crash{3.38} & \crash{3.28} &
    % UNs
    & \nocrash{15} & \nocrash{15} & \nocrash{15} & \nocrash{15} \\
    \bottomrule
\end{tabular}}
    \caption{
       Crash or final simulation times for the DG approximations of the Kelvin-Helmholtz instability with the Atwood number $A =0.6$.
    }
    \label{tab:ce-kh-crash-times-DG}
    %\vspace{-0.5cm}
\end{table}
\begin{figure}[htbp!]
    \centering
    \footnotesize
    \begin{tabular}{@{}m{1em}*3{>{\centering\arraybackslash}m{0.25\linewidth}@{}}c@{}}
        &DP DG & DP FD & DRP DP FD \\
		\rotatebox{90}{\(t=5\)} &
        \adjincludegraphics[width=\linewidth, trim = {1.0cm, 0.5\height, 1.0cm, 1.0cm}, clip]
            {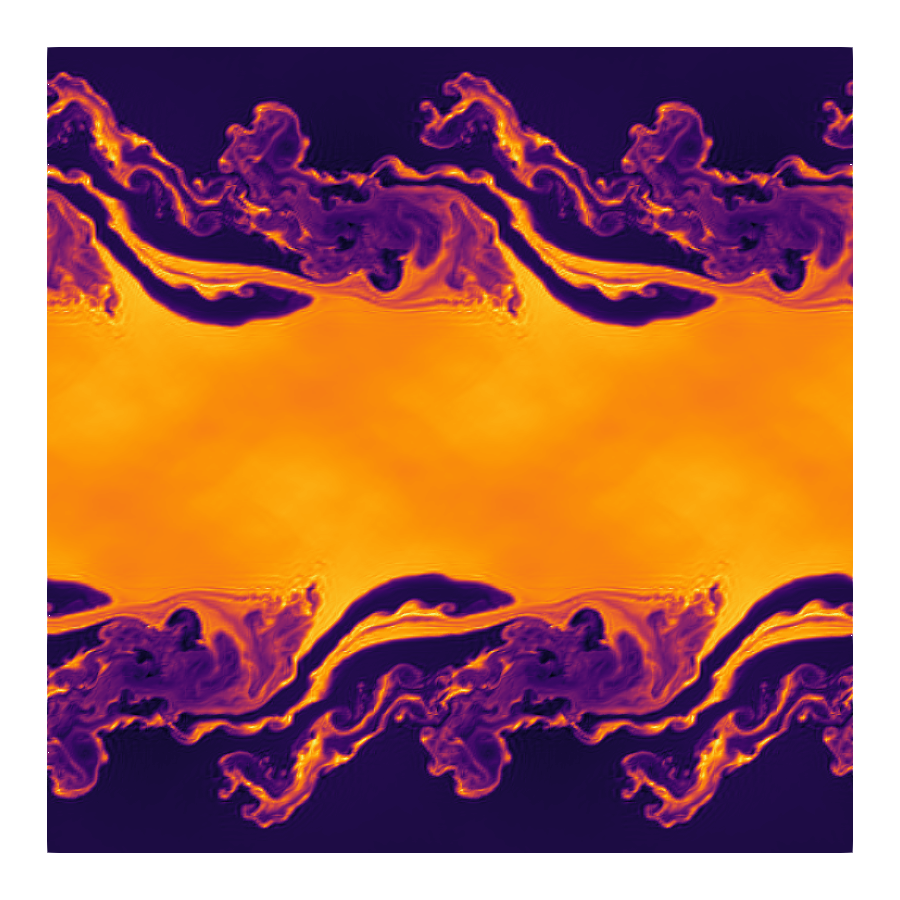} &
        \adjincludegraphics[width=\linewidth, trim = {1.0cm, 0.5\height, 1.0cm, 1.0cm}, clip]
            {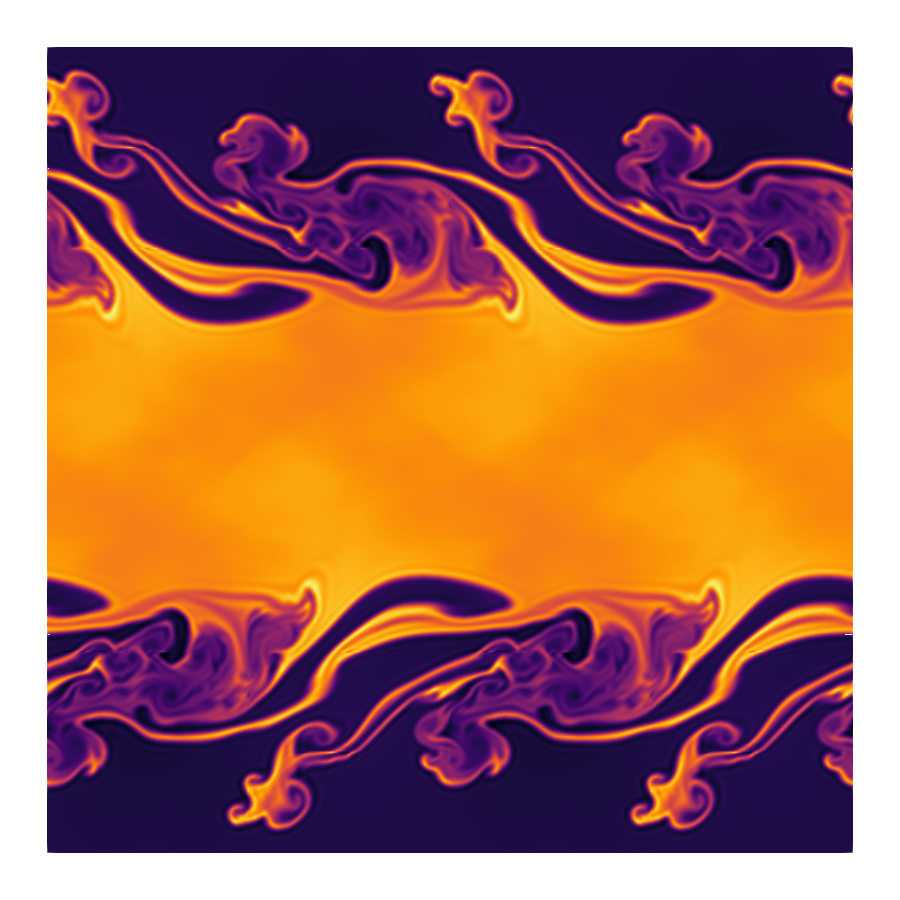} &
        \adjincludegraphics[width=\linewidth, trim = {1.0cm, 0.5\height, 1.0cm, 1.0cm}, clip]
            {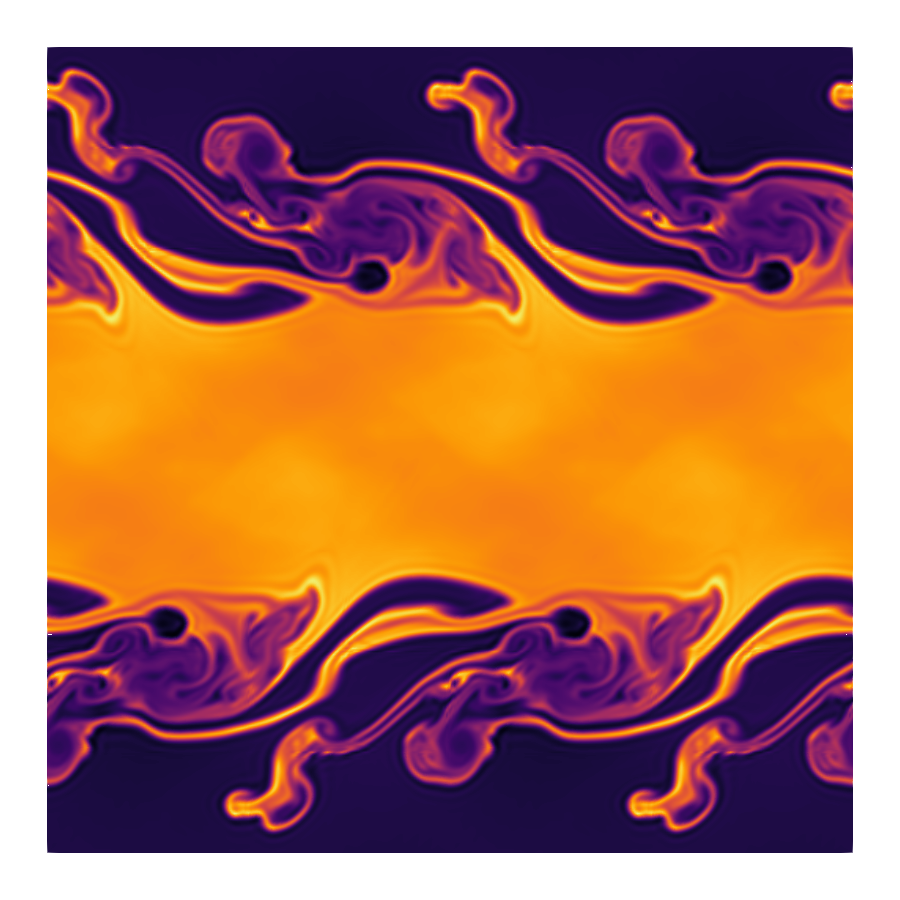}
        & \multirow{12}{*}{\includegraphics[width=2cm, trim={14cm, 9cm, 10.5cm, 9cm}, clip]{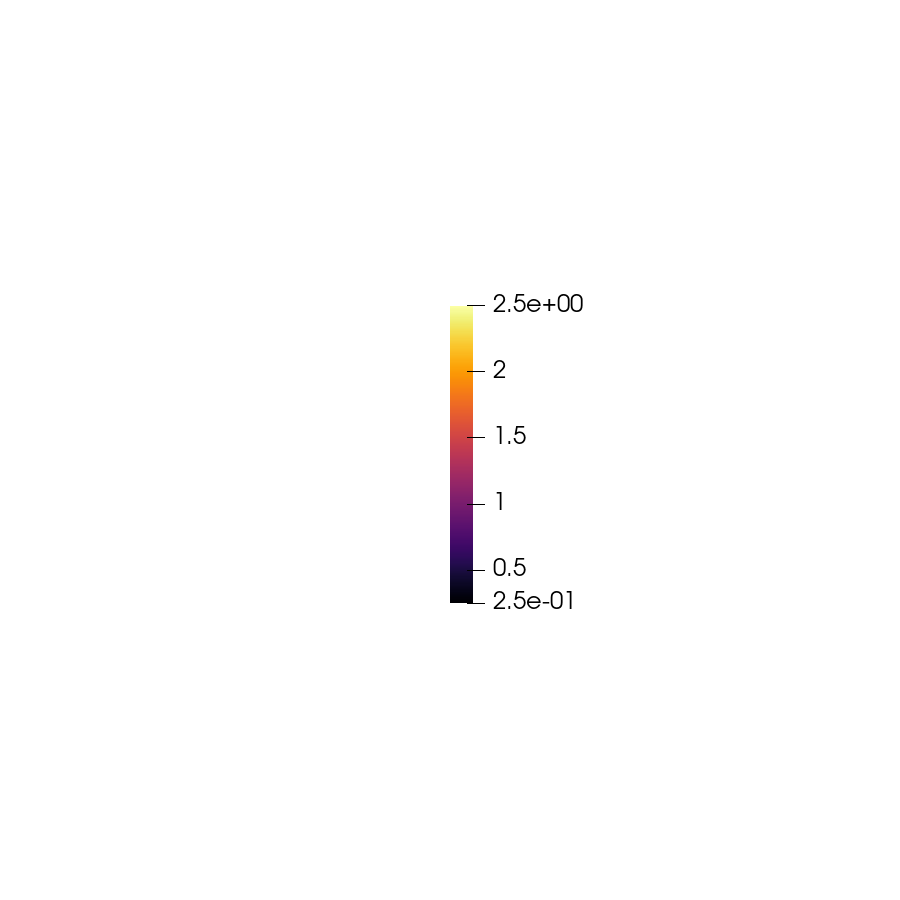}}
        \\[-4pt]
		\rotatebox{90}{\(t=6\)} &
        \adjincludegraphics[width=\linewidth, trim = {1.0cm, 0.5\height, 1.0cm, 1.0cm}, clip]
            {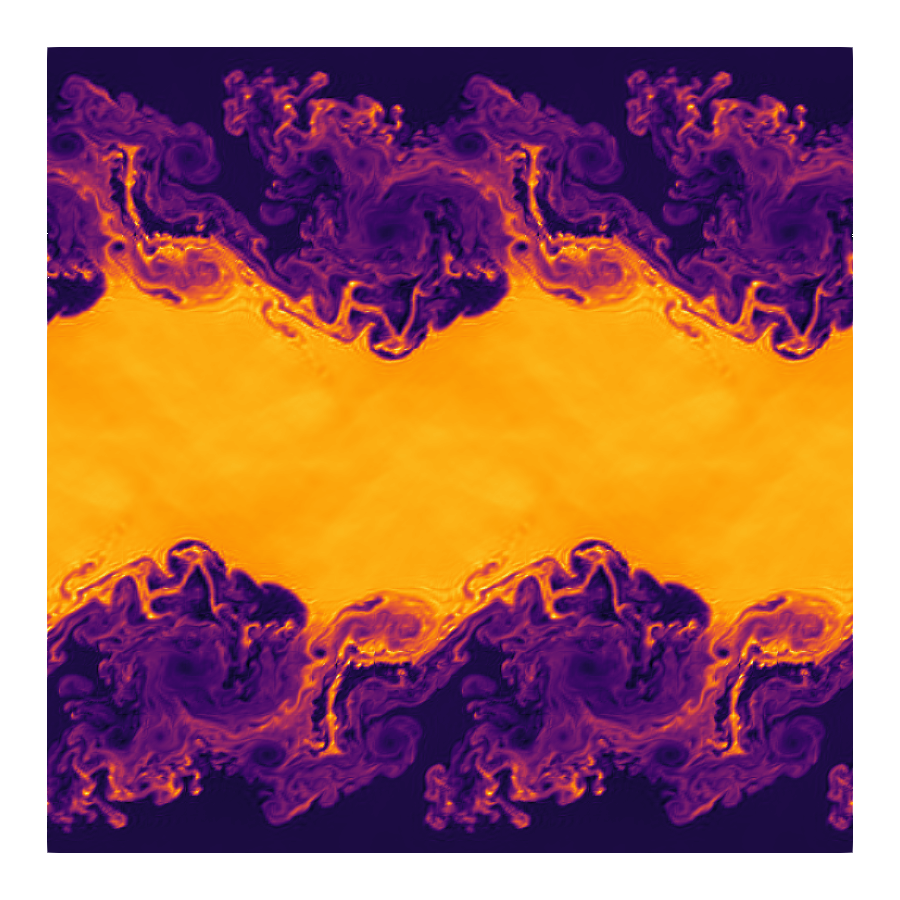} &
        \adjincludegraphics[width=\linewidth, trim = {1.0cm, 0.5\height, 1.0cm, 1.0cm}, clip]
            {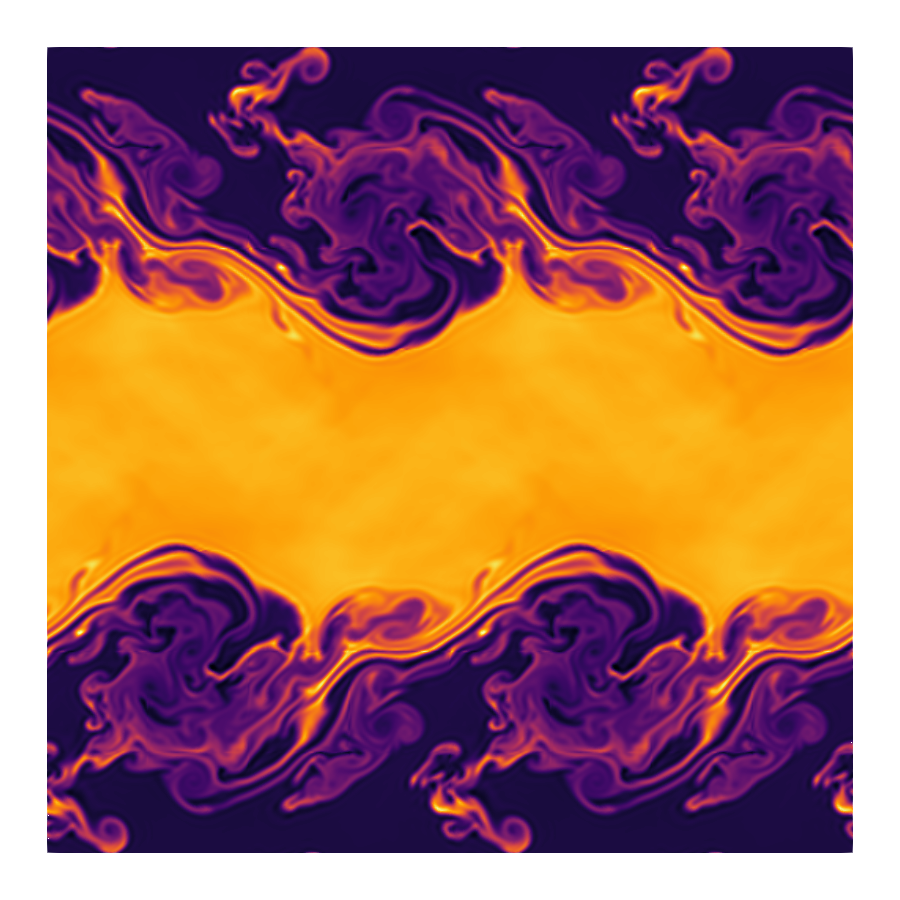} &
        \adjincludegraphics[width=\linewidth, trim = {1.0cm, 0.5\height, 1.0cm, 1.0cm}, clip]
            {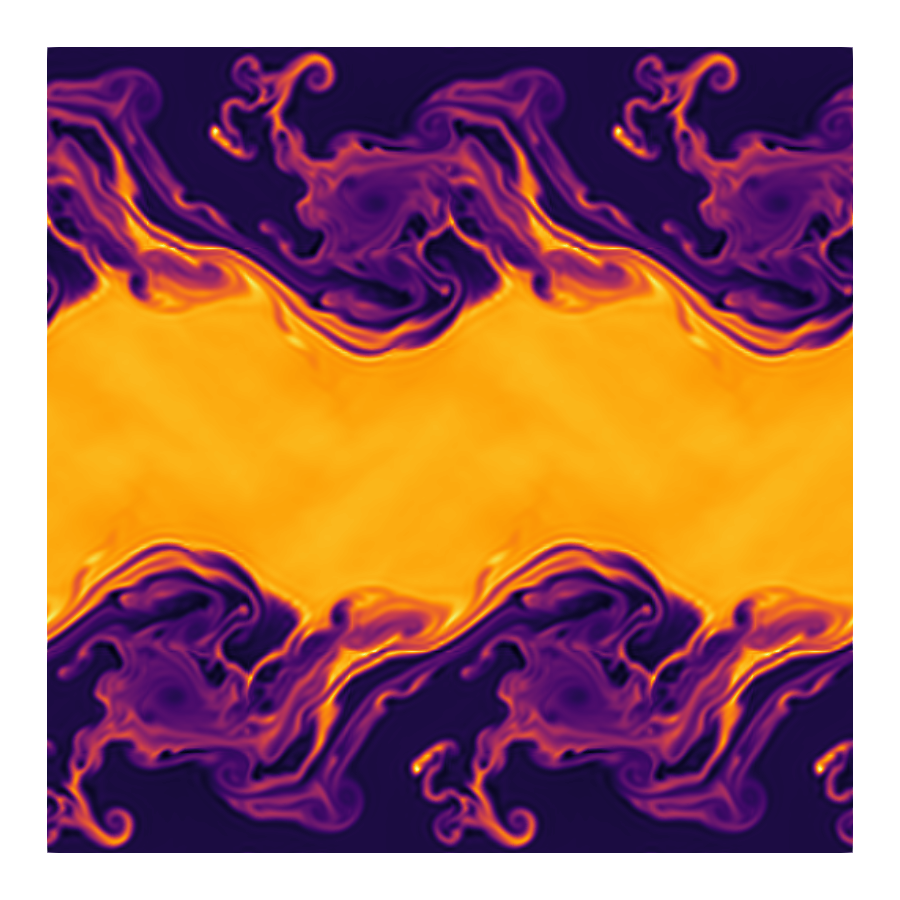}
        \\[-4pt]
		\rotatebox{90}{\(t=10\)} &
        \adjincludegraphics[width=\linewidth, trim = {1.0cm, 0.5\height, 1.0cm, 1.0cm}, clip]
            {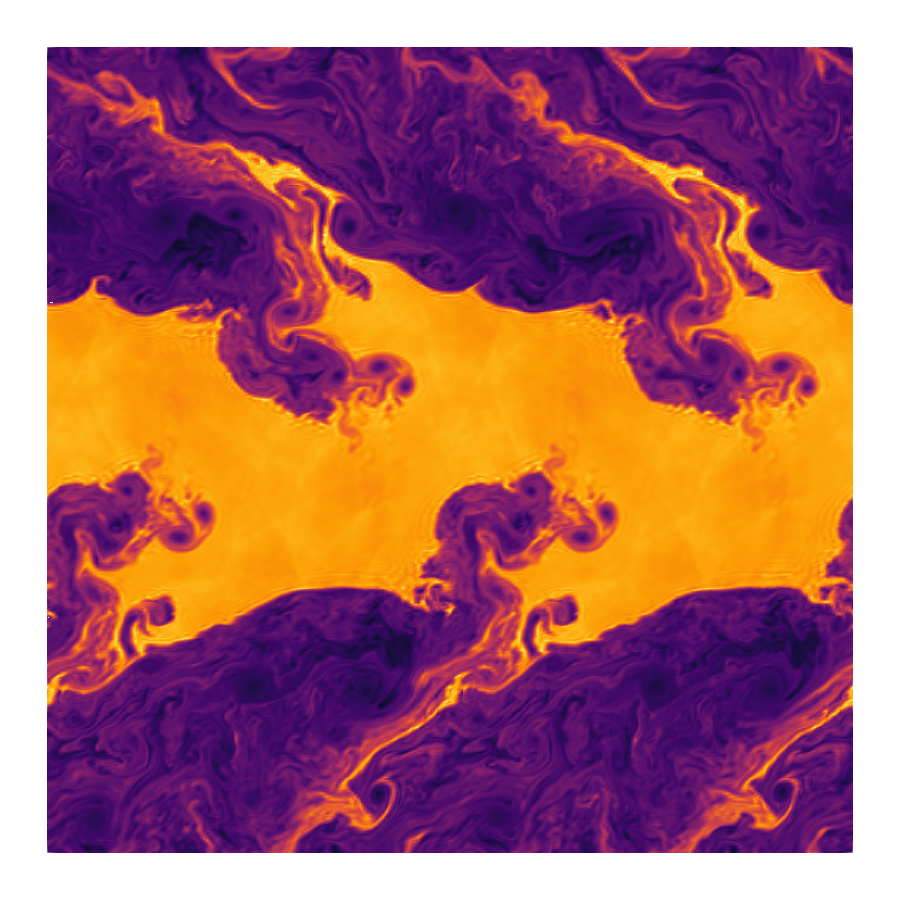} &
        \adjincludegraphics[width=\linewidth, trim = {1.0cm, 0.5\height, 1.0cm, 1.0cm}, clip]
            {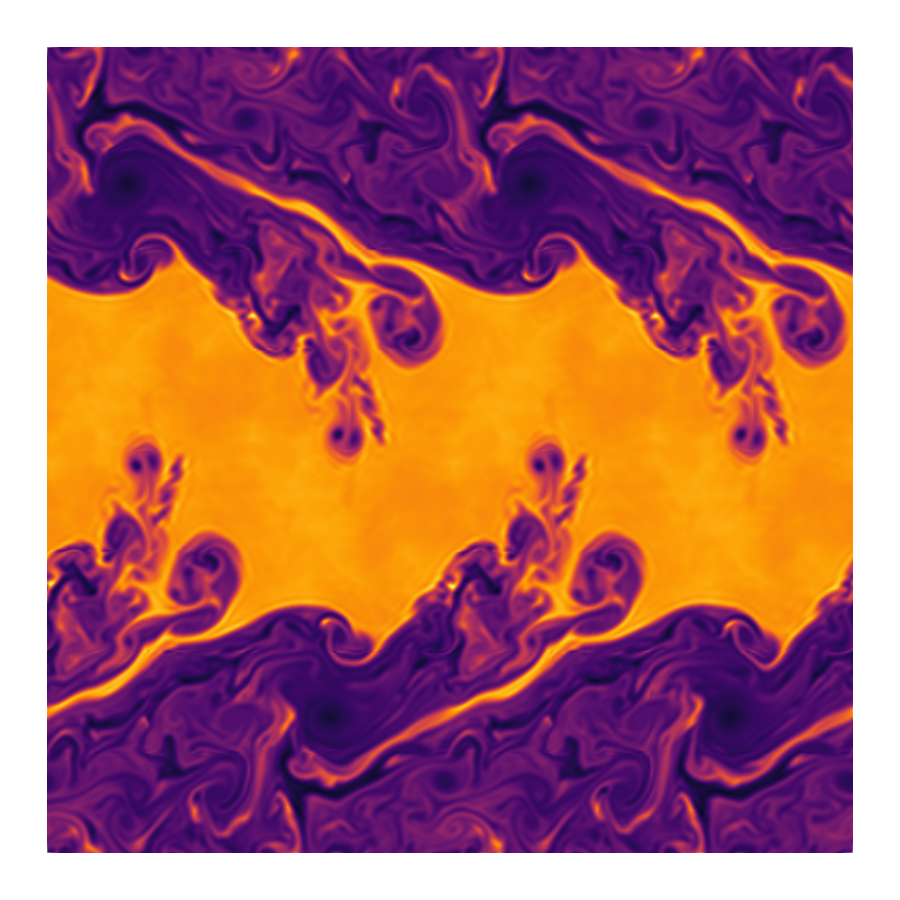} &
        \adjincludegraphics[width=\linewidth, trim = {1.0cm, 0.5\height, 1.0cm, 1.0cm}, clip]
            {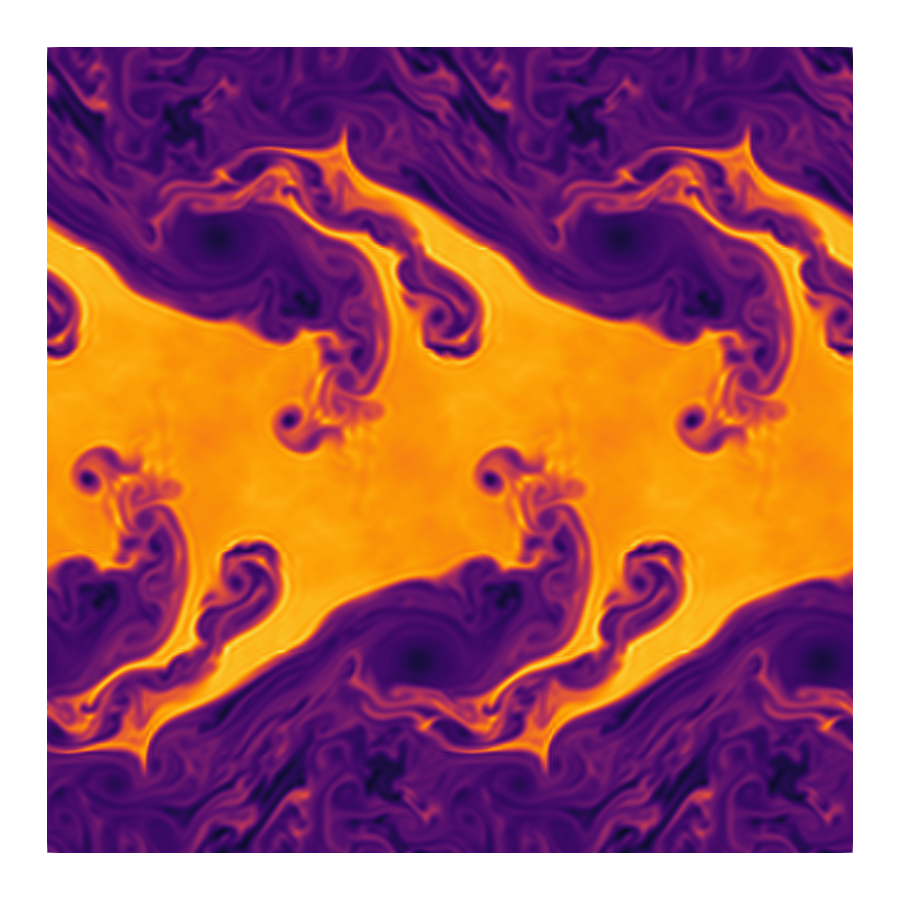}
        \\
        % &\multicolumn{3}{c}{\includegraphics[width=5cm]{images/comp-euler-2D/kelvin-helmholtz/density/colorbar.png}}
    \end{tabular}
	\caption{Snapshots of the density for the DP DG scheme on \(64^2\) elements and degree 6 polynomials per element and the 6th order DP FD scheme on \(8^2\) elements with \(65^2\) nodes per element.}
    \label{fig:snapshots-comp-euler-kh}
    %\vspace{-0.5cm}
\end{figure}
%%
% \vspace{-0.2cm}
%
% \noindent
In \cref{fig:ce-kh-conservation} we have plotted the total entropy and the relative changes in total mass, total momentum, total energy. As above, the DP SBP FD/DG schemes (with volume upwinding $\Gamma>0$) are conservative and entropy dissipative.  In addition, our DP SBP FD/DG schemes (with volume upwinding $\Gamma>0$) are stable and completed the simulation until the final time, $t=15$, without crashing.  Meanwhile all other schemes, including the linearly stable schemes (with volume upwinding $\Gamma>0$) and the standard entropy stable SBP/DGSEM (without volume upwinding $\Gamma=0$) schemes crashed before the final simulation time.
\begin{figure}[htbp]
    \centering
    \includegraphics[scale=\figurescaling]{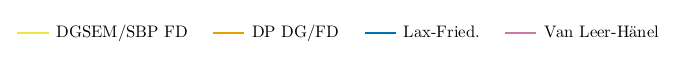} \\[-1em]
    \begin{subcaptionblock}{\textwidth}
        \centering
        \includegraphics[scale=\figurescaling]{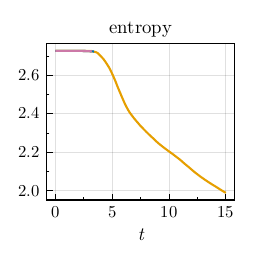}%
        \includegraphics[scale=\figurescaling]{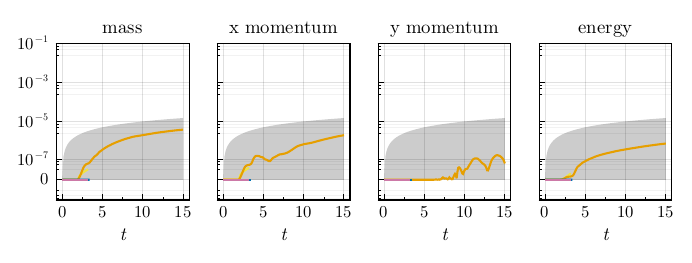}
        \\[-1em]
        \caption{DG methods using degree 6 polynomials on \(64^2\) elements.}
    \end{subcaptionblock}
    \begin{subcaptionblock}{\textwidth}
        \centering
        \includegraphics[scale=\figurescaling]{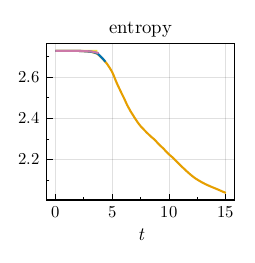}%
        \includegraphics[scale=\figurescaling]{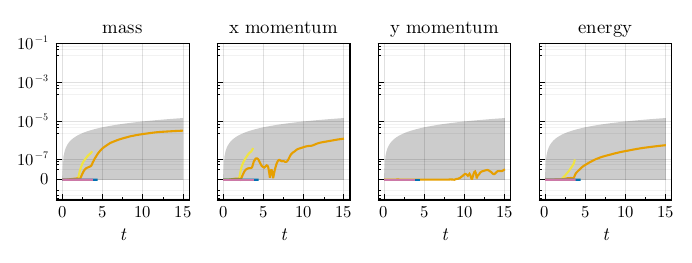}
        \\[-1em]
        \caption{6th order FD methods on \(8^2\) blocks with \(65^2\) nodes each.}
    \end{subcaptionblock}
    %\vspace{-0.5cm}
    \caption{The change over time in the total entropy and the relative change over time in total mass, momentum, and energy. The gray band shows a constant error growth of \(\qty{e-6}{s^{-1}}\), the relative tolerance used by the adaptive time stepper. The method is entropy dissipative and conservative up to time-stepping errors.}
    \label{fig:ce-kh-conservation}
   % \vspace{-0.5cm}
\end{figure}
%%%
%%
% \vspace{-0.1cm}
% %
% \noindent
%%%
We have also performed numerical simulations with different discretisation configurations. The crash or final  times are reported in  \cref{tab:ce-kh-crash-times-FD,tab:ce-kh-crash-times-DG}. Our entropy stable DP SBP FD/DG schemes (with $\Gamma >0 $) run until the final time, $t=15$, for all configurations without crashing. These reinforce the robustness and efficacy of  our entropy stable DP SBP FD/DG schemes (with $\Gamma>0$).
\begin{figure}[htbp]
    \centering
    \includegraphics[scale=\figurescaling]{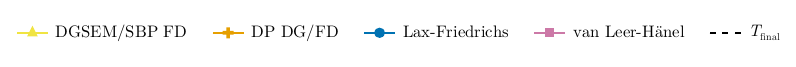} \\[-0.5em]
    \begin{subcaptionblock}{0.48\textwidth}
        \centering
        \includegraphics[scale=\figurescaling]{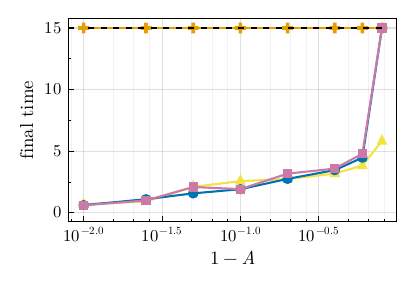}
        \\[-1em]
        \caption{DG methods using degree 6 polynomials on \(32^2\) elements.}
    \end{subcaptionblock}\hfill
    \begin{subcaptionblock}{0.48\textwidth}
        \centering
        \includegraphics[scale=\figurescaling]{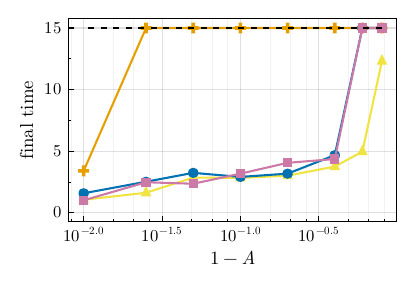}
        \\[-1em]
        \caption{6th order FD methods on \(8^2\) blocks with \(33^2\) nodes each.}
    \end{subcaptionblock}
    %\vspace{-0.5cm}
    \caption{End times of the DG (left) and FD methods (right) for the Kelvin-Helmholtz instability for Atwood numbers 0.2, 0.4, 0.6, 0.8, 0.9, 0.95, 0.975, and 0.99.}
    \label{fig:ce-kh-atwood-crash-times}
   % \vspace{-0.5cm}
\end{figure}
Furthermore, we have performed numerical experiments for the Atwood numbers $A \in \{0.2, 0.4, 0.6, 0.8, 0.9, 0.95, 0.975, 0.99\}$ and different discretizations. The snapshot of the mass density  for high Atwood numbers $A \in \{0.8, 0.9, 0.975, 0.99\}$ are shown in \cref{fig:kh-atwood-6} for the DP DG method using 6th order polynomials on \(32^2\)~elements. In \cref{fig:ce-kh-atwood-crash-times} we have plotted the crash or final times of the simulation for various schemes against the Atwood numbers. Note that while other schemes crashed,  our novel entropy stable DP FD/DG schemes are stable for high Atwood numbers. Note again that no artificial viscosity or ad hoc stabilization strategy is employed.  These  are strong testaments to the robustness of the numerical framework.
\begin{figure}[thbp]
	\centering
	\footnotesize
	\begin{tabular}{@{}m{1.0em}*4{>{\centering\arraybackslash}m{0.23\linewidth}@{}}}
		& \(A = 0.8\) &\(A = 0.9\) &\(A = 0.975\) &\(A = 0.99\) \\
		& \includegraphics[width=\linewidth, trim={8cm, 9cm, 8cm, 0cm}, clip]
			{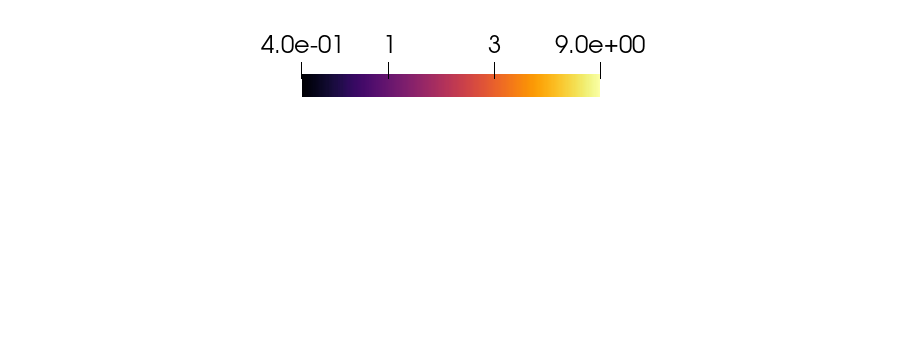}
		& \includegraphics[width=\linewidth, trim={8cm, 9cm, 8cm, 0cm}, clip]
			{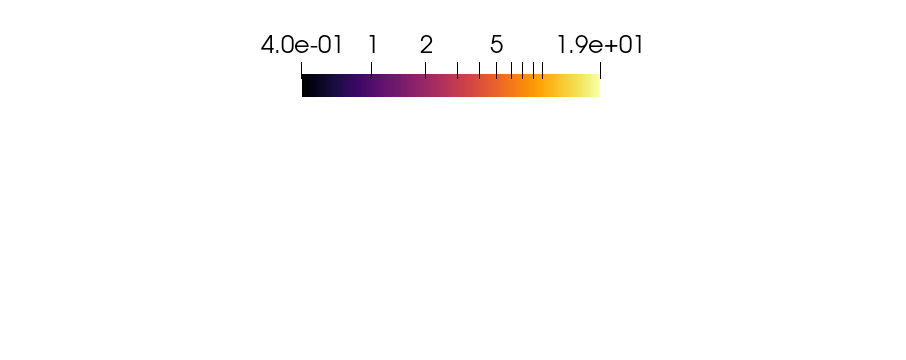}
		& \includegraphics[width=\linewidth, trim={8cm, 9cm, 8cm, 0cm}, clip]
			{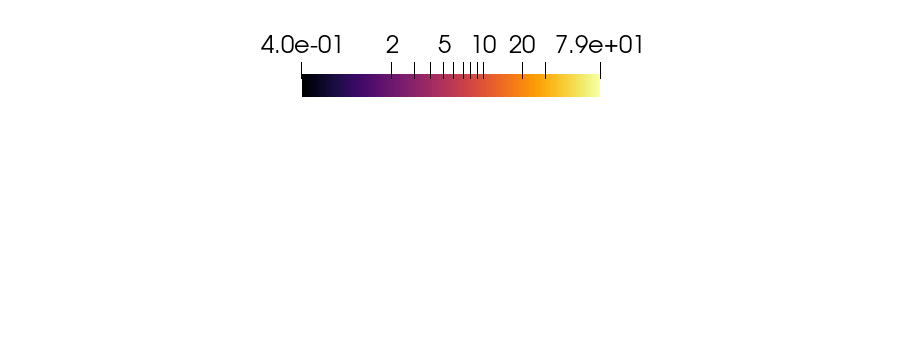}
		& \includegraphics[width=\linewidth, trim={8cm, 9cm, 8cm, 0cm}, clip]
			{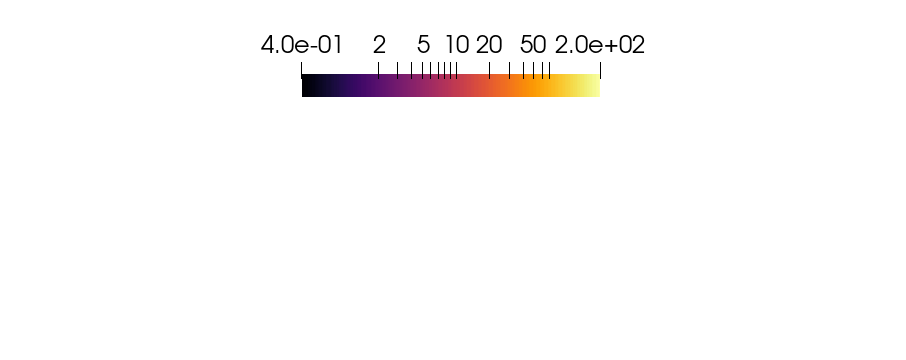}
        \\[-4pt]
		\rotatebox{90}{\(t=2\)}
		& \adjincludegraphics[width=\linewidth, trim={1cm, 0.5\height, 1cm, 1cm}, clip]
			{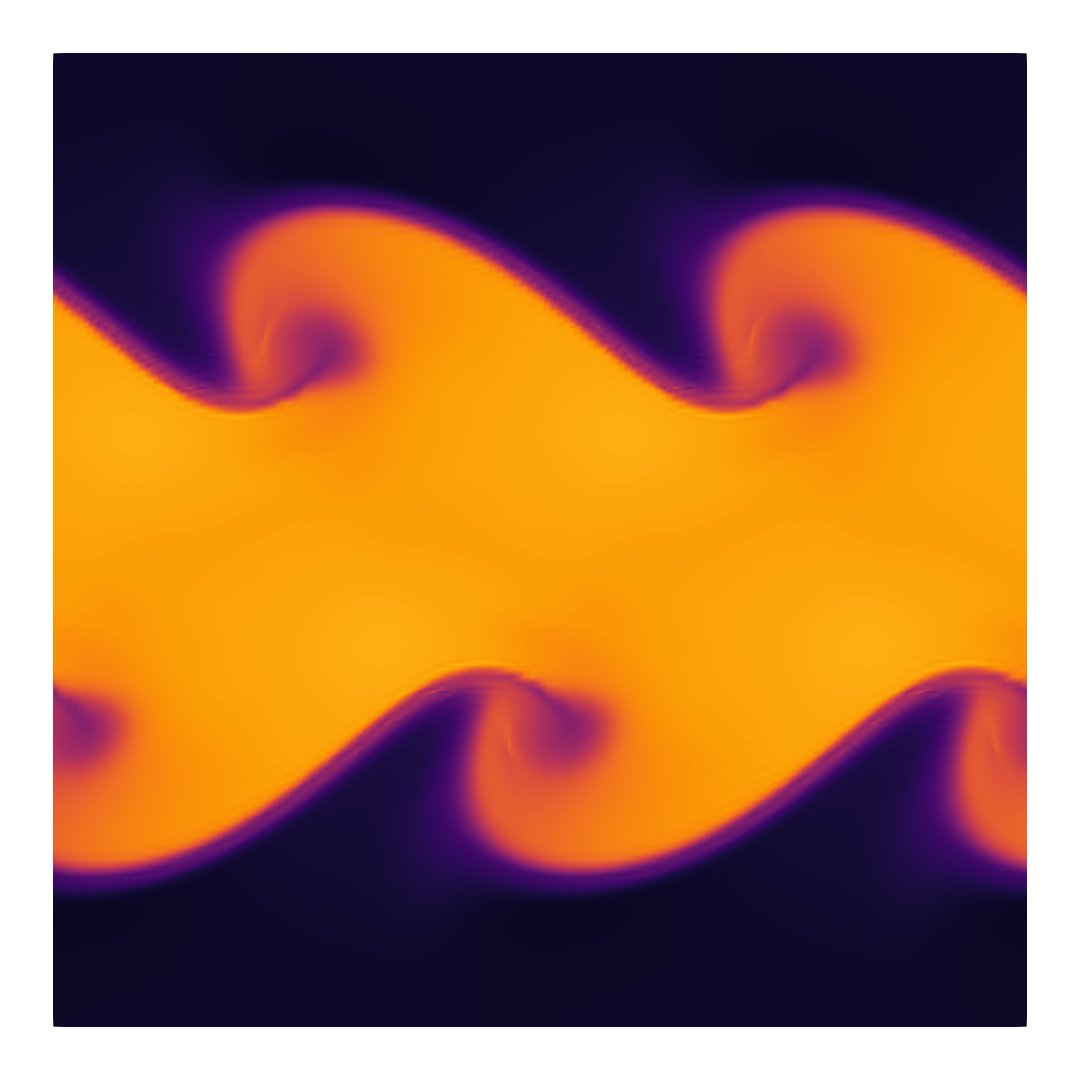}
		& \adjincludegraphics[width=\linewidth, trim={1cm, 0.5\height, 1cm, 1cm}, clip]
			{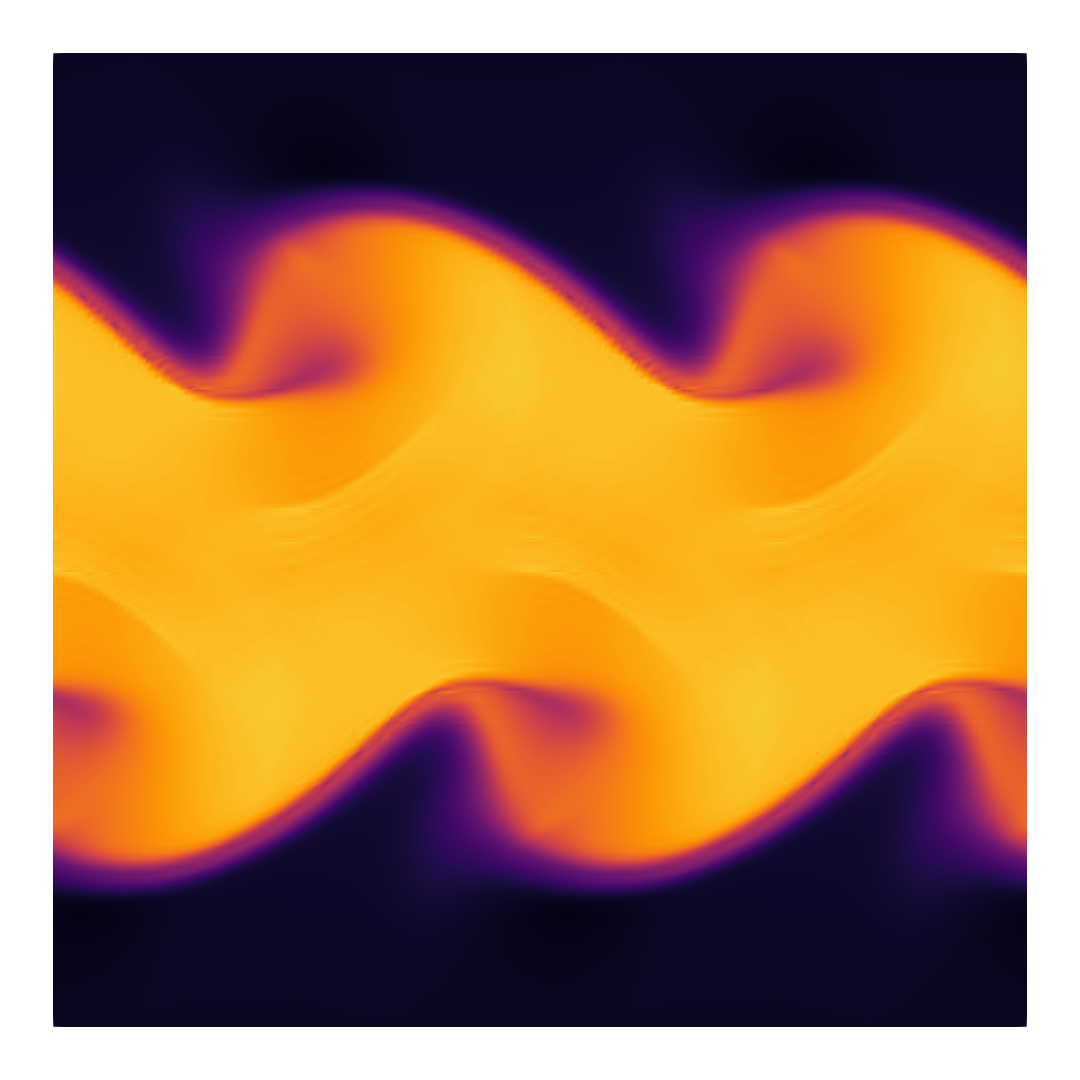}
		& \adjincludegraphics[width=\linewidth, trim={1cm, 0.5\height, 1cm, 1cm}, clip]
			{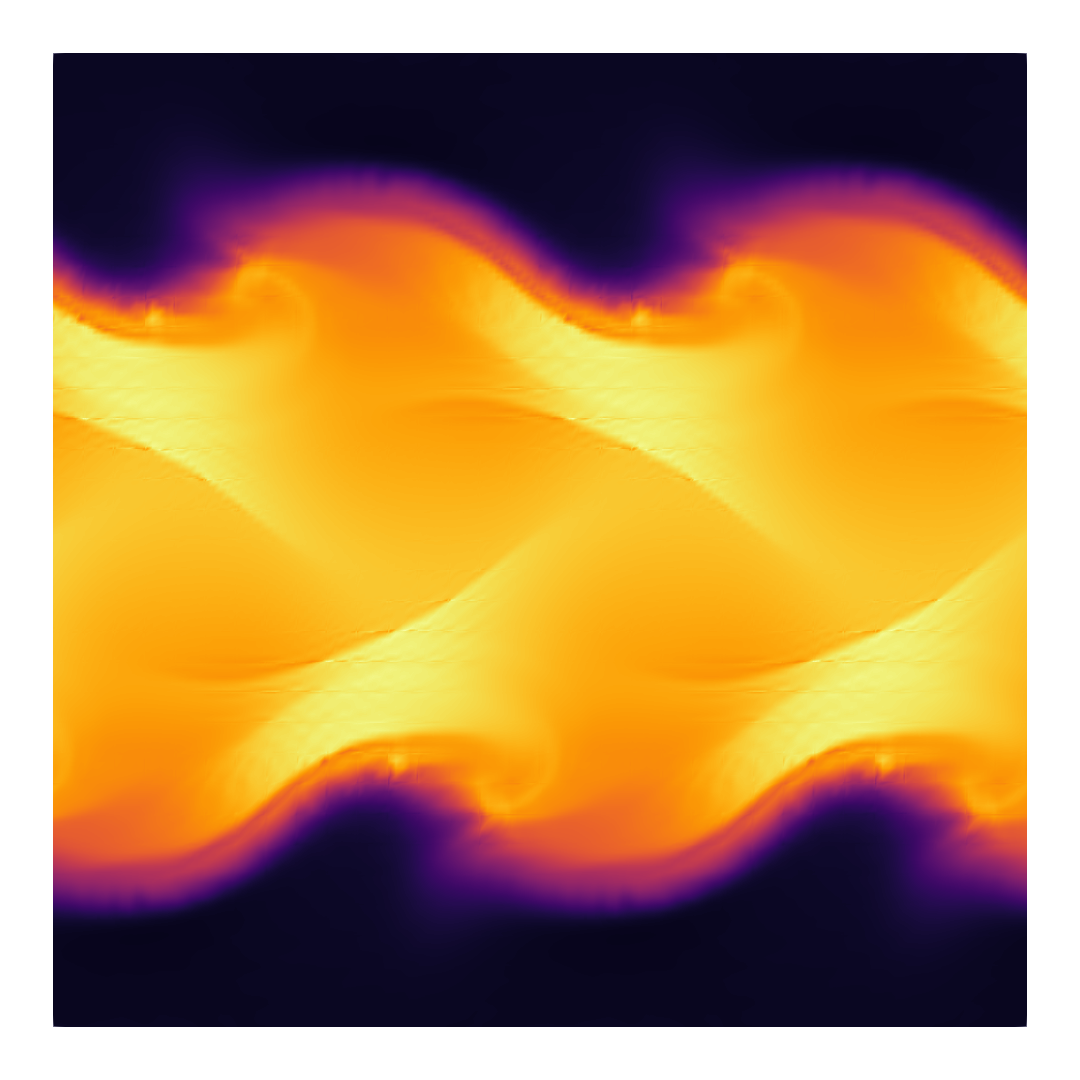}
		& \adjincludegraphics[width=\linewidth, trim={1cm, 0.5\height, 1cm, 1cm}, clip]
			{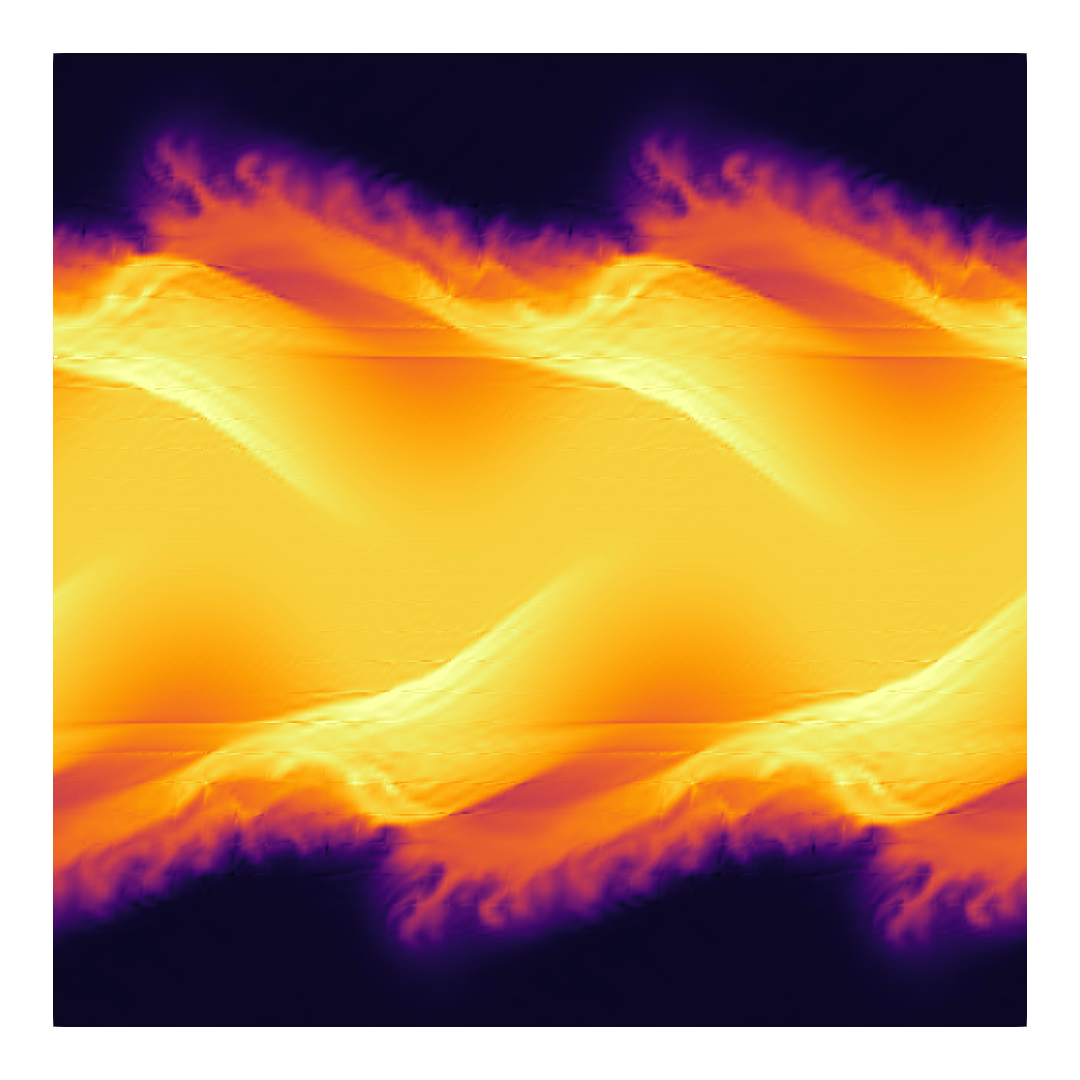}
        \\[-4pt]
		\rotatebox{90}{\(t=4\)}
		& \adjincludegraphics[width=\linewidth, trim={1cm, 0.5\height, 1cm, 1cm}, clip]
			{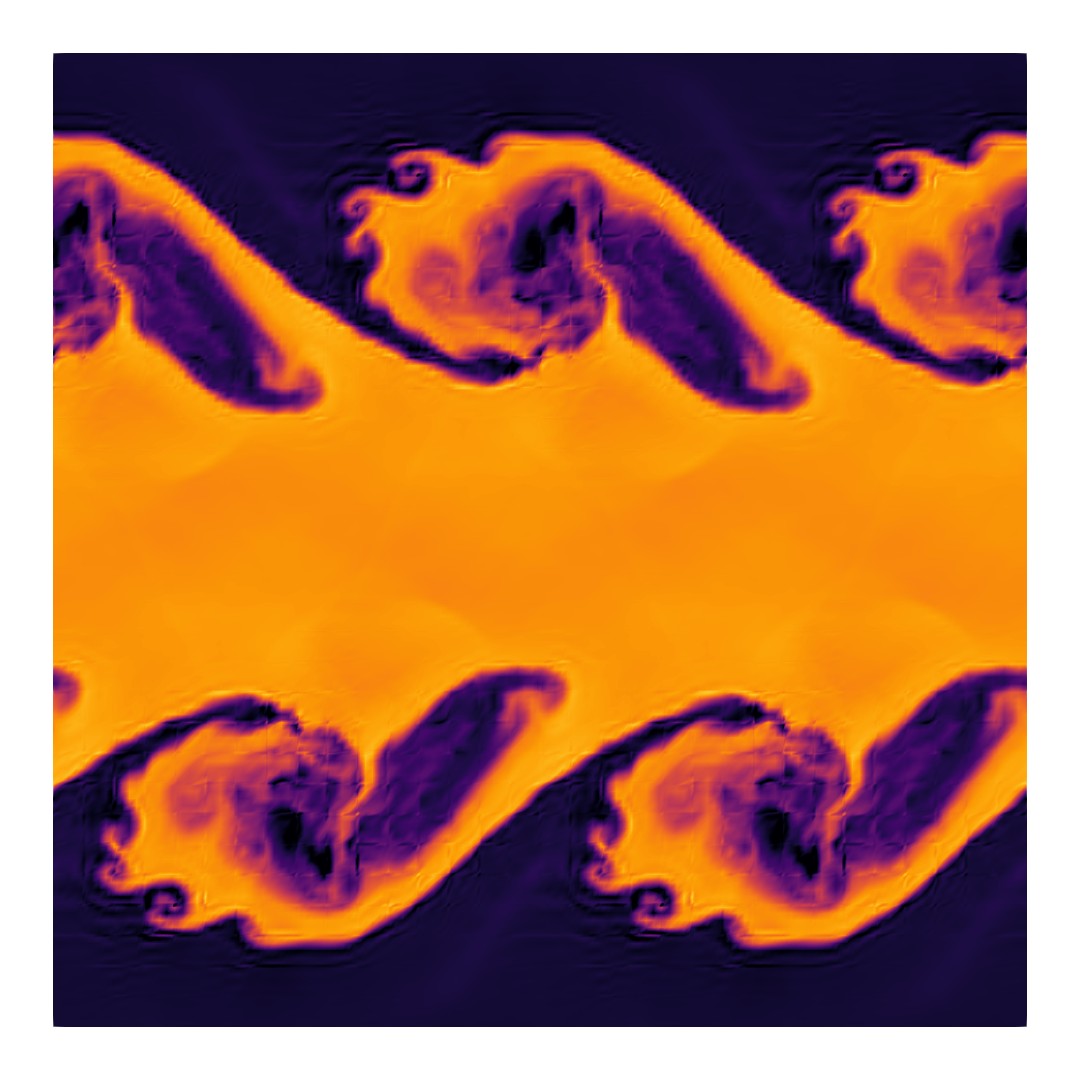}
		& \adjincludegraphics[width=\linewidth, trim={1cm, 0.5\height, 1cm, 1cm}, clip]
			{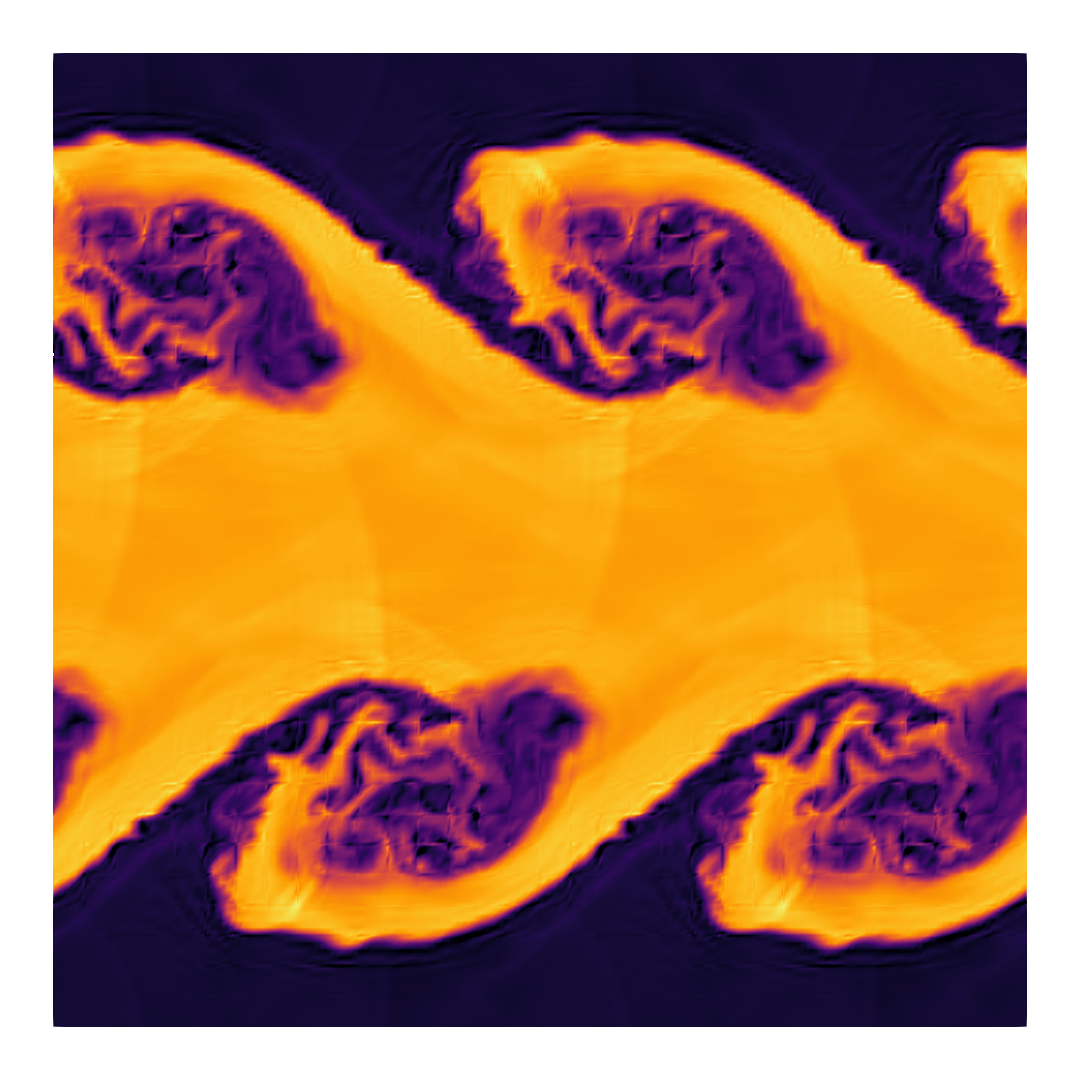}
		& \adjincludegraphics[width=\linewidth, trim={1cm, 0.5\height, 1cm, 1cm}, clip]
			{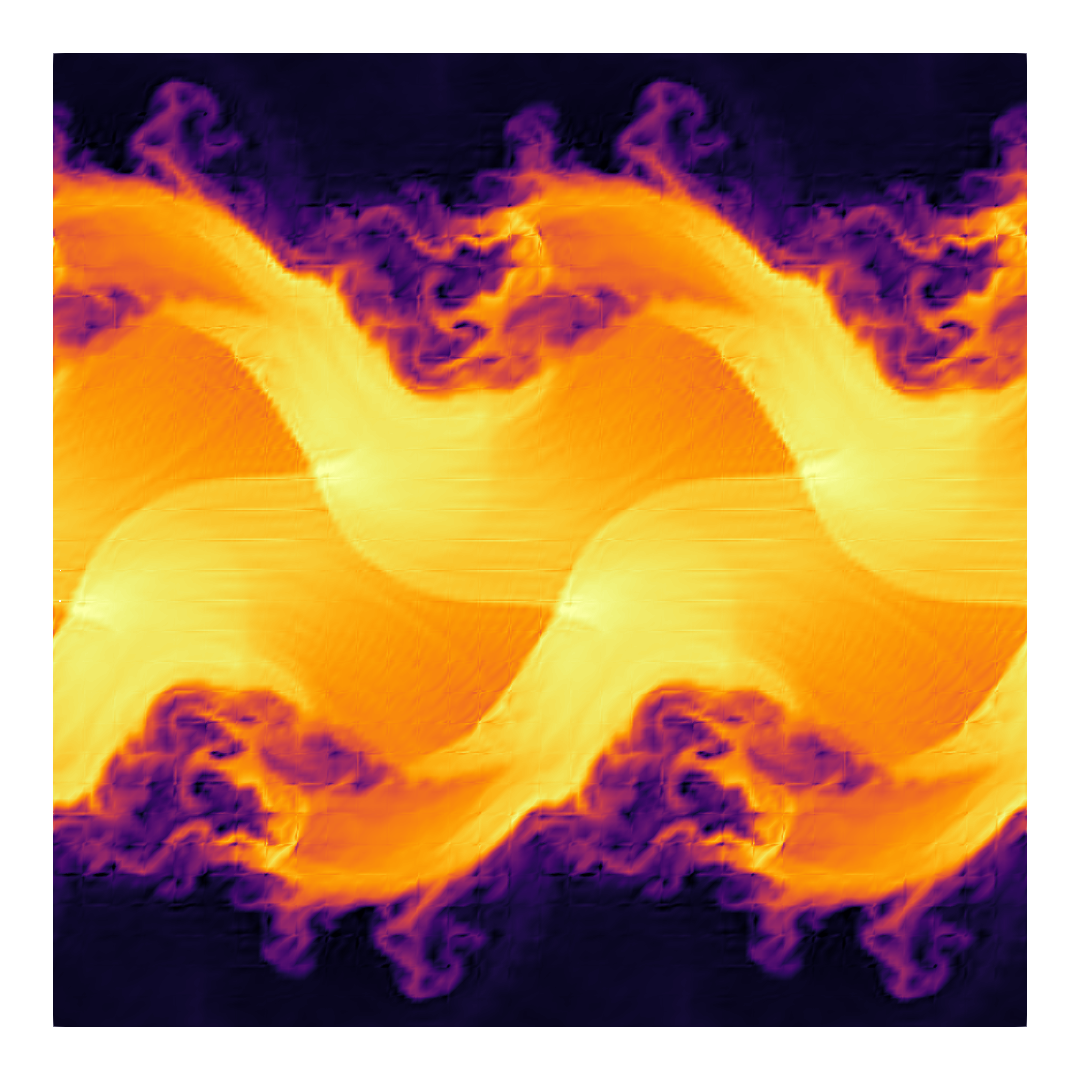}
		& \adjincludegraphics[width=\linewidth, trim={1cm, 0.5\height, 1cm, 1cm}, clip]
			{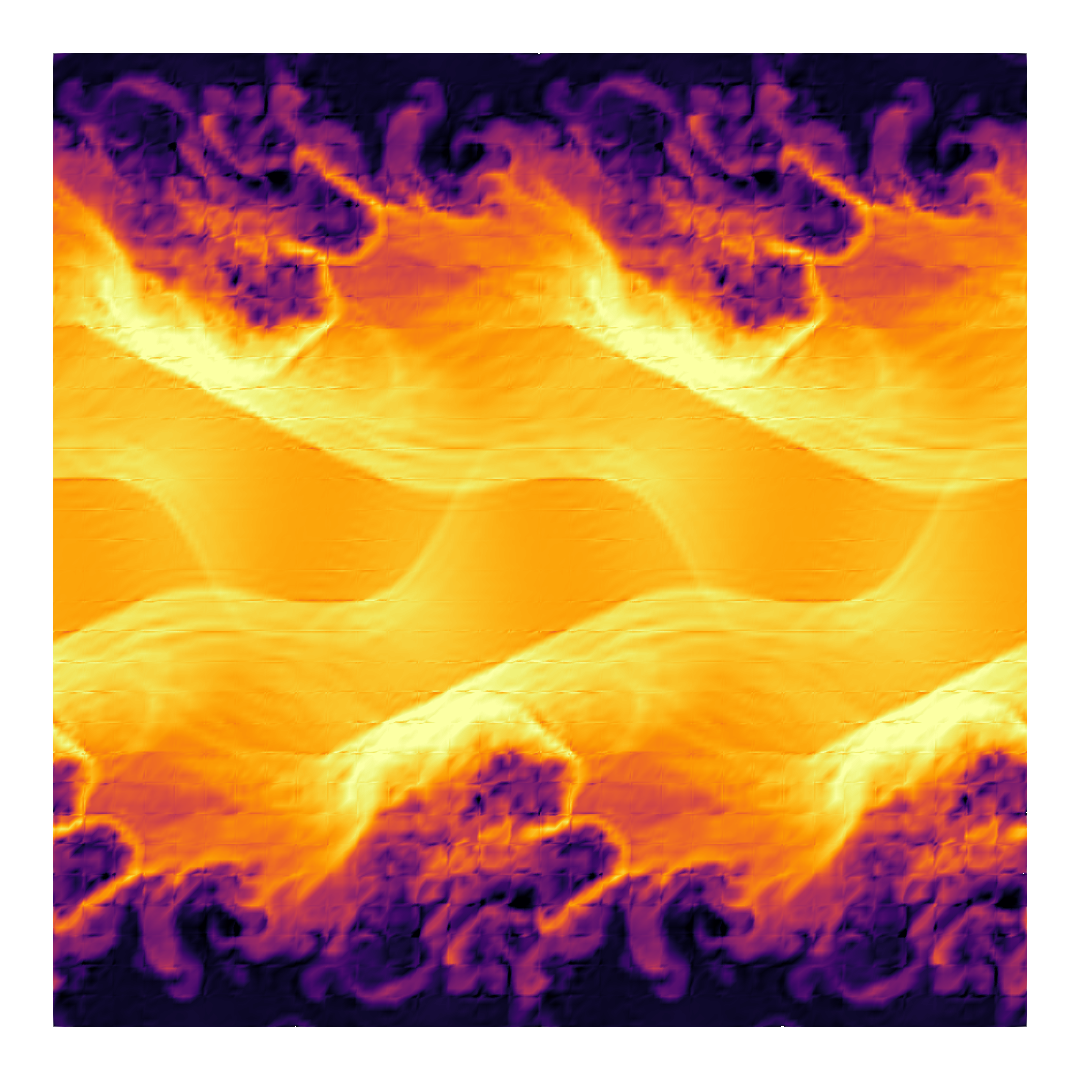}
        \\[-4pt]
		\rotatebox{90}{\(t=6\)}
		& \adjincludegraphics[width=\linewidth, trim={1cm, 0.5\height, 1cm, 1cm}, clip]
			{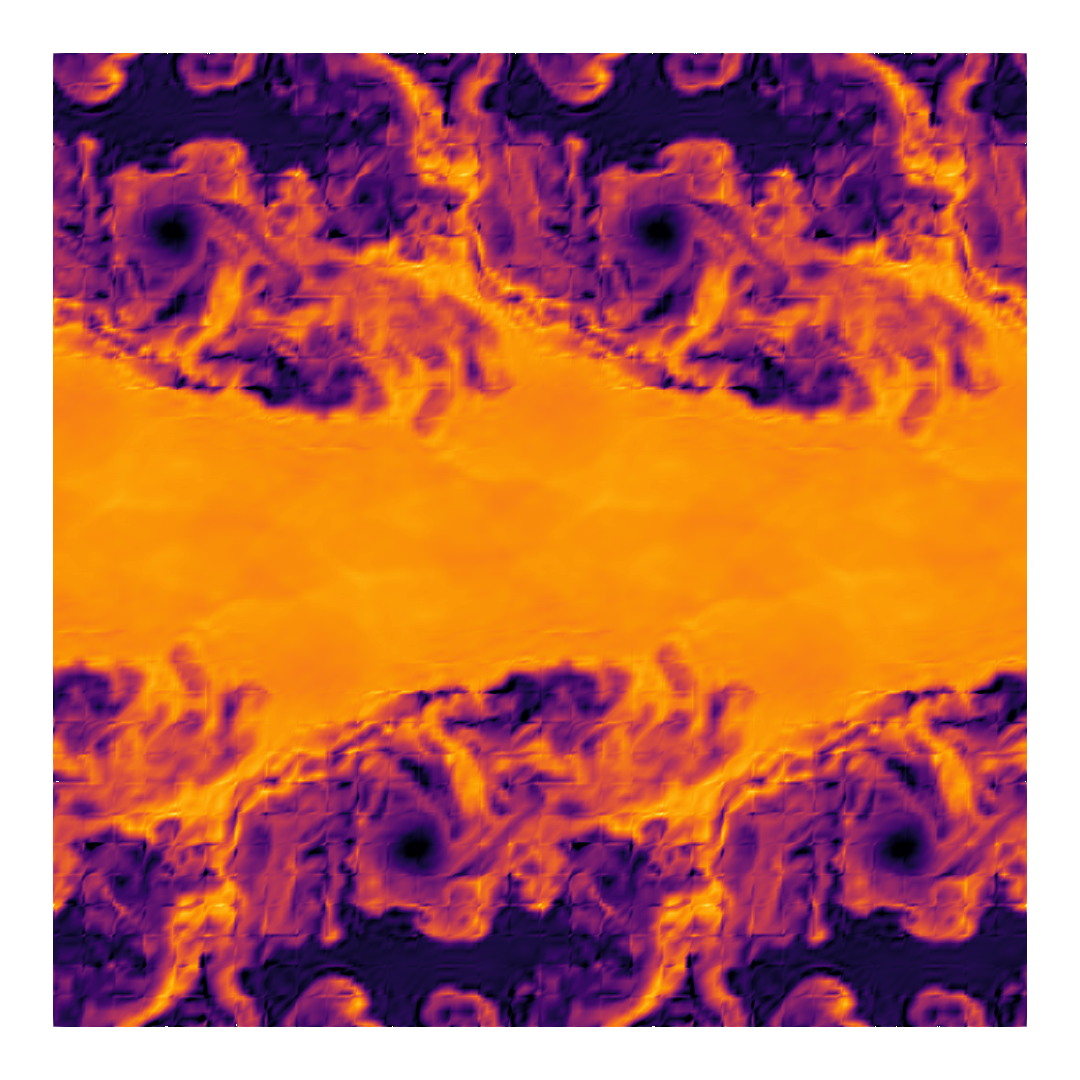}
		& \adjincludegraphics[width=\linewidth, trim={1cm, 0.5\height, 1cm, 1cm}, clip]
			{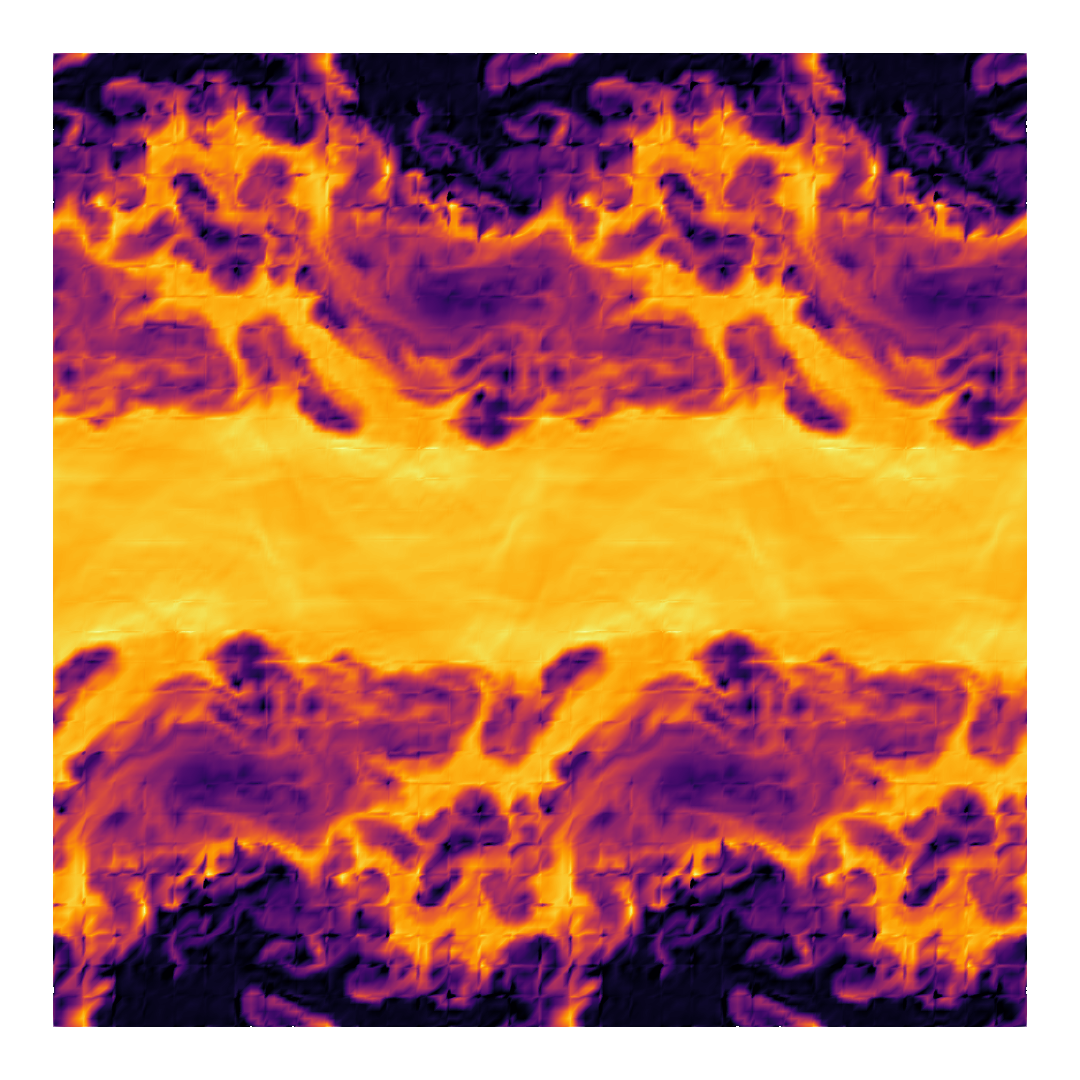}
		& \adjincludegraphics[width=\linewidth, trim={1cm, 0.5\height, 1cm, 1cm}, clip]
			{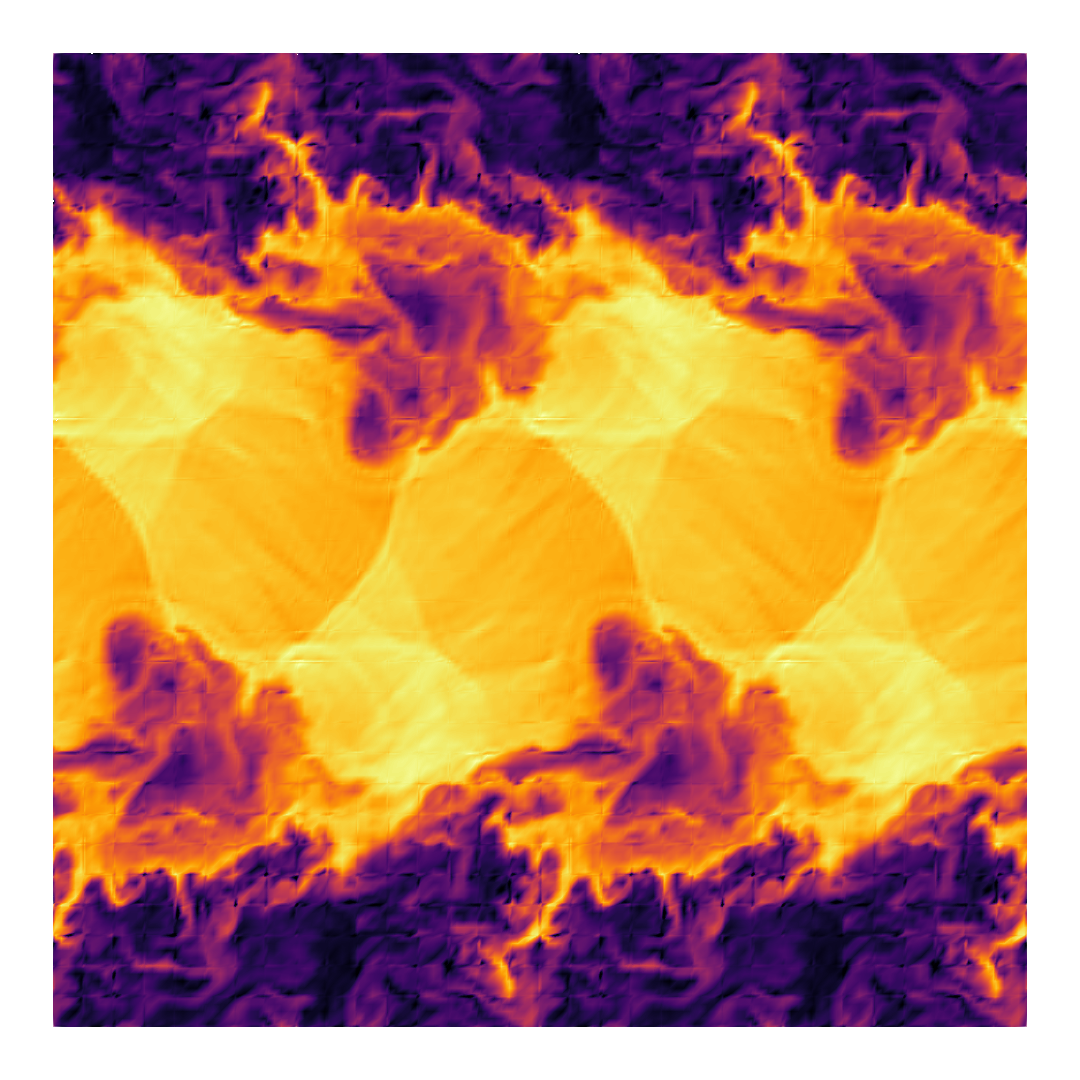}
		& \adjincludegraphics[width=\linewidth, trim={1cm, 0.5\height, 1cm, 1cm}, clip]
			{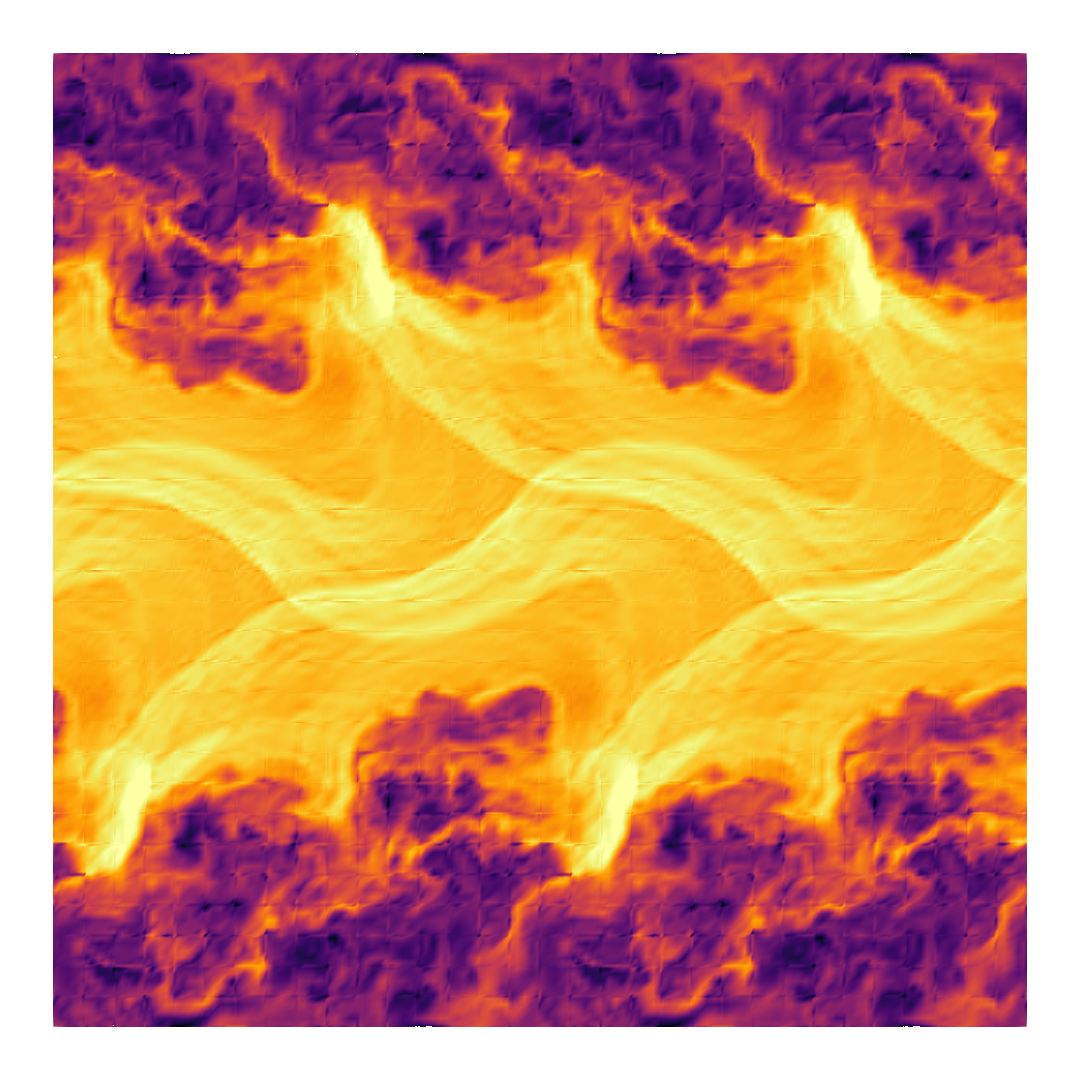}
	\end{tabular}
	\caption{
		Snapshots of the density for the high Atwood number Kelvin-Helmholtz
		instability. The results were generated using the DP DG
		method using degree 6 polynomials on \(32^2\)~elements. We use a
		logarithmic color scale to show the range of scales present in the
		solution.
	}
	\label{fig:kh-atwood-6}
    %\vspace{-0.5cm}
\end{figure}

\section{Summary and outlook } \label{sec:conclusion}
The design of provably stable  numerical methods for nonlinear hyperbolic conservation laws pose a significant challenge. In this study we have proposed and analyzed the DP and SBP framework for accurate and robust numerical approximations of nonlinear conservation laws. The DP SBP FD operators \cite{MATTSSON2017upwindsbp,williams2024drp} are a dual-pair of backward and forward difference operators, which together preserve the SBP property. In addition, the DP SBP  operators are designed to be upwind, that is they come with some built-in dissipation everywhere, as opposed to traditional SBP and DG methods which can only induce dissipation through numerical fluxes acting at element interfaces. We combine the DP SBP operators together with skew-symmetric and upwind flux splitting of nonlinear hyperbolic conservation laws. Our semi-discrete approximation is provably entropy-stable for arbitrary nonlinear hyperbolic conservation laws.  Some examples are given to validate the theoretical guarantees,  and demonstrate significant improvements over the state-of-the-art methods. 

There are multiple directions in which this work can be extended. Important and possible directions  for future work include numerical treatments of non-periodic boundary conditions, multi-block schemes and curvilinear meshes to handle complex geometries and 2D manifolds, local energy stability analysis, proper treatments of supersonic shocks, extensions to more complex applications like thermal shallow water equations,  moist compressible Euler equations and multi-component compressible Euler equations with chemistry, and to other PDEs such as  magnetohydrodynamics equations and plasma simulations.

\section*{Acknowledgments}
Dougal Stewart and Nathan Lee would like to 
acknowledge the kind support and scholarship from The Australian National University (ANU) received through the  Summer Research Scholarship Program at the Mathematical Sciences Institute, ANU.

%\clearpage
% 
% 
\appendix
%%%
%%%
% 
\section{Burgers' Equations}\label{sec:Burgers-equation-numerical-method}
We consider the semi-discrete approximation of the  Burgers' equation with the unknown $u$, the  PDE flux $f(u) = u^2/2$, the entropy $e(u)= u^2/2$ and the entropy variable $g = \partial_u e(u)=u$.
The skew-symmetric form of the divergence of the flux is 
$$
F(u, f(u), \partial_x) = \frac13 u \partial_xu + \frac13 \partial_x(u^2).
$$
The two-element upwind DP SBP semi-discrete approximation of the Burgers'  is given by
\begin{align} \label{eq:burgers-semidiscrete}
\dv{}{t} \mathbf{u} +\frac13 \vb u \vec{D}_x\vb u + \frac13 \vec{D}_x \vb u^2 = \frac{1}{2}\mathbf{H}_x^{-1}\left(\boldsymbol{\alpha} \otimes\left(\vec{B}_{Ix}+\vec{B}_{nx}\right)\right)\mathbf{u} + \frac{1}{2}\Gamma\left(I_m \otimes\left(\vec{D}_{+}-\vec{D}_{-}\right)\right)\mathbf{u}, 
\end{align}
where $\gamma^k=\alpha^k = \max_{j\in \{1, \cdots,n\}}|{ u}^k_j | >0$.
From the first identity in  \eqref{eq:B4_periodic_systems} we find that 
$$
{\inp{\vec g}{ \vb{F}(\vb u, \vb{f}(\vb{u}),  \vec{D}_x)}_{H_x}}=\frac13 \inp{\vb u^2}{\vec{D}\vb u}_H + \frac13 \inp{\vb u}{\vec{D}  \vb u^2}_{H} = 0.
$$
It therefore follows that 
\begin{equation}
   \dv{}{t} E_h(t):= \frac12 \dv{}{t} \norm{\vb u}_{H_x}^2 =-\frac12\sum_{k=1}^{K}\frac{\alpha^{k} + \alpha^{k+1}}{2}\lJump{\mathbf{u}^k}\rJump^2  -\frac{1}{2}\sum_{k=1}^{K}{\gamma^k\inp{ \mathbf{u}^k} {\left(D_{-}-D_{+} \right)\mathbf{u}^k}_{H}}  \le 0.
\end{equation}
Note that because of the periodic boundary conditions, at $k=K$ we have 
$$
\frac{\alpha^{K} + \alpha^{K+1}}{2}\lJump{\mathbf{u}^K}\rJump^2 =\frac{\alpha^{K} + \alpha^{1}}{2}({\mathbf{u}_1^{1}-\mathbf{u}_n^K})^2.
$$

Next we show the conservation property of the semi-discretisation approximation~\eqref{eq:burgers-semidiscrete} of the Burgers' equation.
We will show that the  total "mass" is semi-discretely conserved, as in the continuous equation.
%%%
%%%
Consider the time derivative of the total "mass", we have
{\small
\begin{align*}
\partial_t \inp{\vb 1}{\vb u}_{H_x} = -\frac13 \inp{\vb 1}{\vb u \vec{D}\vb u}_{H_x} - \frac13\inp{\vb 1}{\vec{D} {\vb u}^2}_{H_x} + \frac {\alpha}{2} \inp{ \left(\vec{B}_{I}+\vec{B}_{n}\right)\mathbf{1}}{\mathbf{u}} + \frac\gamma2\inp{(\vec{D}_+-\vec{D}_-) \vb 1}{\vb u}_{H}.
\end{align*}
}
%%%
With the fact that $H_x$ is diagonal we have
\begin{align*}
\partial_t \inp{\vb 1}{\vb u}_{H_x} &= -\frac13 \inp{\vb u}{\vec{D} \vb u}_{H_x} + \frac13\inp{\vec{D}\vb 1}{{\vb u}^2}_{H_x}
= -\frac16(\inp{\vb u}{\vec{D}\vb u}_{H_x} + \inp{\vb u}{\vec{D}\vb u}_{H_x})=0. 
\end{align*}

 \section{Shallow Water Equations}\label{sec:swe-equation-numerical-method}
Now, consider the numerical approximation of the nonlinear shallow water equations in flux form, where the unknowns are $\vb{U}=\left(h, uh\right)^T$ with $h>0$ being the water height, $u$ is the flow velocity and $g>0$ is the acceleration constant due to gravity. The entropy is the total mechanical energy $e(U) = \frac12\left(hu^2 + gh^2\right)$. The skew-symmetric  form of the divergence of the PDE flux and the entropy variables are given by
\begin{align} 
\vb{F}(\vb U, \vb{f}(\vb{U}), \partial_{\vb x})=
\begin{bmatrix}
      \partial_x (hu)\\
    \frac12 \partial_x \left(hu^2 \right) + \frac12 u \partial_x \left(hu \right) + \frac12 hu \partial_x u + gh\partial_x h
\end{bmatrix},
\quad
\vec{g} = \begin{bmatrix} gh -\frac12 u^2\\ u\end{bmatrix}.
\end{align}
A multi-block upwind DP  SBP method for the nonlinear shallow water equations in flux form   is given by
\begin{align}\label{eq:swe1D}
    \dv{}{t} \vb U +  \vb{F}(\vb U, \vb{f}(\vb{U}), \vec{D}_x) =  \frac{1}{2}\mathbf{H}_x^{-1}\left(\boldsymbol{\alpha} \otimes\left(\vec{B}_{Ix}+\vec{B}_{nx}\right)\right)\mathbf{g} +   \frac{1}{2}\Gamma\left(I_m \otimes\left(\vec{D}_{+}-\vec{D}_{-}\right)\right)\mathbf{g}, 
\end{align}
%%%
where  $\Gamma^k = \fn{\diag}{[\gamma_1^k, \gamma_2^k]}$ with, e.g., $\alpha_1^k = \gamma_1^k = \max_{1\le j \le n} h^k_j(|u| + \sqrt{gh^k_j})/(h^k_jg + (u^k_j)^2) \ge 0$, $\alpha_2^k=\gamma_2^k=  \max_{1 \le j \le n}| h_j^ku_j^k |\ge 0$ and 
{
   \begin{align} 
\vb{F}(\vb U, \vb{f}(\vb{U}), \vec{D}_x)=
\begin{bmatrix}
      \vec{D} ({\vec h \vec u}) \\
   \frac12 \vec{D} \left({\vec h \vec u}^2 \right) + \frac12 {\vec u} \vec{D} \left({\vec h \vec u} \right) + \frac12 {\left(\vec h \vec u\right)} \vec{D} {\vec u} + { g \vec h} \vec{D} {\vec h} 
\end{bmatrix},
\quad
\vec{g} = \begin{bmatrix}
 \vec u\\
g\vec h  -\frac12 \vec u ^2
\end{bmatrix}.
\end{align}
}
From the first identity in  \eqref{eq:B4_periodic_systems}, it is easy to check that
{\small 
%\scriptsize
\begin{align*}
{\inp{\vec g}{ \vb{F}(\vb U, \vb{f}(\vb{U}),  \vec{D}_x)}_{H_x}}=\underbrace{\inp{g\vb{h}}{\vec{D} \left(\vb{h}\vb{u} \right)}_{H}+ \inp{\vb{u}\vb{h}}{\vec{D} \left(g\vb{h} \right)}_{H}}_{=0}  + \frac12 \underbrace{\left(\inp{\vb{u}}{\vec{D} \left(\vb{h}\vb{u}^2 \right)}_{H} +  \inp{\vb{h}\vb{u}^2}{\vec{D} \vb{u}}_{H}\right)}_{=0}  =0.
\end{align*}
}
%%%
Then it follows that 
{%\small 
\begin{align*}
   \dv{}{t} E_h(t)={\inp{\vec g}{ \dv{}{t} \vec{U}}_{H_x}}   = - 
   \frac 1 2\sum_{k=1}^{K}\sum_{i=1}^2\frac{\alpha_i^{k} + \alpha_i^{k+1}}{2}\lJump{\mathbf{g}_i^k}\rJump^2
   -\frac 1 2  \sum_{k=1}^K\sum_{i=1}^2{\gamma_i^k\inp{ \mathbf{g}_i^k} {\left(D_{-}-D_{+} \right)\mathbf{g}_i^k}_{H}} 
      \le  0.
\end{align*}
}
Note that because of the periodic boundary conditions, at $k=K$ we have 
$$
\frac{\alpha_i^{K} + \alpha_i^{K+1}}{2}\lJump{\mathbf{g}_i^K}\rJump^2 =\frac{\alpha_i^{K} + \alpha_i^{1}}{2}({\mathbf{g}_{i1}^{1}-\mathbf{g}_{in}^K})^2.
$$

%\begin{example}
Next, we  will show that \eqref{eq:swe1D} semi-discretely conserves mass and momentum. 

To begin, we consider the time derivative of the total mass,
\begin{align*}
    \partial_t \inp{\vb{1}}{\vb{h}}_H  = -\inp{\vec{D}\vb{1}}{\vb{hu}}_H + \frac {\alpha_1}{2} \inp{ \left(\vec{B}_{I}+\vec{B}_{n}\right)\mathbf{1}}{\mathbf{u}} + \frac{\gamma_{1}}{2}\inp{(\vec{D}_+-\vec{D}_-) \vb 1}{\vb u}_{H} =0.
\end{align*}
%%%
% 
Similarly, we show that the semi-discrete approximation \eqref{eq:swe1D}  semi-discretely conserves momentum. We take the time derivative of the total momentum,
{\small
\begin{align*}
    \partial_t \inp{\vb{1}}{\vb{hu}}_H  &= -\inp{\vb{1}}{\frac12 \vec{D} \left(\vb{h}\vb{u}^2 \right) + \frac12 \vb{u} \vec{D} \left(\vb{hu} \right) + \frac12 \vb{hu} \vec{D} \vb{u} + g\vb{h}\vec{D} \vb{h}}_H 
     \\
     &
    -\frac{\alpha_2}{2}\inp{\left(\vec{B}_{I}+\vec{B}_{n}\right) \vb{1}}{\left(g\vb{h} - \frac12 \vb{u}^2\right)}_H  -\frac{\gamma_2}{2}\inp{(\vec{D}_- -\vec{D}_+) \vb{1}}{\left(g\vb{h} - \frac12 \vb{u}^2\right)}_H \\
   % &= \inp{\vb{1}}{\frac12 \vb{u} \vec{D} \left(\vb{hu} \right) + \frac12 \left(\vb{hu}\right) \vec{D} \vb{u} + g\vb{h}\vec{D} \vb{h}}_H \\
    &= \frac12 \inp{\vec{D}\vb{u}}{\vb{hu}}_H - \frac12 \inp{\vb{hu}}{\vec{D}\vb{u}}_H - \frac{g}{2}\inp{\vb{h}}{\vec{D}\vb{h}}_H + \frac{g}{2}\inp{\vec{D}\vb{h}}{\vb{h}}_H = 0.
\end{align*}
}
\section{Compressible Euler Equations}\label{sec:euler-equation-numerical-method}
Here, we consider the 1D compressible Euler equations for ideal gas law in conservative form
\begin{equation} \label{eq:euler-conservative-form}
\begin{split}
\partial_t\vec{U} + \partial_x\vec{f} =0, \quad x\in [0, L], \quad t \in [0, T],
\end{split}
\end{equation}
where the conserved variables and the PDE flux are given by
$$
\vec{U} = 
\begin{bmatrix}
    \rho \\
    \rho u\\
    E
\end{bmatrix}, \quad
\vec{f}(\vec{U}) =
\begin{bmatrix}
    \rho u\\
    \rho u^2 + p\\
    \left(E + p\right) u\\
\end{bmatrix}, \quad E=\frac{p}{\gamma-1} + \frac12 \varrho u^2.
$$
Here, \(\rho\) denotes the fluid density,  \(u\) is the fluid velocity, \(p\) denotes the
pressure, \(E\) is the specific total energy, and \(\gamma\) denotes the adiabatic index.
We also introduce the thermodynamic entropy
\(\eta\) defined by
\begin{equation}\label{eq:thermodynamic-entropy-euler}
\eta = - \rho \log\left(\frac p {\rho^\gamma}\right), \quad p = (\gamma -1) (E - \frac{1}{2}\rho u^2).
\end{equation}
For the compressible Euler equations, deriving a skew-symmetric form that conserves the thermodynamic entropy
\(\eta\) is challenging. In fact, to the best of
our knowledge, we are not aware of a skew-symmetric form  that conserves the thermodynamic entropy
\(\eta \) and the conserved variables~\(\vec U\), that can be targeted by high-order methods. Instead, we will use the conservative skew-symmetric form of
Nordstr\"om \cite{Nordstrom_2022_skewsym_comp}, which considers a  set of transformed variables,  denoted here by \(\vec
V\), and identify an appropriate mathematical entropy functional \(e(\vec V)\) with respect
to~\(\vec V\). However, only entropy functionals with respect to the conservative variables \(\vec U\) must satisfy an entropy inequality, otherwise numerical solutions will converge to the wrong weak solutions when shocks are present. As we will see below, our mathematical entropy is a linear combination of the total specific energy $E$ and the density $\rho$. Thus the mathematical entropy \(e(\vec V)\) is a conserved functional and the skew-symmetric form conserves the total mathematical entropy by using integration by parts only. Our numerical scheme will be designed to conserve the total mathematical entropy. However,  we  will construct both interface and volume upwind terms with respect to the thermodynamic entropy, so that we can ensure thermodynamic entropy consistency in the presence of shocks and discontinuities.
Discrete entropy consistency of the method has been verified through numerical experiments as  shown in Figure~\ref{fig:ce-sod-shock-conservation}.

The skew-symmetric form of the compressible Euler equations
from~\cite{Nordstrom_2022_skewsym_comp}
is given by
\begin{equation} \label{eq:ce-nordstrom}
	\partial_t \vec V + \vec F(\vec V, \vec{f}(\vec V), \partial_x) = 0,
\end{equation}
where the transformed variables and  divergence of flux are given by
\begin{equation}\label{eq:skew-symmetric-Euler}
	\vec V = \begin{bmatrix}
		v_1 \\
		v_2 \\
		v_3
	\end{bmatrix}
    :=
	\begin{bmatrix}
		\sqrt{\rho} \\
		\sqrt{\rho} u \\
		\sqrt{p}
	\end{bmatrix}, \quad
	\vec F(\vec V, \vec{f}(\vec V), \partial_x) =
	\begin{bmatrix}
		\frac{1}{2} \left(
			u \partial_x \sqrt{\rho}
			+ \partial_x (\sqrt{\rho}u)
		\right) \\
		\frac{1}{2} \left(
			\partial_x (\sqrt{\rho} u^2)
			+ u \partial_x (\sqrt{\rho} u)
		\right)
		+ 2\sqrt{\frac{p}{\rho}} \partial_x \sqrt{p} \\
		\frac{1}{2} \left[
			(\gamma - 1) \left(\partial_x (\sqrt{p} u) - u \partial_x \sqrt{p} \right)
			+ \partial_x (\sqrt{p} u) + u \partial_x \sqrt{p}
		\right]
	\end{bmatrix}.
\end{equation}
We define the Jacobian matrices of the forward and inverse transformations $\vec{U} \leftrightarrow \vec{V}$,
\begin{align}\label{eq:Jacobian}
	\partial_{\vec U} \vec{\mathbf{V}}
	=
	\begin{bmatrix}
		\frac{1}{2\sqrt{\rho}} & 0 & 0 \\
		-\frac{u}{2\sqrt{\rho}} & \frac{1}{\sqrt{\rho}} & 0 \\
		\frac{(\gamma - 1)u^2}{4\sqrt{p}} & -\frac{(\gamma-1)u}{2\sqrt{p}} &
			\frac{\gamma-1}{2\sqrt{p}}
	\end{bmatrix}, 
    \quad
    \partial_{\vec V} \vec{\mathbf{U}}
	=
	\begin{bmatrix}
		2\sqrt \rho& 0& 0 \\
		\sqrt \rho u &\sqrt \rho  & 0 \\
		0& \sqrt \rho u & 2\sqrt{p}/(\gamma - 1)
	\end{bmatrix}.
\end{align}
Note that $\partial_{\vec V} \vec{\mathbf{U}} = (\partial_{\vec U} \vec{\mathbf{V}})^{-1}$ and $\partial_{\vec U} \vec{\mathbf{V}}\partial_{\vec V} \vec{\mathbf{U}} = I$. Thus for all $p>0$ and $\rho >0$ the Jacobian matrices are invertible.

The associated entropy pair is
\[
    e({\vec V})
        =  \left(\sqrt{\rho}\right)^2
        + \frac{\left(\sqrt{\rho} u\right)^2}{2}
        + \frac{(\sqrt{p})^2}{\gamma - 1} \equiv \rho + E,\quad
    q
        = u(\rho + E + p).
\]
Note that
\begin{align}
e({\vec V}) = v_1^2 + \frac{1}{2}v_2^2 + \frac{1}{\gamma -1} v_3^2, \quad \norm{\vec V}^2_{(1-\gamma) }:= \int_{\Omega} e d\Omega.
\end{align}
Thus the total mathematical entropy $\norm{\vec V}^2_{(1-\gamma) }$ leads to a norm that is sufficiently strong to control the transformed variables \(\vec V\). The following Lemma will be important for the coming analysis
\begin{lemma}\label{lem:skew-symmetric-continuous-euler}
    Consider the skew-symmetric divergence of the flux $\vec F(\vec V, \vec{f}(\vec V), \partial_x)$ defined in \eqref{eq:skew-symmetric-Euler} and the Jacobian of the inverse transformation $\vec{U} \leftarrow \vec{V}$, $ \partial_{\vec V} \vec{\mathbf{U}}$  given in  \eqref{eq:Jacobian}. By using integration by parts we have
    $$
    \int_{\Omega}\partial_{\vec V} \vec{\mathbf{U}}\vec F(\vec V, \vec{f}(\vec V), \partial_x)\mathrm{dx} = \oint_{\partial \Omega} \vec{f}(\vec U) n_x \mathrm{dx}, \quad 
    \vec{f}(\vec{U}) =
\begin{bmatrix}
    \rho u\\
    \rho u^2 + p\\
    \left(E + p\right) u\\
\end{bmatrix}.
    $$
\end{lemma}
%where \(\alpha > 0\) is an arbitrary constant.
\begin{proof}
    We consider 
    {\small
    \begin{align*}
    \partial_{\vec V} \vec{\mathbf{U}}\vec F(\vec V, \vec{f}(\vec V), \partial_x)
    = 
\begin{bmatrix}
    {\sqrt \rho u}{
               \partial_x \sqrt{\rho}} +{\sqrt \rho}{\partial_x (\sqrt{\rho}u)}
    \\
    \frac{1}{2} \textcolor{Black}{\left(\sqrt{\rho} u^2
			\partial_x \sqrt{\rho}
            + \sqrt{\rho}
			\partial_x (\sqrt{\rho} u^2)\right)}
			+  \frac{1}{2} \textcolor{Black}{\left(\sqrt{\rho} u \partial_x (\sqrt{\rho}u) + \sqrt{\rho} u \partial_x (\sqrt{\rho} u)\right)}
			+ \sqrt{p} \partial_x \sqrt{p}
    \\
    \textcolor{Black}{ \sqrt{\rho} u
			\partial_x (\sqrt{\rho} u^2)
			+ \sqrt{\rho} u^2 \partial_x (\sqrt{\rho} u)
		}
		+ {\sqrt{p}u \partial_x \sqrt{p}}
	+ {\sqrt p}{\partial_x (\sqrt{p} u)
            + \frac{1}{\gamma-1}\left(\textcolor{Black}{\sqrt p \partial_x (\sqrt{p} u)}
            + \textcolor{Black}{\sqrt{p} u \partial_x \sqrt{p}}\right)
		}
    \\
\end{bmatrix}.
    \end{align*}
    }
    We integrate the expression above and apply integration by parts only, and we have
     {\small
    \begin{align*}
   \int_{\Omega}\partial_{\vec V} \vec{\mathbf{U}}\vec F(\vec V, \vec{f}(\vec V), \partial_x)\mathrm{dx} =
\int_{\Omega} {\partial_x
\begin{bmatrix}
    \rho u\\
    \rho u^2 + p\\
    \left(E + p\right) u\\
\end{bmatrix}
} \mathrm{dx} = \oint_{\partial \Omega} {
\begin{bmatrix}
    \rho u\\
    \rho u^2 + p\\
    \left(E + p\right) u\\
\end{bmatrix} n_x
} \mathrm{dx}.
    \end{align*}
    }
\end{proof}
Using the above Lemma the conservation of total mass, total momentum and total energy  follows from multiplying \cref{eq:ce-nordstrom} by \(\partial_{\vec V} \vec{\mathbf{U}}\) and integrating over the domain to obtain
\begin{align*}
\int_{\Omega}\partial_{\vec V} \vec{\mathbf{U}}\partial_t \vec V + \partial_{\vec V} \vec{\mathbf{U}}\vec F(\vec V, \vec{f}(\vec V), \partial_x) \mathrm{dx} =0 \iff  \int_{\Omega} \partial_t \vec U \mathrm{dx} + \oint_{\partial \Omega} \vec{f}(\vec U) n_x =0.
\end{align*}
Note that the conservation of the total mass and  the total energy implies that the total mathematical entropy is also conserved. That is
\begin{align*}
\partial_t\norm{\vec V}^2_{(1-\gamma) }:= \partial_t\int_{\Omega} e d\Omega  = - \oint_{\partial \Omega}(q n_x) \mathrm{dx}, \quad q =  u(\rho + E + p).
\end{align*}
For periodic boundary conditions the boundary terms vanish and we have 
$
\partial_t\norm{\vec V}^2_{(1-\gamma) } = 0.
$
We will formulate the result as the theorem
\begin{theorem}\label{theo:conservation_of_entropy}
Consider the the 1D compressible Euler equations in the skew-symmetric form \cref{eq:ce-nordstrom} on the spatial domain $\Omega= [0, L]$ with periodic boundary conditions and smooth initial data. For continuous solutions $\vb{V}: \Omega\times [0, T] \to \mathbb{R}^{3}$, at time $t \in [0, T]$ let the total mathematical entropy  be denoted by 
\begin{align}
 \norm{\vec V}^2_{(1-\gamma) }:= \int_{\Omega} e d\Omega, \quad e({\vec V}) = v_1^2 + \frac{1}{2}v_2^2 + \frac{1}{\gamma -1} v_3^2, \quad \gamma > 1.
\end{align}
The total mathematical entropy is conserved, that is we have
$$
\partial_t\norm{\vec V}^2_{(1-\gamma) } = 0 \iff \norm{\vec V(\cdot, t)}^2_{(1-\gamma) } = \norm{\vec V(\cdot, 0)}^2_{(1-\gamma) }, \quad \forall \, t \in [0, T].
$$
\end{theorem}
It is significantly noteworthy that the transformed variables $\vec V$ are our prognostic variables and  the conserved variables $\vec U$ are the diagnostic variables.  Thus, for continuous solutions, \cref{theo:conservation_of_entropy} will provide a strong bound for  the transformed variables $\vec V$. It is desirable that our numerical method emulate \cref{theo:conservation_of_entropy}.

Next, we turn our attention to the semi-discrete approximation of the skew-symmetric form \eqref{eq:ce-nordstrom}. We particularly target to conserve the weighted $l_2$-norm of the transformed variables and potentially dissipate the total thermodynamic entropy at shocks and discontinuities. Note that since the proofs at the continuous level use only integration by parts, we can replicate these at the discrete level by using SBP operators. 

As before we consider the two-element discrete DP SBP approximation
{\small
\begin{align}\label{eq:skew_symm_upwind_SBP_SAT-multi-block-euler}
    \dv{}{t} \mathbf{V} +\vb{F}( \vec V, {\vec f}({\vec V}), \vec{D}_x) = \frac{1}{2}\partial_{\vec U}{\vec{\mathbf V}}\mathbf{H}_x^{-1}\left(\boldsymbol{\alpha} \otimes\left(\vec{B}_{Ix}+\vec{B}_{nx}\right)\right)\mathbf{g} + \frac{1}{2}\partial_{\vec U}{\vec{\mathbf V}}\Gamma\left(I_m \otimes\left(\vec{D}_{+}-\vec{D}_{-}\right)\right)\mathbf{g}, 
\end{align}
}
%%%
where the  matrices $\boldsymbol{\alpha}$, $\boldsymbol{\Gamma}$ are given by
{\small
$$
\boldsymbol{\alpha} = \frac{1}{2}\diag\left([\alpha_{1}^1+\alpha_{1}^2, \alpha_{2}^1+\alpha_{2}^2,\alpha_{3}^1+\alpha_{3}^2]\right), \quad
\Gamma = \diag\left([\Gamma^1,\Gamma^2]\right)\otimes I_n, \quad \Gamma^k = \fn{\diag}{[\gamma_1^k, \gamma_2^k, \gamma_3^k]},
$$
}
and the upwind parameters are given by
\begin{equation*}
	\gamma_1 = \max_x \frac{\lambda_\mathrm{max}}{\partial_{\rho}^2 \eta},
	\quad
	\gamma_2 = \max_x \frac{\lambda_\mathrm{max}}{\partial_{\rho u}^2 \eta},
	\quad
	\gamma_3 = \max_x \frac{2M^2 \lambda_\mathrm{max}}{(1 + M^2)\partial_{E}^2 \eta},
\end{equation*}
with the maximum wave speed is \(\lambda_\mathrm{max} = |u| + \sqrt{\gamma p
\rho^{-1}}\), \(M = u / \sqrt{\gamma p \rho^{-1}}\) is the signed Mach number.
We have particularly chosen the upwind entropy variables
$
\vec{g} = \partial_{\vec U} \mathbf \eta
$
where $\mathbf \eta$ is the thermodynamic entropy and $\vec{U}$ is the conserved variable.

We will now establish the conservative properties of the numerical method \cref{eq:skew_symm_upwind_SBP_SAT-multi-block-euler}, which will consequently lead to numerical stability.
The following Lemma is a discrete \cref{lem:skew-symmetric-continuous-euler}
\begin{lemma}\label{lem:skew-symmetric-discrete-euler}
    Consider the semi-discrete approximation of the skew-symmetric divergence of the flux $\vec F(\vec V, {\vec f}({\vec V}), \vec{D}_x)$ defined in \eqref{eq:skew-symmetric-Euler} and the Jacobian of the inverse transformation $\vec{U} \leftarrow \vec{V}$, $ \partial_{\vec V} \vec{\mathbf{U}}$  given in  \eqref{eq:Jacobian}. If the discrete operators $\vec{D}_{+},\vec{D}_{-}$ satisfy  the DP SBP property and the periodic boundary conditions implemented weakly we have
    $$
    \inp{\mathbf{I}_3}{\partial_{\vec V} \vec{\mathbf{U}} \vb{F}( \vec V, {\vec f}({\vec V}), \vec{D}_x)}_{\mathbf{H}_x} = 
    \begin{bmatrix}
    {0}\\
    {0}\\
    {0}\\
\end{bmatrix}.
    $$
\end{lemma}
\begin{proof}
    We consider 
    {\scriptsize
    \begin{align*}
    &\inp{\mathbf{I}_3}{\partial_{\vec V} \vec{\mathbf{U}} \vb{F}( \vec V, {\vec f}({\vec V}), \vec{D}_x)}_{\mathbf{H}_x}=\\
    & 
\begin{bmatrix}
    \inp{\sqrt \rho u}{\vec{D} \sqrt{\rho}}_{H} + \inp{\vec{D} (\sqrt{\rho}u)}{\sqrt \rho}_{H}
    \\
    \frac{1}{2}\left(\inp{\sqrt{\rho} u^2}{\vec{D} \sqrt{\rho}}_{H} 
    + 
    \inp{\vec{D} (\sqrt{\rho} u^2)}{\sqrt \rho}_{H}\right)
    %%%
   %  
 +
  \frac{1}{2}\left(\inp{\sqrt{\rho} u}{\vec{D} \sqrt{\rho} u}_{H} 
    + 
    \inp{\vec{D} (\sqrt{\rho} u)}{\sqrt \rho u}_{H}\right)
    +
     \frac{1}{2}\left(\inp{\sqrt{p}}{\vec{D} \sqrt{p}}_{H} 
    + 
    \inp{\vec{D} (\sqrt{p} )}{\sqrt p}_{H}\right)
    \\
  \frac{1}{2}\left(\inp{\sqrt{\rho} u}{\vec{D} \sqrt{\rho} u^2}_{H} 
    + 
    \inp{\vec{D} (\sqrt{\rho} u)}{\sqrt \rho u^2}_{H}\right)
    %%%
  %   
        +
        \frac{1}{2}\left(\inp{\sqrt{p} u}{\vec{D} \sqrt{p}}_{H} 
    + 
    \inp{\vec{D} (\sqrt{p} u)}{\sqrt p}_{H}\right)
    +
        \frac{1}{\gamma-1}\left(\inp{\sqrt{p} u}{\vec{D} \sqrt{p}}_{H} 
    + 
    \inp{\vec{D} (\sqrt{p} u)}{\sqrt p}_{H}\right)
\end{bmatrix}
= \begin{bmatrix}
    {0}\\
    {0}\\
    {0}
\end{bmatrix}.
    \end{align*}
    }
\end{proof}
\begin{theorem}\label{theo:conservation-discrete-euler}
Consider the semi-discrete DP SBP approximation  \cref{eq:skew_symm_upwind_SBP_SAT-multi-block-euler}. If the discrete derivative operators satisfy the DP SBP framework, then the semi-discretisation is \cref{eq:skew_symm_upwind_SBP_SAT-multi-block-euler} is conservative. That is,
	\(\dv{}{t} \inp{\mathbf{1}}{\vec{U}_i}_{\mathbf{H}_x} =0 \) for all $i = 1, 2, 3$ and $t \ge 0$.
\end{theorem}
\begin{proof}
	Multiplying \cref{eq:skew_symm_upwind_SBP_SAT-multi-block-euler} by  \(\mathbf{I}_3^T \mathbf{H}_x \partial_{\vec V} \vec{\mathbf{U}}\) and use the fact \(\partial_{\vec V} \vec{\mathbf{U}}\dv{}{t} \vec{V} = \dv{}{t} \vec{U}\), we obtain
    {\small
	\begin{align*} 
		 \dv{}{t} \inp{\mathbf{I}_3}{\vec{\mathbf{U}}}_{\mathbf{H}_x}
		+ \inp{\mathbf{I}_3}{\partial_{\vec V} \vec{\mathbf{U}} \vb{F}( \vec V, {\vec f}({\vec V}), \vec{D}_x)}_{\mathbf{H}_x} 
        &= \inp{\mathbf{I}_3}{\left(\boldsymbol{\alpha} \otimes\left(\vec{B}_{Ix}+\vec{B}_{nx}\right)\right)\mathbf{g}}
		+ \inp{\mathbf{I}_3}{\Gamma\left(I_m \otimes\left(\vec{D}_{+}-\vec{D}_{-}\right)\right)\mathbf{g}}_{\mathbf{H}_x}.
	\end{align*}
    }
    By using  \cref{lem:skew-symmetric-discrete-euler} the volume term vanishes $\inp{\mathbf{I}_3}{\partial_{\vec V} \vec{\mathbf{U}} \vb{F}( \vec V, {\vec f}({\vec V}), \vec{D}_x)}_{\mathbf{H}_x} =\vb{0}$. Note that $\alpha$ and $\Gamma$ are diagonal and locally constant within the element. The periodic boundary conditions are implemented weakly.  We use the symmetric properties  of the interface and volume upwind operators, and we have
    \begin{align*} 
        \dv{}{t} \begin{bmatrix}
    \inp{\mathbf{1}}{\vec{U}_1}_{\mathbf{H}}\\
    \inp{\mathbf{1}}{\vec{U}_2}_{\mathbf{H}}\\
    \inp{\mathbf{1}}{\vec{U}_3}_{\mathbf{H}}
\end{bmatrix}= 
		 \dv{}{t} \inp{\mathbf{I}_3}{\vec{\mathbf{U}}}_{\mathbf{H}_x}=
         \begin{bmatrix}
     \inp{\left(\vec{B}_{I}+\vec{B}_{n}\right)\mathbf{1}}{\alpha_1\mathbf{g}_1}\\
     \inp{\left(\vec{B}_{I}+\vec{B}_{n}\right)\mathbf{1}}{\alpha_2\mathbf{g}_2}\\
    \inp{\left(\vec{B}_{I}+\vec{B}_{n}\right)\mathbf{1}}{ \alpha_3\mathbf{g}_3} 
\end{bmatrix}
+
  \begin{bmatrix}
     \inp{\left(\vec{D}_{+}-\vec{D}_{-}\right)\mathbf{1}}{\gamma_1\mathbf{g}_1}\\
     \inp{\left(\vec{D}_{+}-\vec{D}_{-}\right)\mathbf{1}}{\gamma_2\mathbf{g}_2}\\
     \inp{\left(\vec{D}_{+}-\vec{D}_{-}\right)\mathbf{1}}{\gamma_3\mathbf{g}_3} 
\end{bmatrix}
	 = 
\begin{bmatrix}
    {0}\\
    {0}\\
    {0}
\end{bmatrix}.
	\end{align*}
    Thus we must have
    \( \dv{}{t} \inp{\mathbf{1}}{\vec{U}_i}_{\mathbf{H}} =0 \) for all $i = 1, 2, 3$ and $t \ge 0$.
\end{proof}
Note that
$$
\partial_t\norm{\vec V}^2_{(1-\gamma) }:=\dv{}{t} \inp{\mathbf{1}}{\vec{U}_1}_{\mathbf{H}} + \dv{}{t} \inp{\mathbf{1}}{\vec{U}_3}_{\mathbf{H}}  =0.
$$
Thus the following corollary is a consequence of the conservative property of the numerical method established by \cref{theo:conservation-discrete-euler}.
\begin{corollary}\label{coro:stability}
Consider the semi-discrete DP SBP approximation  \cref{eq:skew_symm_upwind_SBP_SAT-multi-block-euler}. If the discrete derivative operators satisfy the DP SBP framework, then the semi-discretisation is \cref{eq:skew_symm_upwind_SBP_SAT-multi-block-euler} is conservative and the total mathematical entropy is conserved. That is,
$$
\partial_t\norm{\vec V}^2_{(1-\gamma) \mathbf{H}_x} = 0 \iff \norm{\vec V(\cdot, t)}^2_{(1-\gamma) \mathbf{H}_x} = \norm{\vec V(\cdot, 0)}^2_{(1-\gamma) \mathbf{H}_x}, \quad \forall \, t \in [0, T].
$$
\end{corollary}
The above corollary establishes $l_2$-stability for the semi-discrete approximation \cref{eq:skew_symm_upwind_SBP_SAT-multi-block-euler}. The next theorem proves that the method will potentially dissipate thermodynamic entropy at shocks and discontinuities.
\begin{theorem}\label{theo:entropy-consistent-discrete-euler}
Consider the semi-discrete DP SBP approximation  \cref{eq:skew_symm_upwind_SBP_SAT-multi-block-euler}. If the discrete derivative operators satisfy the DP SBP framework, then the semi-discretisation is \cref{eq:skew_symm_upwind_SBP_SAT-multi-block-euler} satisfies the inequality,
{\footnotesize
    \begin{align*} 
		 \dv{}{t} \inp{\mathbf{1}}{\vec{\mathbf{\eta}}}_{\mathbf{H}_x}
		+ \inp{\mathbf{\partial_{\vec V} \vec{\mathbf{\eta}}}}{ \vb{F}( \vec V, {\vec f}({\vec V}), \vec{D}_x)}_{\mathbf{H}_x} 
        &= - 
   \frac 1 2\sum_{k=1}^{K}\sum_{i=1}^3\frac{\alpha_i^{k} + \alpha_i^{k+1}}{2}\lJump{\mathbf{g}_i^k}\rJump^2
   -\frac 1 2  \sum_{k=1}^K\sum_{i=1}^3{\gamma_i^k\inp{ \mathbf{g}_i^k} {\left(D_{-}-D_{+} \right)\mathbf{g}_i^k}_{H}}  \le 0,
	\end{align*}
    }
    where $\mathbf{\eta}$ is the thermodynamic entropy.
\end{theorem}
\begin{proof}
	Multiplying \cref{eq:skew_symm_upwind_SBP_SAT-multi-block-euler} by  \(\mathbf{1}^T \mathbf{H}_x \partial_{\vec V} \vec{\mathbf{\eta}}\). We use the facts \(\partial_{\vec V} \vec{\mathbf{\eta}}\dv{}{t} \vec{V} = \dv{}{t} \vec{\eta}\) and \(\vec{g}= \partial_{\vec U} \vec{\mathbf{\eta}} =\partial_{\vec V} \vec{\mathbf{\eta}} \partial_{\vec U} \vec{\mathbf{V}}\), we obtain
    {\small
	\begin{align*} 
		 \dv{}{t} \inp{\mathbf{1}}{\vec{\mathbf{\eta}}}_{\mathbf{H}_x}
		+ \inp{\mathbf{\partial_{\vec V} \vec{\mathbf{\eta}}}}{ \vb{F}( \vec V, {\vec f}({\vec V}), \vec{D}_x)}_{\mathbf{H}_x} 
        &= \inp{\mathbf{g}}{\left(\boldsymbol{\alpha} \otimes\left(\vec{B}_{Ix}+\vec{B}_{nx}\right)\right)\mathbf{g}}
		+ \inp{\mathbf{g}}{\Gamma\left(I_m \otimes\left(\vec{D}_{+}-\vec{D}_{-}\right)\right)\mathbf{g}}_{\mathbf{H}_x} \le 0.
	\end{align*}
    }
    and
    {\footnotesize
    \begin{align*} 
		 \dv{}{t} \inp{\mathbf{1}}{\vec{\mathbf{\eta}}}_{\mathbf{H}_x}
		+ \inp{\mathbf{\partial_{\vec V} \vec{\mathbf{\eta}}}}{ \vb{F}( \vec V, {\vec f}({\vec V}), \vec{D}_x)}_{\mathbf{H}_x} 
        &= - 
   \frac 1 2\sum_{k=1}^{K}\sum_{i=1}^3\frac{\alpha_i^{k} + \alpha_i^{k+1}}{2}\lJump{\mathbf{g}_i^k}\rJump^2
   -\frac 1 2  \sum_{k=1}^K\sum_{i=1}^3{\gamma_i^k\inp{ \mathbf{g}_i^k} {\left(D_{-}-D_{+} \right)\mathbf{g}_i^k}_{H}}  \le 0,
	\end{align*}
    }
    \end{proof}

\section{Numerical Experiments: Shallow Water Equations}\label{swe-num-experiments}
Here, we present numerical simulations for the nonlinear SWEs. We will  consider the SWEs in 1D and 2D, and investigate  robustness, efficacy and numerical well-balanced properties.

\subsection{1D Dambreak}
To investigate the robustness of the schemes further, we consider the 1D dam break problem, with the initial conditions
$$
h_0 = 
    \begin{cases}
        1.2 & \text{if } |x| > 15 \\
        0.2 & \text{if } |x| < 15
    \end{cases}, \quad
    u_0 = 0,\quad
    x \in [-30, 30],\quad
    t \in [0, 10],
$$
where the gravitational acceleration is set to \(g=1\). 
For the DG scheme, we vary the number of elements $64\le K\le 256$, the polynomial degree $3\le p\le 6$ and we increase the free parameter in the upwind operators to \(\lambda_n = -0.2\).
Additionally, we consider the DP FD  operators with interior order of accuracy $5,6,7,8,9$, discretized with $32$ elements, and vary the number of nodes within each element $17\le N\le 65$. A fixed timestep of \(\Delta t = 0.001 \Delta x\) is used, and the simulation is run until the final time. Snapshots of the water height and momentum are shown in \cref{fig:swe-dambreak-snapshot} at $t=4$. 
In \cref{fig:swe-dambreak-conservation}, we plot the relative changes in total energy/entropy, total mass, and total momentum. It is noteworthy that our DP SBP FD/DG schemes (with $\gamma >0 $) are stable and complete the simulation without crashing until the final time. In contrast, all other schemes---including linearly stable schemes (with $\gamma >0 $) and the standard SBP/DGSEM (with $\gamma =0 $)---fail before reaching the final simulation time. These results further demonstrate the robustness of our entropy-stable DP SBP FD/DG schemes (with $\gamma >0 $). 
Furthermore, our DP SBP FD/DG schemes (with $\gamma >0 $) are conservative and entropy consistent, as illustrated in the first column of \cref{fig:swe-dambreak-conservation}.
\begin{figure}[htbp]
    \centering
    \includegraphics[scale=\figurescaling]{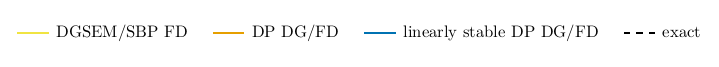} \\[-0.5em]
    \begin{subcaptionblock}{0.49\textwidth}
        \centering
        \includegraphics[scale=\figurescaling]{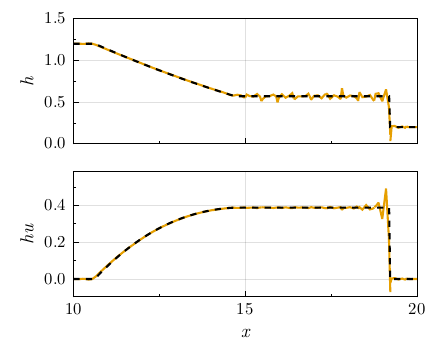}
        \vspace{-0.25cm}
        \caption{DP DG with 6th degree polynomials on 128 elements}
    \end{subcaptionblock}
    \hfill
    \begin{subcaptionblock}{0.49\textwidth}
        \centering
        \includegraphics[scale=\figurescaling]{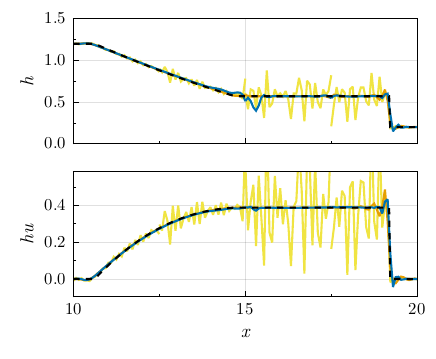}
        \vspace{-0.25cm}
        \caption{6th order DP FD on 24 elements with 33 nodes each}
    \end{subcaptionblock}
    %\vspace{-0.5cm}
    \caption{Snapshot of the solution to the dambreak problem at \(t=4.0\).}
    \label{fig:swe-dambreak-snapshot}
\end{figure}
\begin{figure}[htbp]
    \centering
    \includegraphics[scale=\figurescaling]{images/swe-1D/dambreak/N_7_deriv_order_6_deriv_type_GlaubitzEtal2024_1_nblocks_128/legend-entr-horizontal.pdf}\\[-1em]
    \begin{subcaptionblock}{\textwidth}
        \centering
        \includegraphics[scale=\figurescaling]{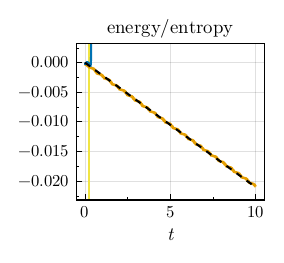}%
        \includegraphics[scale=\figurescaling]{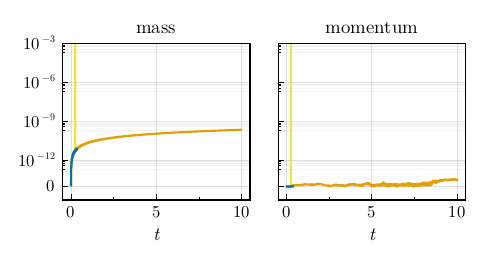}
        \\[-1.0em]
        \caption{DP DG with 6th degree polynomials on 128 elements.}
    \end{subcaptionblock}
    \begin{subcaptionblock}{\textwidth}
        \centering
        \includegraphics[scale=\figurescaling]{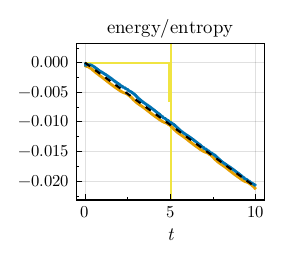}%
        \includegraphics[scale=\figurescaling]{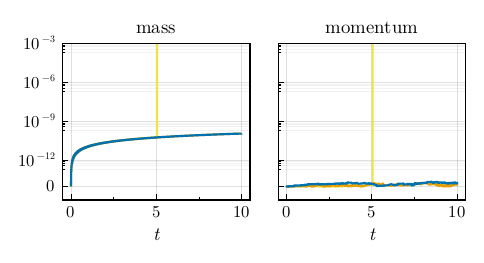}
        \\[-1.0em]
        \caption{6th order DP FD on 24 elements with 3 nodes per element.}
    \end{subcaptionblock}
     \vspace{-0.5cm}
    \caption{Relative change in the total entropy/energy, total mass and total momentum for the dam beak problem.}
    \label{fig:swe-dambreak-conservation}
\end{figure}
We have also performed numerical simulations with different discretisation configurations. The crash or final  times are reported in \cref{tab:swe-dambreak-crash-times-fd,tab:swe-dambreak-crash-times-dg}. It is significantly noteworthy that our DP FD and DP DG schemes (with $\gamma >0 $) run until the final time, $t=10$, for all configurations without crashing.
\begin{table}[htbp]
    \centering
    \scalebox{\tablescaling}{\begin{tabular}{@{}r@{}*{3}{c@{}*{4}{r}@{}}}
    \toprule
    & \hspace{2em}
    & \multicolumn{4}{c}{linearly stable DP DG} &\hspace{2.5em}
    & \multicolumn{4}{c}{DGSEM} &\hspace{2.5em}
    & \multicolumn{4}{c}{DP DG} \\
    \cmidrule{3-6}
    \cmidrule{8-11}
    \cmidrule{13-16}
    &
    & \multicolumn{4}{c}{polynomial degree} &
    & \multicolumn{4}{c}{polynomial degree} &
    & \multicolumn{4}{c}{polynomial degree} \\
    K &
    & 3 & 4 & 5 & 6 &
    & 3 & 4 & 5 & 6 &
    & 3 & 4 & 5 & 6 \\
    \midrule
    64 &
    % LxF
	& \crash{6.70} & \crash{0.64} & \crash{0.67} & \crash{0.72} &
    % Ns
	& \crash{0.53} & \crash{0.68} & \crash{0.91} & \crash{0.49} &
    % UNs
    & \nocrash{10} & \nocrash{10} & \nocrash{10} & \nocrash{10} \\
    128 &
    % LxF
	& \crash{3.35} & \crash{0.32} & \crash{0.33} & \crash{0.36} &
    % Ns
	& \crash{0.27} & \crash{0.34} & \crash{0.46} & \crash{0.25} &
    % UNs
    & \nocrash{10} & \nocrash{10} & \nocrash{10} & \nocrash{10} \\
    256 &
    % LxF
	& \crash{1.68} & \crash{0.16} & \crash{0.17} & \crash{0.18} &
    % Ns
	& \crash{0.13} & \crash{0.17} & \crash{0.23} & \crash{0.12} &
    % UNs
    & \nocrash{10} & \nocrash{10} & \nocrash{10} & \nocrash{10} \\
    \bottomrule
\end{tabular}}
    % \vspace{-0.5cm}
    \caption{Crash or final simulation times for the DG approximations of the  dam break problem.}
    \label{tab:swe-dambreak-crash-times-dg}
\end{table}
%%%
\begin{table}[htbp]
    \centering
    \scalebox{\tablescaling}{\begin{tabular}{@{}rr*{3}{@{}c@{}*{5}{r}}@{}}
	\toprule
	& & \hspace{2.5em}
	& \multicolumn{5}{c}{linearly stable DP FD} & \hspace{2.5em}
	& \multicolumn{5}{c}{SBP FD} & \hspace{2.5em}
	& \multicolumn{5}{c}{DP FD} \\
	\cmidrule{4-8}
	\cmidrule{10-14}
	\cmidrule{16-20}
	& &
	& \multicolumn{5}{c}{accuracy order} &
	& \multicolumn{5}{c}{accuracy order} &
	& \multicolumn{5}{c}{accuracy order} \\
	K & n &
	& 5 & 6 & 7 & 8 & 9 &
	& 5 & 6 & 7 & 8 & 9 &
	& 5 & 6 & 7 & 8 & 9 \\
	\midrule
	32 & 17 &
	& % LxF
	\nocrash{10} & \nocrash{10} & \nocrash{10} & \crash{1.46} & \crash{1.47} &
	& % SkS
	\crash{7.38} & \crash{6.50} & \crash{0.63} & \crash{0.57} & \crash{0.59} &
	& % USkS
	\nocrash{10} & \nocrash{10} & \nocrash{10} & \nocrash{10} & \nocrash{10} \\
	32 & 33 &
	& % LxF
	\nocrash{10} & \nocrash{10} & \nocrash{10} & \crash{1.18} & \crash{1.00} &
	& % SkS
	\crash{4.30} & \crash{3.79} & \crash{0.61} & \crash{0.28} & \crash{0.26} &
	& % USkS
	\nocrash{10} & \nocrash{10} & \nocrash{10} & \nocrash{10} & \nocrash{10} \\
	32 & 65 &
	& % LxF
	\nocrash{10} & \nocrash{10} & \nocrash{10} & \crash{0.59} & \crash{0.50} &
	& % SkS
	\crash{4.31} & \crash{2.99} & \crash{0.28} & \crash{0.14} & \crash{0.13} &
	& % USkS
	\nocrash{10} & \nocrash{10} & \nocrash{10} & \nocrash{10} & \nocrash{10} \\
	\bottomrule
\end{tabular}}
    \caption{Crash or final simulation times for the FD approximations of the  dam break problem.}
    \label{tab:swe-dambreak-crash-times-fd}
\end{table}

\subsection{2D Well-balanced}
Next, we consider the 2D well-balanced test \cite{Delestre_2012}. A desirable property for a numerical method for conservation laws is to preserve steady state solutions. This property is known as well-balanced properties for the SWE.
We consider the 2D lake at rest problem with gravitational acceleration \(g=9.81\), and the initial and  exact solutions given by
$$
    h=h_0 = \frac{1}{2} - b(x,y),\quad
    u=u_0 = 0,\quad
    v=v_0 = 0,\quad
    (x,y) \in [0,25]^2,\quad 
    t \in [0, 1],
$$
%%%
where the bottom topography is
$$
    b(x,y) = \begin{cases}
        0.2 - 0.05\left((x-10)^2 + (y-10)^2\right), & \text{if } (x-10)^2 + (y-10)^2 < 4, \\
        0, & \text{otherwise}.
    \end{cases}
$$
%%%
Note that we can prove that our semi-discrete scheme is perfectly well-balanced for steady-state solutions with zero velocity, but we omit this for brevity. Thus, we expect our method to achieve machine-precision rounding error for the lake at rest problem.  We use a fixed timestep of \(\Delta t = 0.01\Delta x\) and run the simulation until the final time.
The maximum error over the domain at the final time \(t=1\) can be found in \cref{tab:swe-2d-well-balanced} for the  FD and  DG methods respectively. Indeed, the numerical errors in \cref{tab:swe-2d-well-balanced} are within the machine-precision rounding errors, which verify the well-balanced property of the schemes.

\begin{table}[htbp]
    \centering
    \begin{subcaptionblock}{0.49\textwidth}
        \centering
        \scalebox{\tablescaling}{
            \begin{tabular}{@{}rr*{3}{r}@{}}
	\toprule
	K & n & lin. stable DP DG & DGSEM & DP DG \\
	\midrule
	  16 &    7 & \num{7.93e-13} & \num{7.32e-13} & \num{1.51e-12} \\
	  32 &    7 & \num{1.50e-12} & \num{1.47e-12} & \num{2.30e-12} \\
	  64 &    7 & \num{2.97e-12} & \num{2.95e-12} & \num{3.70e-12} \\
	\bottomrule
\end{tabular}

        }
        \caption{DG with degree 6 polynomials}
    \end{subcaptionblock}
    \begin{subcaptionblock}{0.49\textwidth}
        \centering
        \scalebox{\tablescaling}{
            \begin{tabular}{@{}rr*{3}{r}@{}}
	\toprule
	K & n & lin. stable DP FD & SBP FD & DP FD \\
	\midrule
	   4 &   17 & \num{2.50e-13} & \num{2.53e-13} & \num{2.50e-13} \\
	   4 &   33 & \num{4.95e-13} & \num{4.97e-13} & \num{4.95e-13} \\
	   4 &   65 & \num{9.92e-13} & \num{9.94e-13} & \num{9.92e-13} \\
	\bottomrule
\end{tabular}

        }
        \caption{6th order FD}
    \end{subcaptionblock}
    %\vspace{-0.5cm}
    \caption{Maximum error at the final time \(t=1\) of the FD and DG methods for the lake at rest problem.}
    \label{tab:swe-2d-well-balanced}
\end{table}

\subsection{Merging Vortices}
Next we consider the merging vortices test case~\cite{mcrae2014energy,hew2024stronglystabledualpairingsummation} for the 2D rotating nonlinear shallow water equations. The initial conditions are a pair of Gaussian vortices with the incompressible stream function $\psi = \psi_+ + \psi_-$ defined by
\begin{equation}
    \psi_\pm(x,y) = \fn{\exp}{-2.5\left[\left(x - \frac{3.05 \pm 0.45}{3}\pi\right)^2 + (y-\pi)^2\right]}.
\end{equation}
To ensure linear geostrophic balance $f (v - u) + g\nabla h = 0$, we set the initial conditions
\[
    h_0 = H + \frac f g \psi(x,y), \quad 
    u_0 = -\partial_y \psi(x,y), \quad
    v_0 = \partial_x \psi(x,y), \quad
    (x,y) \in [0,2\pi]^2, \quad
    t \in [0, 20],
\]
where $f = g = 5$ and $H = 8$.

We consider DG schemes using 6th degree polynomials on \(32^2\) elements and 6th interior order accurate FD schemes on \(6^2\) elements with \(33^2\) nodes each. The solutions are evolved using a fixed timestep of \(\Delta t = 10^{-2}\Delta x\) until the final time

Snapshots of the potential vorticity \(h^{-1}\omega\), where \(\omega = \partial_x v - \partial_y u + f\) is the absolute vorticity, can be found in \cref{fig:swe-2d-merging-vortex-snapshot}.
Shown in \cref{fig:swe-2d-merging-vortex-entropies} is the relative change with time of the total energy/entropy and the total enstrophy \(h^{-1}\omega^2\), which is also conserved for smooth solutions.

Although the DGSEM/SBP FD methods without volume upwinding are nonlinearly stable, they are unable to control the total enstrophy, which bounds gradients and higher moments of the solution. As a result, we observe  growth of the total enstrophy along with the growth of spurious oscillations that destroy the structure of the solution. Note that at time $2$, the vortices are well resolved and the total enstrophy is unchanged, but at a later time \(t=16\), the solution has been consumed by numerical noise and the total enstrophy has more than doubled. On the other hand, the DP DG/FD methods with volume upwinding are able to control the enstrophy and limit the unphysical growth of high frequency oscillations. In addition, note that, because the solution is well resolved, the dissipated energy is less than \(10^{-5}\) of the total energy and the dissipated enstrophy is \(10^{-2}\) of the total enstrophy.

\begin{figure}[htbp]
    \centering
    \includegraphics[scale=\figurescaling]{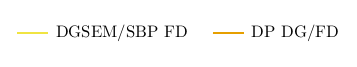}\\[-0.5em]
    \begin{subcaptionblock}{\textwidth}
        \centering
        \includegraphics[scale=\figurescaling]{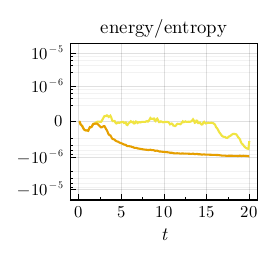}%
        \includegraphics[scale=\figurescaling]{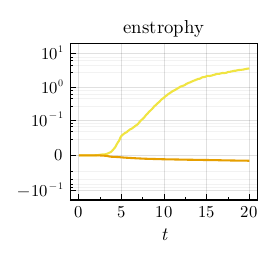}
        \caption{DG using 6th degree polynomials on \(32^2\) elements.}
    \end{subcaptionblock}
    \begin{subcaptionblock}{\textwidth}
        \centering
        \includegraphics[scale=\figurescaling]{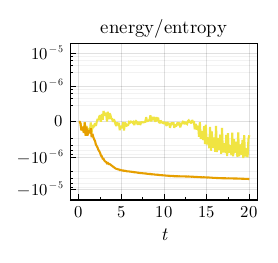}%
        \includegraphics[scale=\figurescaling]{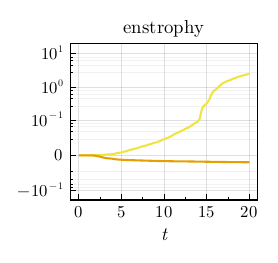}
        \caption{6th order FD on \(6^2\) elements with \(33^2\) nodes each.}
    \end{subcaptionblock}
    \caption{The relative change with time of the total energy/entropy and the total enstrophy for the merging vortices testcase.}
    \label{fig:swe-2d-merging-vortex-entropies}
\end{figure}

\begin{figure}[htbp]
    \centering
    \begin{subcaptionblock}{0.43\textwidth}
        \centering
        \includegraphics[scale=\figurescaling, trim={0cm, 0cm, 2.4cm, 0cm}, clip]{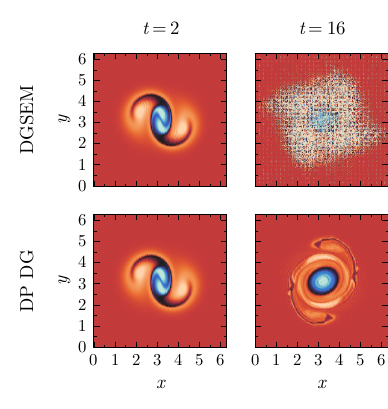}
        \caption{DG using 6th degree polynomials on \(32^2\) elements.}
    \end{subcaptionblock}\hfill%
    \begin{subcaptionblock}{0.55\textwidth}
        \centering
        \includegraphics[scale=\figurescaling]{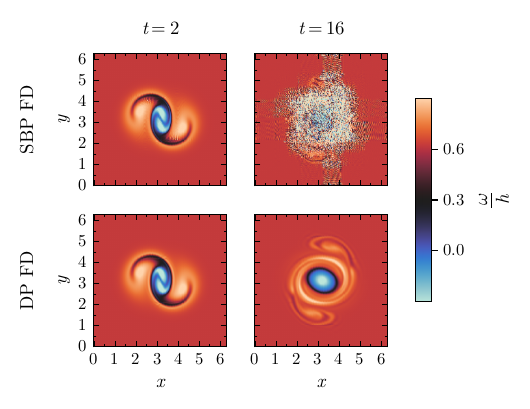}
        \caption{6th order FD on \(6^2\) elements with \(33^2\) nodes per element.}
    \end{subcaptionblock}
    \caption{Snapshots of the potential vorticity for the skew-symmetric entropy conservative DGSEM/SBP FD methods and our novel DP DG/FD methods for the merging vortex testcase.}
    \label{fig:swe-2d-merging-vortex-snapshot}
\end{figure}

\subsection{Barotropic Shear Instability}
Our final numerical example for the nonlinear shallow water equations is the  the barotropic shear instability \cite{hew2024stronglystabledualpairingsummation,galewsky2004initial,peixoto2019semi}. This is also known as the Kelvin-Helmholtz instability, and involves triggering a barotropic shear within zonal jets by initializing the flow with a thin fluid discontinuity, supplemented with small Gaussian perturbations. 

The computational domain is $\Omega = [0, L]^2$, with $L = 2\pi \times \SI{6371.22}{\km}$ with doubly periodic boundary conditions. The initial conditions are
\[
    u_0 = \bar{u}_0 \left(\fn{\mathrm{sech}}{\frac{y - y_+}{10^6}}
        - \fn{\mathrm{sech}}{\frac{y - y_-}{10^6}}\right), \quad
    v_0 =0, \quad
    {h}_0 = H -\frac f g \int_0^y u(x,s) \dd{s} + \tilde{h}(x,y),
\]
with
\[
    \tilde{h}(x,y) = \bar{h}_0 \sum_{i=1}^2 \exp(-k d_i(x,y)), \quad
    d_i(x, y)=\frac{\left(x-x_i\right)^2}{L^2}+\frac{\left(y-y_i\right)^2}{L^2}
    \quad \forall i=\{1,2\}, 
\]
where
$\bar{h}_0=0.01 H$,
$\left(x_1, y_1\right)=\left(0.15 L, y_+\right)$,
$\left(x_2, y_2\right)=\left(0.85 L, y_-\right)$,
$y_+ = 0.25L$,
$y_- = 0.75L$,
$\bar{u}_0 = \SI{50}{\m\per\s}$,
$f = \num{7.292e-5}$,
$g = \SI{9.80616}{\m\per\square\s}$,
$H = \SI{10}{\km}$,
$k = 10^3$.

We consider DG schemes using 6th degree polynomials on \(64^2\) elements and 6th interior order accurate FD schemes on \(6^2\) elements with \(65^2\) nodes each. The solutions are evolved using a fixed timestep of \(\Delta t = 2\times10^{-3}\Delta x\) until the final time \(t=80\,\mathrm{days}\).
Snapshots of the absolute vorticity can be found in \cref{fig:swe-2d-kh-snapshot}.
In \cref{fig:swe-2d-kh-entropies}, we plot the relative change with time of the total energy/entropy and the total enstrophy.

Similar to the merging vortices testcase, the entropy/energy conservative DGSEM/SBP FD methods without volume upwinding exhibit uncontrolled growth of the total enstrophy and generate numerical noise that destroys the solution. Meanwhile, the entropy stable DP FD/DG methods with volume upwinding are able to suppress the unphysical enstrophy growth and resolve the structure of the vortices without generating spurious high frequency oscillations.

\begin{figure}[htbp]
    \centering
    \includegraphics[scale=\figurescaling]{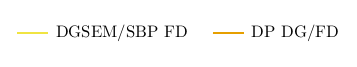}\\[-0.5em]
    \begin{subcaptionblock}{\textwidth}
        \centering
        \includegraphics[scale=\figurescaling]{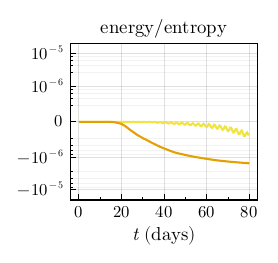}%
        \includegraphics[scale=\figurescaling]{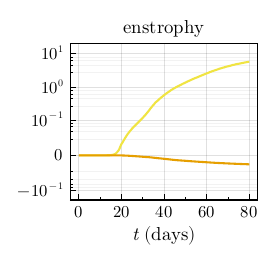}
        \caption{DG using 6th degree polynomials on \(64^2\) elements.}
    \end{subcaptionblock}
    \begin{subcaptionblock}{\textwidth}
        \centering
        \includegraphics[scale=\figurescaling]{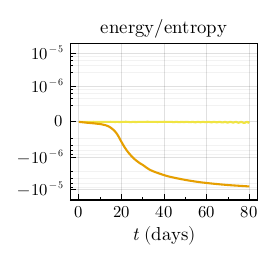}%
        \includegraphics[scale=\figurescaling]{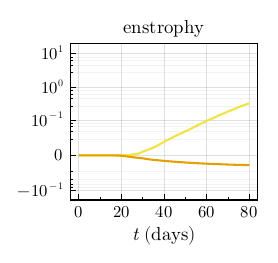}
        \caption{6th order FD on \(6^2\) elements with \(65^2\) nodes each.}
    \end{subcaptionblock}
    \caption{The relative change with time of the total energy/entropy and the total enstrophy for the barotropic shear instability.}
    \label{fig:swe-2d-kh-entropies}
\end{figure}

\begin{figure}[htbp]
    \centering
    \begin{subcaptionblock}{\textwidth}
        \centering
        \includegraphics[scale=\figurescaling]{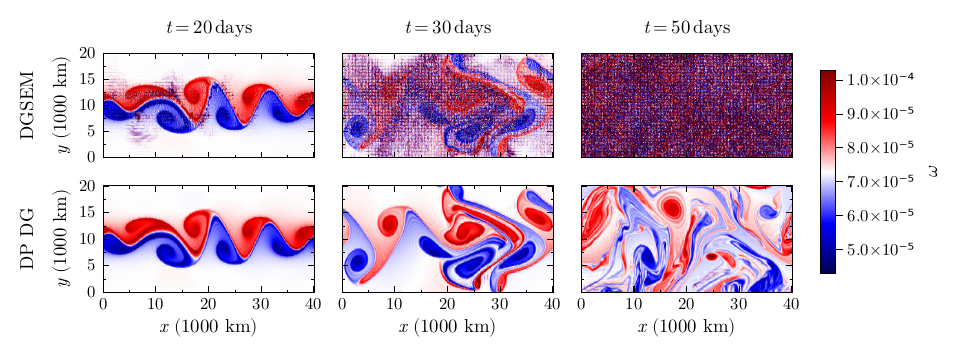}
        \caption{DG using 6th degree polynomials on \(64^2\) elements.}
    \end{subcaptionblock}
    \begin{subcaptionblock}{\textwidth}
        \centering
        \includegraphics[scale=\figurescaling]{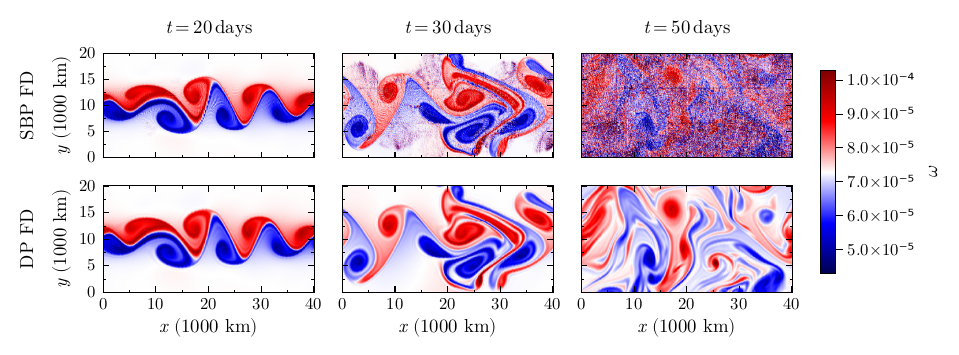}
        \caption{6th order FD on 6 elements with \(65^2\) nodes each.}
    \end{subcaptionblock}
    \caption{Snapshots of the vorticity on the bottom half of the domain for the skew-symmetric entropy conservative DGSEM/SBP FD methods and our novel DP DG/FD methods for the barotropic shear instability.}
    \label{fig:swe-2d-kh-snapshot}
\end{figure}

%

%%%
%%%
\bibliographystyle{unsrt} % We choose the "plain" reference style
\bibliography{references}
%\printbibliography

\end{document}